\documentclass[final]{ukthesis}
\usepackage{url}
\usepackage{setspace}
\usepackage{memhfixc}
\usepackage{amsmath,amsthm,amssymb,latexsym}
\usepackage{graphicx,epsfig} 
\usepackage{fullpage}
\usepackage{longtable}
\usepackage{color} 
\usepackage[left=1.5in,right=1.0625in,top=1in,bottom=1in,footskip=0.5in]{geometry}
\input xy
\xyoption{all}

\newtheorem{theorem}{Theorem}[section]
\newtheorem{proposition}[theorem]{Proposition}
\newtheorem{lemma}[theorem]{Lemma}
\newtheorem{corollary}[theorem]{Corollary}

\newtheorem{definition}[theorem]{Definition}

\theoremstyle{definition}
\newtheorem{example}[theorem]{Example}
\newtheorem{continuation}[theorem]{Continuation of Example}

\newcommand{\hz}{\hat{0}}
\newcommand{\ho}{\hat{1}}

\newcommand{\tensor}{\otimes}

\newcommand{\define}[1]{\emph{#1}}

\newcommand{\G}{\mathcal{G}}

\DeclareMathOperator{\wt}{wt}

\DeclareMathOperator{\Hom}{Hom}
\DeclareMathOperator{\id}{id}
\DeclareMathOperator{\im}{im}

\newcommand{\free}[1]{\langle #1\rangle}

\renewcommand{\int}[1]{{#1}^{\circ}}

\newcommand{\Dv}{\mathbf{D}}
\newcommand{\rv}{\mathbf{R}}
\newcommand{\uv}{\mathbf{U}}
\newcommand{\ur}{\mathbf{UR}}
\newcommand{\sv}{\mathbf{S}}

\newcommand{\av}{\mathbf{a}}
\newcommand{\bv}{\mathbf{b}}

\newcommand{\cv}{\mathbf{c}}
\newcommand{\dv}{\mathbf{d}}
\newcommand{\ab}{\av\bv}

\newcommand{\ba}{\bv\av}
\newcommand{\cd}{\cv\dv}

\newcommand{\ctd}{\cv\mbox{-}2\dv}

\newcommand{\fab}{\free{\av, \bv}}

\newcommand{\fcd}{\free{\cv, \dv}}

\newcommand{\Zab}{\mathbb{Z}\fab}
\newcommand{\Zcd}{\mathbb{Z}\fcd}

\newcommand{\vanish}[1]{}

\newcommand{\fp}{\Box}

\renewcommand{\int}[1]{{#1}^{\circ}}

\newcommand{\journal}[6]{\textsc{#1,} #2, \textit{#3} \textbf{#4} (#5), #6.}
\newcommand{\collection}[6]{\textsc{#1,} #2, \textit{#3} \textbf{#4} (#5), #6.} 
\renewcommand{\book}[4]{\textsc{#1,} ``#2,'' #3, #4.}

\newcommand{\thesis}[4]{\textsc{#1,} ``#2,'' Doctoral dissertation, #3, #4.}
\newcommand{\springer}[4]{\textsc{#1,} ``#2,'' Lecture Notes in Math.,
                          Vol.\ #3, Springer-Verlag, Berlin, #4.}
\newcommand{\preprint}[3]{\textsc{#1,} #2, preprint #3.}
\newcommand{\preparation}[2]{\textsc{#1,} #2, in preparation.}
\newcommand{\appear}[3]{\textsc{#1,} #2, to appear in \textit{#3}}
\newcommand{\submitted}[4]{\textsc{#1,} #2, submitted to \textit{#3}, #4.}
\newcommand{\JCTA}{J.\ Combin.\ Theory Ser.\ A}
\newcommand{\AdvancesinMathematics}{Adv.\ Math.}
\newcommand{\JournalofAlgebraicCombinatorics}{J.\ Algebraic Combin.}

\newcommand{\communication}[1]{\textsc{#1,} personal communication.}

\newcommand{\NC}{\mathcal{N}}

\DeclareMathOperator{\Pyr}{Pyr}
\DeclareMathOperator{\Pri}{Pri}

\newenvironment{proof_special}{\noindent \textbf{Proof:}}{}

\makeatletter
\@addtoreset{equation}{section}
\makeatother

\newcommand{\onethingatopanother}[2]
{\genfrac{}{}{0pt}{}{#1}{#2}}

\newcommand{\Lzero}{\mathcal{L} \cup \{\hz\}}
\newcommand{\Pzero}{\mathcal{P} \cup \{\hz\}}

\newcommand{\tensordots}{\tensor \cdots \tensor}

\newcommand{\pair}[2]{\left\langle#1 \: \vrule \: #2\right\rangle}

\newcommand{\zub}{z_{ub}}
\newcommand{\Zub}{Z_{ub}}
\newcommand{\zt}{z_{t}}
\newcommand{\Zt}{Z_{t}}

\newcommand{\ma}{\av}
\newcommand{\mb}{\bv}
\newcommand{\maa}{\av^{2}}
\newcommand{\mab}{\av\bv}
\newcommand{\mba}{\bv\av}
\newcommand{\mbb}{\bv^{2}}
\newcommand{\maaa}{\av^{3}}
\newcommand{\maab}{\av^{2}\bv}
\newcommand{\maba}{\av\bv\av}
\newcommand{\mabb}{\av\bv^{2}}

\newcommand{\mbab}{\bv\av\bv}

\newcommand{\nabb}{\av\bv\bv}
\newcommand{\nbaa}{\bv\av\av}
\newcommand{\naab}{\av\av\bv}
\newcommand{\nbba}{\bv\bv\av}
\newcommand{\capdots}{\cap \cdots \cap}
\newcommand{\cupdots}{\cup \cdots \cup}

\newcommand{\mccc}{\cv^3}
\newcommand{\mcd}{\cv \dv}
\newcommand{\mdc}{\dv \cv}

\newcommand{\meet}{\wedge}
\newcommand{\join}{\vee}
\newcommand{\fracp}[1]{\{ #1 \}}

\newcommand{\lat}[1]{\mathcal{L}({#1})}
\newcommand{\swapm}[1]{\overline{#1}}

\newcommand{\sweedle}[2]{\sideset{}{^{#1}}\sum{#2}}
\renewcommand{\NG}{\mathbf{G}}
\newcommand{\NA}{\mathbf{A}}
\renewcommand{\NC}{\mathbf{C}}
\newcommand{\HG}{\mathbf{G}^{\bullet}}
\newcommand{\HA}{\mathbf{A}^{\bullet}}
\newcommand{\HC}{\mathbf{C}^{\bullet}}
\newcommand{\freejoin}{\: \mbox{$\vee \hspace{-9.5pt} \bigcirc$} \:}

\DeclareMathOperator{\coker}{coker}
\DeclareMathOperator{\diag}{diag}
\newcommand{\Chain}[2]{C_{#2}({#1})}
\newcommand{\bd}{\partial}
\newcommand{\tree}[1]{\kappa({#1})}

\newcommand{\trans}[1]{{#1}^{\mathrm{T}}}
\newcommand{\MLCM}{N}

\newcommand{\btc}{\noindent\begin{tabular}{p{0.75in}p{5.25in}}}
\newcommand{\etc}{\end{tabular}}

\newcommand{\btcc}{\noindent\begin{tabular}{p{1in}p{5in}}}
\newcommand{\etcc}{\end{tabular}}
\changetocdepth{2}
\setsecnumdepth{section}
\maxsecnumdepth{subsection}
\begin{document}
\author{Michael Slone}
\year=2008
\title{HOMOLOGICAL COMBINATORICS AND EXTENSIONS OF THE CD-INDEX}
\abstract{
Many combinatorial proofs rely on induction.  When these proofs are 
formulated in traditional language, they can be bulky and unmanageable.
Coalgebras provide a language which can reduce
reduce many inductive proofs
in graded poset theory to comprehensible size.  As a bonus, the visual
form of the resulting recursive proofs suggests combinatorial interpretations
for constants appearing in the longer arguments.  We use the techniques
of coalgebras to compute invariants of toric and affine arrangements
as well as of poset products.  In additional chapters we prove structure
theorems for acyclic orientations and critical groups of graphs.
}
\advisor{Richard Ehrenborg}

\frontmatter

\maketitle

\begin{acknowledgments}
\noindent
I have many people to thank.
I would like to single out
Dora~Ahmadi,
Jimmy~Booth,
Tom~Chapman,
Vivian~Cyrus,
Scott~Davison,
Richard~Ehrenborg,
Jennifer~Eli,\newline
Edgar~Enochs, 
Charles~P.~Fairchild,
Claire~A.~Foley,
Brauch~Fugate,
Scott~Godefroy,
Trish~Hall,
Brad~Hamlin,
Mike~Hammond,
David~Johnson,
Eric~Kahn,
Daniel~Kiteck,
Carl~Lee,
David~Little,
Kathryn~Lybarger,
Penny~Pajel~McCollum,
Neil~Moore,
Mark~Motley,
Mary~Motley,
Tricia~Muldoon,
Sunil~Nanwani,
Carlos~M.~Nicol\'as,
Rebecca~Novak,
Wendell~O'Brien,
Sonja~Petrovi\'c,
Pat~Quillen,
Margaret~Readdy,
Josh~Roberts,
Robert~D.~Royar, 
Jack~Schmidt, 
Yuho~Shin,
Aekyoung~Shin~Kim,
R.~Duane~Skaggs, 
Bethany~Slone,
Cephas~Slone,
Donald~J.~Spickler, 
Erik~Stokes,
Brett~Strassner,
Jack~Weir,
and
Yu~Xiang.

The chapter ``Affine and toric arrangements'' is based on joint work
with\newline Richard~Ehrenborg and Margaret~Readdy.  Ehrenborg and Slone
were partially supported by 
National Security Agency grant H98230-06-1-0072.

The chapter ``A geometric approach to acyclic orientations'' is based
on joint work with Richard~Ehrenborg.
\end{acknowledgments}
\clearpage
\tableofcontents*\clearpage
\listoffigures\clearpage
\mainmatter
%
\setcounter{chapter}{-1}
\chapter{Introduction}

For any collection of mathematical objects, two questions have
fundamental importance.
\begin{enumerate}
\item
Can we \emph{enumerate} the objects in the collection?

\item
Can we \emph{classify} the objects in the collection?
\end{enumerate}
This dissertation deals primarily with the question of enumeration 
in the field of algebraic combinatorics.

Here ``enumerate'' is intended in both its common senses: counting objects
and listing objects.
We should be able to count objects so 
we have a rough idea of the complexity of the task of organizing
them.  But we should also be able to give representative examples
of the objects.  In particular, if we can construct representative
examples in a recursive way, then we can teach a computer to perform
operations on the objects.  Moreover, in spending the time to find
appropriate recursively-defined representations of objects, we 
generally discover properties of the objects which will be
useful when we turn to the question of classification.

The chapters of this dissertation can be read independently.  
However, there are strong connections between some of the
chapters.  Here we indicate some of the connections and briefly explain
the topics to be discussed.

Chapters~\ref{chap:affinetoric} and~\ref{chap:mixing} deal with 
the~$\cd$-index, which is a polynomial invariant encoding the 
flag structure of polytopes and similar objects.  With
the $\cd$-index of a polytope
available, one can quickly answer questions such as:
\begin{itemize}
\item
How many vertices does this polytope have? or

\item
How many ways can one select a connected chain of a vertex, an edge, and a face 
in this polytope?
\end{itemize}
The $\cd$-index is not fully understood.  In particular,
even in cases where the coefficients are known to be nonnegative it is
not always known what they count.  

In 
Chapter~\ref{chap:affinetoric} we examine the behavior of the $\cd$-index
(and more generally, the $\ab$-index) on non-spherical manifolds.
This viewpoint allows
combinatorial questions for polytopes, which are spheres, to be
transported to other manifolds.  We start this 
by handling the simplest possible case, that of the $n$-dimensional torus, via 
the notion of
toric hyperplane arrangement.

In 
Chapter~\ref{chap:mixing} we 
streamline computation of and proofs regarding the $\cd$-index.
Recursive formulas are already known for the effects of some natural
geometric operations on the $\cd$-index.  However, some of these 
rely on delicate chain-counting arguments, since their proofs are
expressed in poset-theoretic rather than $\cd$-theoretic terms.  
By importing the
arguments into the $\cd$-language, we are able to simplify many arguments.
We are also able to interpret the coefficients
of the $\cd$-index in a special case as counting lattice paths.
Several results in this chapter were discovered with the assistance
of GAP~\cite{GAP4}.

Chapters~\ref{chap:acyclic} and~\ref{chap:critical} deal, in one way
or another, with chip-firing games on graphs.  Chip-firing games 
arise out of statistical mechanics, where they are called
abelian sandpile models.  There are also connections to Kirchhoff's
fundamental work in circuit theory.

In Chapter~\ref{chap:acyclic} we use chip-firing games as a tool
to give a geometric proof of the result of Propp that acyclic orientations
of a graph with a fixed sink have the structure of a distributive
lattice.

Finally, in Chapter~\ref{chap:critical} we study the critical group,
which is the group of configurations of a chip-firing game.  It is
known that the order of this group is equal to the number of spanning
trees of the graph.  However, the structure of the critical group is
only known for a few classes of graphs.  We can shed a little light
on the structure of the critical group of uniformly cleft graphs,
which are introduced in this dissertation.  We can also count the
spanning trees of non-uniformly cleft trees.

Some work in this dissertation is jointly authored.  In particular,
Chapter~\ref{chap:affinetoric} is joint work with Richard~Ehrenborg and
Margaret~Readdy, while Chapter~\ref{chap:acyclic} is joint work with 
Richard~Ehrenborg.  We have submitted Chapter~\ref{chap:affinetoric}
to the journal \textit{Discrete and Computational Geometry}.  It has
been refereed, and we are preparing a new version for resubmission.
The chapter is based on a snapshot of that new version.  None of the
other chapters have yet been submitted for publication.

\vfill
\begin{center}
Copyright \copyright\ Michael Slone 2008
\end{center}
%
\setcounter{chapter}{0}
\chapter{Affine and toric arrangements}\label{chap:affinetoric}

\section{Introduction}
\label{section_introduction}

Traditionally combinatorialists have studied topological objects
that are spherical, such as polytopes, or which are homeomorphic to
a wedge of spheres, such as those obtained from shellable complexes.
In this chapter we break from this practice and study hyperplane 
arrangements on the $n$-dimensional torus.

It is classical that the convex hull
of a finite collection of points in Euclidean space
is a polytope
and its boundary is a sphere.
The key ingredient in this construction is convexity.
At the moment there is no natural analogue of this process
to obtain a complex whose geometric realization is
a torus.

In this chapter we are taking a zonotopal approach to 
working with arrangements on the torus.
Recall that a zonotope can be defined without the notion
of convexity, that is, it is a Minkowski sum of line segments.
Dually, a central hyperplane arrangement gives rise to
a spherical cell complex.
By considering an arrangement on the torus, we are able to
obtain a subdivision whose geometric realization is indeed the torus.
We will see later in
Section~\ref{section_toric}
that
this amounts to restricting ourselves to arrangements
whose subspaces
in the Euclidean space $\mathbb{R}^{n}$
have coefficient matrices with rational entries.
Under the quotient map
$\mathbb{R}^{n} \longrightarrow \mathbb{R}^{n}/\mathbb{Z}^{n} = T^{n}$
these subspaces 
are sent to subtori of the $n$-dimensional torus $T^{n}$.

Zaslavsky
initiated the modern study of hyperplane arrangements in his
fundamental treatise~\cite{Zaslavsky}.
For early work in the field,
see the references given
in Gr\"unbaum's text~\cite[Chapter 18]{Grunbaum}.
Zaslavsky showed that evaluating the characteristic
polynomial of a central hyperplane arrangement at
$-1$ gives the number of regions
in the complement of the arrangement.
For central hyperplane arrangements,
Bayer and Sturmfels~\cite{Bayer_Sturmfels} proved the flag $f$-vector
of the arrangement can be determined from the intersection
lattice;
see Theorem~\ref{theorem_Bayer_Sturmfels}.
However, their result is stated as a sum
of chains in the intersection lattice and hence
it is hard to apply.
Billera, Ehrenborg, and Readdy
improved the 
Bayer--Sturmfels result by showing that it is enough
to know the flag $f$-vector of the intersection lattice
to compute the flag~$f$-vector of a central arrangement.
Recall that the $\cd$-index of a regular
cell complex is an efficient tool to encode its
flag $f$-vector without linear redundancies~\cite{Bayer_Klapper}.
The Billera--Ehrenborg--Readdy theorem gives an explicit way to compute
the $\cd$-index of the arrangement,
and hence its flag $f$-vector~\cite{Billera_Ehrenborg_Readdy_om}.

We generalize Zaslavsky's theorem on the number
of regions of a hyperplane arrangement to the toric case.
Although there is no intersection lattice per se,
one works with the intersection poset.
 From the Zaslavsky result we obtain a toric
version of the Bayer--Sturmfels result for
hyperplane arrangements,
that is, there is a natural poset map from
the face poset to the intersection poset,
and furthermore, the cardinality
of the inverse image of a chain under this map
is described.

As in the case of a central hyperplane 
arrangement, our toric version of the Bayer--Sturmfels result determines
the flag $f$-vector of the face poset of
a toric arrangement in terms of its intersection poset.
However, this is far from being explicit.
Using the coalgebraic techniques
from~\cite{Ehrenborg_Readdy_c},
we are able to determine the flag $f$-vector
explicitly in terms of the flag $f$-vector
of the intersection poset. 
Moreover, the 
answer is given by a $\cd$ type of polynomial.
The flag $f$-vector of a regular spherical complex
is encoded by the $\cd$-index, a non-commutative polynomial
in the variables~$\cv$ and~$\dv$,
whereas
the $n$-dimensional
toric analogue
is a $\cd$-polynomial plus the $\ab$-polynomial~$(\av-\bv)^{n+1}$.

Zaslavsky also showed 
that evaluating the characteristic
polynomial of an affine arrangement at
$1$ gives the number of bounded regions
in the complement of the arrangement.
Thus we return to affine arrangements in
Euclidean space with the twist that
we study the {\em unbounded} regions.
The unbounded regions form a spherical
complex.
In the case of central arrangements,
this complex is exactly what was studied
previously  
by Billera, Ehrenborg, and Readdy~\cite{Billera_Ehrenborg_Readdy_om}.
For non-central arrangements,
we determine the $\cd$-index of
this complex in terms of the lattice
of unbounded intersections of the arrangement.

Interestingly, the techniques for studying toric arrangements and
the unbounded complex of 
non-central arrangements are similar.
Hence, we present these
results
in the same chapter.
For example, 
the toric and non-central
analogues of the Bayer--Sturmfels theorem
only differ
by which Zaslavsky invariant is used.
The coalgebraic translations of the
two analogues
involve exactly the same argument,
and
the resulting underlying maps
$\varphi_{t}$ (in the toric case) and 
$\varphi_{ub}$ (in the non-central case)
differ only slightly
in their definitions.

We end with many open questions about subdivisions of manifolds.

\section{Preliminaries}
\label{section_preliminaries}

All the posets we will work with are
graded,
that is,
posets having a unique minimal element
$\hz$, a
unique maximal element
$\ho$,
and rank function $\rho$.
For two elements $x$ and $z$ in a graded poset $P$ such that $x \leq z$,
let $[x,z]$ denote the interval
$\{y \in P \: : \: x \leq y \leq z\}$.
Observe that the interval $[x,z]$ is itself a graded poset.
Given a graded poset $P$ of rank $n+1$
and
$S \subseteq \{1, \ldots, n\}$,
the
{\em $S$-rank-selected poset $P(S)$}
is
the poset
consisting of the elements
$P(S) = \{x \in P \: : \: \rho(x) \in S\} \cup \{\hz, \ho\}$.
The partial orders of~$[x,y]$ and $P(S)$ are each inherited from that
of $P$.
The \define{dual poset} of $P$, written $P^*$, is the poset having the
same underlying set as $P$ but with the order relation reversed:
$x <_{P^*} y$ if and only if $y <_P x$.
For standard poset terminology,
we refer the reader
to Stanley's work~\cite{Stanley_EC_I}.

The M\"obius function $\mu(x,y)$ on a poset $P$ is
defined recursively by
$\mu(x,x) = 1$
and for elements $x, y \in P$ with $x < y$ by 
$\mu(x,y) = - \sum_{x \leq z < y} \mu(x,z)$;
see Section~3.7 in~\cite{Stanley_EC_I}.
For a graded poset~$P$ 
with minimal element $\hz$ and maximal element $\ho$
we write $\mu(P) = \mu_{P}(\hz,\ho)$.

We now review important  results about
hyperplane arrangements, the $\cd$-index,
and coalgebraic techniques.
All are essential
for proving the main results of this chapter.

\subsection{Hyperplane arrangements}

Let $\mathcal{H} = \{H_{1}, \ldots, H_{m}\}$
be a hyperplane arrangement in $\mathbb{R}^{n}$,
that is,
a finite collection of affine hyperplanes in 
$n$-dimensional Euclidean space. 
For brevity, throughout this chapter we will 
often refer to a
hyperplane arrangement as an arrangement.
We call an arrangement {\em essential} if
the normal vectors to the hyperplanes in $\mathcal{H}$ span $\mathbb{R}^n$.
In this chapter we are only interested
in essential arrangements.

Observe that the intersection 
$\bigcap_{i=1}^{m} H_{i}$
of all of the hyperplanes in an essential arrangement is either
the empty set $\emptyset$ or a singleton point.
We call an arrangement {\em central} if the
intersection of all the hyperplanes is one point.
We may assume that this point is the origin ${\mathbf 0}$
and hence all of the hyperplanes are 
subspaces of codimension $1$.
If the intersection is the empty set,
we call the arrangement {\em non-central}.

The {\em intersection lattice} $\mathcal{L}$
is the lattice formed by ordering all the intersections
of hyperplanes in~$\mathcal{H}$ by reverse inclusion.
If the intersection
of all the hyperplanes
in a given arrangement is empty,
then we include the empty set
$\emptyset$ as the 
the maximal element in the intersection lattice.
If the arrangement is central, the
maximal element
is  $\{{\mathbf 0}\}$.
In all cases, the minimal element
of $\mathcal{L}$ will be all of
$\mathbb{R}^{n}$.

For a hyperplane arrangement $\mathcal{H}$
with intersection lattice $\mathcal{L}$, the
{\em characteristic polynomial} is defined by
$$   \chi(\mathcal{H}; t)
   =
     \sum_{\onethingatopanother{x \in \mathcal{L}}{x \neq \emptyset}}
             \mu(\hz,x) \cdot t^{\dim(x)}  ,   $$
where $\mu$ denotes the
M\"obius function.
The characteristic polynomial is a combinatorial
invariant of the arrangement. 
The fundamental result of Zaslavsky~\cite{Zaslavsky}
is that this
invariant determines the number and type of
regions in the complement of the arrangement.
\begin{theorem}[Zaslavsky]
For a hyperplane arrangement $\mathcal{H}$ in $\mathbb{R}^{n}$
the number of
regions in the complement of the arrangement 
is given by $(-1)^{n} \cdot \chi(\mathcal{H}; -1)$.
Furthermore,
the number of
bounded regions is given by $(-1)^{n} \cdot \chi(\mathcal{H}; 1)$.
\label{theorem_Zaslavsky}
\end{theorem}
For a graded poset $P$,
define the two Zaslavsky invariants $Z$ and $Z_{b}$
by
\begin{eqnarray*}
      Z(P) 
  & = &
      \sum_{\hz \leq x \leq \ho} (-1)^{\rho(x)} \cdot \mu(\hz,x)  ,\\
      Z_{b}(P) 
  & = &
      (-1)^{\rho(P)} \cdot \mu(P)    .
\end{eqnarray*}
In order to work with Zaslavsky's result, we need the following
reformulation of Theorem~\ref{theorem_Zaslavsky}.
\begin{theorem}
\begin{enumerate}
\item[(i)]
For a central hyperplane arrangement the number of
regions is given by $Z(\mathcal{L})$, where $\mathcal{L}$
is the intersection lattice of the arrangement.

\item[(ii)]
For a non-central hyperplane arrangement the number of
regions is given by $Z(\mathcal{L}) - Z_{b}(\mathcal{L})$,
where $\mathcal{L}$
is the intersection lattice of the arrangement.
The number of bounded regions is given by~$Z_{b}(\mathcal{L})$.
\end{enumerate}
\label{theorem_Zaslavsky_poset}
\end{theorem}

Given a central hyperplane arrangement $\mathcal{H}$
there are two associated lattices, namely,
the 
intersection lattice $\mathcal{L}$
and the 
lattice $T$ of faces of the arrangement.
The minimal element of $T$ is the empty set $\emptyset$
and the maximal element is the whole space~$\mathbb{R}^{n}$.
The lattice of faces can be seen as the face poset
of the cell complex obtained by intersecting the arrangement
$\mathcal{H}$ with a small sphere 
centered at the origin.
Each hyperplane corresponds to a great circle on the sphere.
An alternative way to view the lattice of faces $T$
is that the dual lattice $T^{*}$ is the 
face lattice of the zonotope corresponding to $\mathcal{H}$.

Let $\Lzero$ denote the intersection lattice
with a new minimal element $\hz$ adjoined.
Define an order- and rank-preserving map $z$
from the dual lattice $T^{*}$ to 
the augmented lattice $\Lzero$
by sending a face of the arrangement, that is, a cone in
$\mathbb{R}^{n}$, to its affine hull.
Note that under the map $z$ the minimal
element of $T^{*}$ is mapped to the 
minimal element of $\Lzero$.
Observe that $z$ maps chains to chains.
Hence we view~$z$ as a map from the set
of chains of $T^{*}$
to the set of chains of $\Lzero$.
Bayer and Sturmfels~\cite{Bayer_Sturmfels}
proved the following result about the
inverse image of a chain under the map $z$.
\begin{theorem}[Bayer--Sturmfels]
Let $\mathcal{H}$ be a central hyperplane arrangement
with intersection lattice $\mathcal{L}$.
Let $c = \{\hz = x_{0} < x_{1} < \cdots < x_{k} = \ho\}$
be a chain in $\Lzero$. Then the cardinality
of the inverse image of the chain $c$ 
under the map $z : T^{*} \longrightarrow \Lzero$
is given by the product
$$      |z^{-1}(c)|
    =
        \prod_{i=2}^{k}
               Z([x_{i-1},x_{i}])   .   $$
\label{theorem_Bayer_Sturmfels}
\end{theorem}

\subsection{The cd-index}

Let $P$ be a graded poset of rank $n+1$
with rank function $\rho$.
For $S = \{s_{1} < \cdots < s_{k-1}\}$
a subset of
$\{1, \ldots, n\}$ define $f_{S}$ to be
the number of chains
$c = \{\hz = x_{0} < x_{1} < \cdots < x_{k} = \ho\}$
that have elements with ranks in the set $S$,
that is, 
$$    f_{S}
   =
      |\{ c \:\: : \:\:
          \rho(x_{1}) = s_{1}, \ldots, 
          \rho(x_{k-1}) = s_{k-1} \}|   .  $$
Observe that $f_{S}$ is the number
of maximal chains in the rank-selected poset $P(S)$.
The flag $h$-vector is obtained by the relation
(here we also present its inverse)
\[       h_{S} 
      =
         \sum_{T \subseteq S}
             (-1)^{|S-T|} \cdot f_{T} 
 \:\:\:\: \mbox{ and } \:\:\:\:
         f_{S} 
      =
         \sum_{T \subseteq S}
             h_{T}               .  \]
Recall that by Philip~Hall's theorem,
the M\"obius function of 
$P(S)$ is 
$\mu(P(S)) = (-1)^{|S|-1} \cdot h_{S}$.

Let $\av$ and $\bv$ be two non-commutative variables
of degree $1$.
For $S$ a subset of $\{1, \ldots, n\}$ let $u_{S}$ be the monomial
$u_{S} = u_{1} \cdots u_{n}$ where
$u_{i} = \bv$ if $i \in S$
and~$u_{i} = \av$ if $i \not\in S$.
Then the {\em $\ab$-index} is the noncommutative polynomial defined by
\[  \Psi(P) = \sum_{S} h_{S} \cdot u_{S} ,  \]
where the sum is over all subsets $S \subseteq \{1, \ldots, n\}$.
The $\ab$-index of a poset $P$ of rank $n+1$
is a homogeneous polynomial of degree $n$.

A poset $P$ is {\em Eulerian} if
every interval $[x,y]$, where $x < y$,
satisfies the Euler-Poincar\'e
relation, that is, there are  the same number of elements
of odd as even rank.
Equivalently, the M\"obius function of $P$ is given by
$\mu(x,y) = (-1)^{\rho(x,y)}$
for all~$x \leq y$ in $P$.
The quintessential result is that
the $\ab$-index 
of an Eulerian poset has the following
form.
\begin{theorem}
The $\ab$-index of an Eulerian poset $P$
can be expressed in terms of
the noncommutative variables
$\cv = \av + \bv$
and $\dv = \av \bv + \bv \av$.
\end{theorem}
This theorem was originally conjectured by Fine and
proved by Bayer and Klapper~\cite{Bayer_Klapper}.  Stanley provided
an alternative proof for Eulerian posets~\cite{Stanley_d}.
There are proofs which have both used and revealed  the underlying
algebraic structure.
See for instance~\cite{Ehrenborg_k-Eulerian,Ehrenborg_Readdy_homology}.
When the $\ab$-index $\Psi(P)$ is written in terms of
$\cv$ and $\dv$,
the resulting polynomial is called the {\em $\cd$-index}.
There are linear relations 
among the entries of the flag $f$-vector
of an Eulerian poset, known as the generalized
Dehn-Sommerville relations; see~\cite{Bayer_Billera}.
The importance of the $\cd$-index is that it removes all of these
linear redundancies among the flag $f$-vector entries.

Observe that the variables $\cv$ and $\dv$ have degrees
$1$ and $2$, respectively.
Thus the $\cd$-index of a poset of rank $n+1$ is
a homogeneous polynomial of degree $n$ in the
noncommutative variables $\cv$ and $\dv$.
Define the reverse of an $\ab$-monomial
$u = u_{1} u_{2} \cdots u_{n}$ to be
$u^{*} = u_{n} \cdots u_{2} u_{1}$
and extend by linearity to an involution on $\Zab$.
Since $\cv^{*} = \cv$ and $\dv^{*} = \dv$,
this involution applied to a $\cd$-monomial
simply reverses the $\cd$-monomial.
Finally, the $\ab$-index respects this involution.
For any graded poset $P$
we have  $\Psi(P)^{*} = \Psi(P^{*})$.

A direct approach to describe the $\ab$-index of a
poset $P$ is to give each chain a weight and
then sum over all chains.
For a chain
$c = \{\hz = x_{0} < x_{1} < \cdots < x_{k} = \ho\}$
in the poset $P$, define its {\em weight} to be
\begin{equation}
      \wt(c)
   =
      (\av-\bv)^{\rho(x_{0},x_{1})-1}
        \cdot
      \bv
        \cdot
      (\av-\bv)^{\rho(x_{1},x_{2})-1}
        \cdot
      \bv
        \cdots
      \bv
        \cdot
      (\av-\bv)^{\rho(x_{k-1},x_{k})-1}  ,
\label{equation_weight}
\end{equation}
where $\rho(x,y)$ denotes the rank difference
$\rho(y) - \rho(x)$.
Then the $\ab$-index of $P$ is the polynomial
$$
     \Psi(P) = \sum_{c} \wt(c),
$$
where the sum is over all chains $c$ in the poset $P$.

Finally, a third description of the $\ab$-index is
Stanley's recursion for the $\ab$-index of a graded
poset~\cite[Equation~(7)]{Stanley_d}.  It is:
\begin{equation}
\label{equation_Stanley_recursion}
 \Psi(P) = (\av-\bv)^{\rho(P)-1}
         + \sum_{\hz < x < \ho}
              (\av-\bv)^{\rho(x)-1} \cdot \bv \cdot \Psi([x,\ho]).
\end{equation}
The initial condition for
 this recursion is the unique poset of rank $1$,
$B_{1}$, where $\Psi(B_{1}) = 1$.

\subsection{Coalgebraic techniques}

A coproduct $\Delta$ on a free $\mathbb{Z}$-module $C$
is a linear map $\Delta : C \longrightarrow C \tensor C$.
In order to be explicit, we use the Heyneman--Sweedler 
sigma notation~\cite{Heyneman_Sweedler}
for writing the coproduct. To explain this notation, notice that
$\Delta(w)$ is an element of $C \tensor C$ and thus has the form
$$   \Delta(w) = \sum_{i=1}^{k} w_{1}^{i} \tensor w_{2}^{i}   ,  $$
where $k$ is the number of terms and $w_{1}^{i}$ and $w_{2}^{i}$
belong to $C$. Since all the maps that are applied to $\Delta(w)$
treat each term the same, the sigma notation drops the index $i$
and instead one writes
$$   \Delta(w) = \sum_{w} w_{(1)} \tensor w_{(2)}   .  $$
Informally, this sum should be thought of as all the ways of breaking
the element $w$ in two pieces, where the
first piece is denoted by  $w_{(1)}$
and the second by $w_{(2)}$.
The Sweedler notation for the expression
$(\Delta \tensor \id) \circ \Delta$,
where $\id$ denotes the identity map,
is the following
$$    ((\Delta \tensor \id) \circ \Delta)(w)
   =
      \sum_{w} \sum_{w_{(1)}}
         w_{(1,1)} \tensor w_{(1,2)} \tensor w_{(2)}  .  $$
The right-hand side should be thought of as
first breaking $w$ into the two pieces $w_{(1)}$ and $w_{(2)}$
and then breaking
$w_{(1)}$ into the two pieces $w_{(1,1)}$ and $w_{(1,2)}$.
See Joni and Rota for a more detailed explanation~\cite{Joni_Rota}.

The coproduct $\Delta$ is coassociative
if
$(\Delta \tensor \id) \circ \Delta = 
 (\id \tensor \Delta) \circ \Delta$.
The sigma notation expresses coassociativity as
$$   \sum_{w} \sum_{w_{(1)}}
         w_{(1,1)} \tensor w_{(1,2)} \tensor w_{(2)}
   =
     \sum_{w} \sum_{w_{(2)}}
         w_{(1)} \tensor w_{(2,1)} \tensor w_{(2,2)}  .  $$
Informally coassociativity states that 
all the possible ways to break $w$ into two pieces and
then breaking the first piece
into the two pieces
is equivalent to
all the ways to break $w$ into two pieces and
then break the second piece
into two pieces.
Compare coassociativity with associativity of
a multiplication map $m: A \tensor A \longrightarrow A$
on an algebra $A$.

Assuming coassociativity, the sigma notation
simplifies to
$$   \Delta^{2}(w) = \sum_{w} w_{(1)} \tensor w_{(2)} \tensor w_{(3)} , $$
where $\Delta^{2}$ is defined as
$(\Delta \tensor \id) \circ \Delta = (\id \tensor \Delta) \circ \Delta$,
and the three pieces have been renamed as
$w_{(1)}$, $w_{(2)}$ and $w_{(3)}$.
Coassociativity allows one to define
the $k$-ary coproduct
$\Delta^{k-1} : C \longrightarrow C^{\tensor k}$
by the recursion
$\Delta^{0} = \id$
and
$\Delta^{k} = (\Delta^{k-1} \tensor \id) \circ \Delta$.
The sigma notation for the $k$-ary coproduct is
$$   \Delta^{k-1}(w)
   = 
     \sum_{w}
         w_{(1)} \tensor w_{(2)} \tensordots w_{(k)}  .  $$

Let $\Zab$ denote the polynomial ring in the non-commutative
variables $\av$ and~$\bv$.
We define a coproduct~$\Delta$ on the algebra $\Zab$
by letting $\Delta$ satisfy the following
identities:
$\Delta(1) = 0$, $\Delta(\av) = \Delta(\bv) = 1 \tensor 1$
and the Leibniz condition 
\begin{equation}
      \Delta(u \cdot v)
    =
      \sum_{u} u_{(1)} \tensor u_{(2)} \cdot v
    +
      \sum_{v} u \cdot v_{(1)} \tensor v_{(2)}    .  
\label{equation_Newtonian}
\end{equation}
For an $\ab$-monomial $u = u_{1} u_{2} \cdots u_{n}$
we have that
$$    \Delta(u)
    =
      \sum_{i=1}^{n}
               u_{1} \cdots u_{i-1}
            \tensor
               u_{i+1} \cdots u_{n}    .    $$
The fundamental result for this coproduct is that
the $\ab$-index is a coalgebra homomorphism~\cite{Ehrenborg_Readdy_c}.
We express this result as the following identity.
\begin{theorem}[Ehrenborg--Readdy]
For a graded poset $P$
with $\ab$-index $w = \Psi(P)$ and for any $k$-multilinear map~$M$ 
on $\Zab$, 
the following coproduct identity holds:
$$   \sum_{c}
           M(\Psi([x_{0},x_{1}]),
             \Psi([x_{1},x_{2}]),
                \ldots,
             \Psi([x_{k-1},x_{k}]))
  =
     \sum_{w}
           M(w_{(1)},
             w_{(2)},
                \ldots,
             w_{(k)})    ,  $$
where the first sum is over all
chains $c = \{\hz = x_{0} < x_{1} < \cdots < x_{k} = \ho\}$
of length $k$ 
and the second sum is over the $k$-ary coproduct
of $w$, that is, over~$\Delta^{k-1}$.
\label{theorem_Ehrenborg_Readdy}
\end{theorem}

\subsection{The cd-index of the face poset of a central arrangement}

We recall the definition of the 
omega map~\cite{Billera_Ehrenborg_Readdy_om}.

\begin{definition}
The linear map $\omega$ from $\Zab$ to $\Zcd$
is formed by 
first replacing every occurrence of $\av\bv$ 
in a given 
$\ab$-monomial 
by $2\dv$
and then replacing the remaining letters by $\cv$.
\end{definition}
For a central hyperplane arrangement~$\mathcal{H}$
the $\cd$-index of the face poset is computed
as follows~\cite{Billera_Ehrenborg_Readdy_om}.
\begin{theorem}[Billera--Ehrenborg--Readdy]
Let $\mathcal{H}$ be
a central hyperplane arrangement
with intersection lattice $\mathcal{L}$
and face lattice $T$. Then the $\cd$-index
of the
face lattice $T$ is given by
$$  \Psi(T) = \omega(\av \cdot \Psi(\mathcal{L}))^{*}  .  $$
\label{theorem_Billera_Ehrenborg_Readdy}
\end{theorem}

We review the basic ideas behind the proof of this theorem.
We will refer back to them when we
prove similar results for toric and affine arrangements
in Sections~\ref{section_toric}
and~\ref{section_affine}.

Define three linear operators $\kappa$, $\beta$ and $\eta$ on
$\Zab$ by
\[
\kappa(v) = \begin{cases}
(\av-\bv)^{m} & \text{if $v=\av^{m}$ for some $m \geq 0$,} \\
0 & \text{otherwise,}
\end{cases}
\]
\[
\beta(v) = \begin{cases}
(\av-\bv)^{m} & \text{if $v=\bv^{m}$ for some $m \geq 0$,} \\
0 & \text{otherwise,}
\end{cases}
\]
and 
\[
\eta(v) = \begin{cases}
2 \cdot (\av-\bv)^{m+k} & \text{if $v = \bv^{m} \av^{k}$ for some $m, k \geq 0$,} \\
0 & \text{otherwise.}
\end{cases}
\]
Observe that $\kappa$ and $\beta$ are both algebra maps.
The following relations hold for a poset $P$.
See~\cite[Section 5]{Billera_Ehrenborg_Readdy_om}.
\begin{eqnarray}
\kappa(\Psi(P)) & = & (\av-\bv)^{\rho(P)-1} ,                  
\label{equation_kappa} \\
\beta(\Psi(P))  & = & Z_{b}(P) \cdot (\av-\bv)^{\rho(P)-1} ,   
\label{equation_beta} \\
\eta(\Psi(P))   & = & Z(P) \cdot (\av-\bv)^{\rho(P)-1} .
\label{equation_eta}
\end{eqnarray}
For $k \geq 1$ the operator $\varphi_{k}$
is defined by the coalgebra expression
$$   \varphi_{k}(v)
   =
     \sum_{v}
         \kappa(v_{(1)}) \cdot \bv \cdot
         \eta(v_{(2)}) \cdot \bv \cdots \bv \cdot
         \eta(v_{(k)})   ,  $$
where the coproduct splits $v$ into $k$ parts.
Finally $\varphi$ is defined as the sum
$$       \varphi(v) = \sum_{k \geq 1} \varphi_{k}(v)  .   $$
Note that 
in this expression
only a finite number of terms are nontrivial.
The connection with hyperplane arrangements is given by
the following proposition.
\begin{proposition}
The $\ab$-index of the lattice of faces of a central hyperplane
arrangement is given by
$$   \Psi(T) = \varphi( \Psi( \Lzero ) )^{*}   .   $$
\label{proposition_varphi}
\end{proposition}
The function $\varphi$ satisfies the functional equation
$$     \varphi(v)
     = 
       \kappa(v)
     + 
       \sum_{v} \varphi(v_{(1)}) \cdot \bv \cdot \eta(v_{(2)}) .  $$
 From this functional equation it follows that
the function $\varphi$ satisfies the initial conditions
$\varphi(1) = 1$ and $\varphi(\bv) = 2 \cdot \bv$
and the recurrence relations:
\begin{eqnarray}
\varphi(v \cdot \av) & = & \varphi(v) \cdot \cv , 
\label{equation_varphi_a} \\
\varphi(v \cdot \bv \bv) & = & \varphi(v \cdot \bv) \cdot \cv , 
\label{equation_varphi_b_b} \\
\varphi(v \cdot \av \bv) & = & \varphi(v) \cdot 2\dv ,
\label{equation_varphi_a_b}
\end{eqnarray}
for an $\ab$-monomial $v$;
see~\cite[Section~5]{Billera_Ehrenborg_Readdy_om}.
These recursions culminate in the following result.
\begin{proposition}
\label{proposition_culminate}
The maps $\varphi$ and $\omega$ agree on $\ab$-monomials
that begin with $\av$, that is, if $w = \av\cdot v$, then
$\varphi(w) = \omega(w)$.
\end{proposition}

Theorem~\ref{theorem_Billera_Ehrenborg_Readdy}
follows from the fact that
$\Psi(\Lzero) = \av \cdot \Psi(\mathcal{L})$
by applying
Proposition~\ref{proposition_culminate}.

\subsection{Regular subdivisions of manifolds}

The face poset $P(\Omega)$ of a cell complex $\Omega$ is
the set of all cells in $\Omega$ together
with a minimal element~$\hz$ and a maximal element $\ho$.
One partially orders two
cells $\tau$ and $\sigma$
by requiring that $\tau < \sigma$ if the cell $\tau$ is contained 
in $\overline{\sigma}$,
the closure of $\sigma$.
In order to define a regular cell complex,
consider the cell complex $\Omega$ embedded in 
Euclidean space $\mathbb{R}^{n}$.
This condition is compatible with toric cell complexes
since the~$n$-dimensional torus can be embedded 
in $2n$-dimensional Euclidean space.
Let~$B^{n}$ denote the ball
$\{x \in \mathbb{R}^{n} \: : \: x_{1}^{2} + \cdots + x_{n}^{2} \leq 1\}$
and
let $S^{n-1}$ denote the sphere
$\{x \in \mathbb{R}^{n} \: : \: x_{1}^{2} + \cdots + x_{n}^{2} = 1\}$.
A cell complex $\Omega$ 
is {\em regular} if 
(i) $\Omega$ consists of a finite number of cells,
(ii) for every cell $\sigma$ of $\Omega$
the pair $(\overline{\sigma}, \overline{\sigma} - \sigma)$
is homeomorphic
to a pair $(B^{k},S^{k-1})$ for some integer $k$,
and
(iii)
the boundary
$\overline{\sigma} - \sigma$
is the disjoint union of smaller cells in $\Omega$.
See Section~3.8 in~\cite{Stanley_EC_I} for more details.
For a discussion of regular cell complexes not embedded
in $\mathbb{R}^n$,
see~\cite{Bjorner_topological_methods}.

The face poset of a regular subdivision of the sphere
is an Eulerian face poset
and hence has a $\cd$-index.
For regular subdivisions of compact manifolds,
a similar result holds.
This was independently observed by
Swartz~\cite{Swartz}.
\begin{theorem}
Let $\Omega$ be a regular cell complex whose geometric realization
is a compact $n$-dimensional manifold $\mathcal{M}$. 
Let $\chi(\mathcal{M})$ denote the Euler characteristic of~$\mathcal{M}$.
Then the $\ab$-index
of the face poset $P$ of $\Omega$ has the following form.
\begin{enumerate}
\item[(i)]
If $n$ is odd then
$P$ is an Eulerian poset and hence
$\Psi(P)$ can written in terms of $\cv$ and $\dv$.

\item[(ii)]
If $n$ is even then $\Psi(P)$ has the form
$$      \Psi(P) 
    = 
        \left( 1-\frac{\chi(\mathcal{M})}{2}
        \right) \cdot (\av-\bv)^{n+1}
      + 
        \frac{\chi(\mathcal{M})}{2} \cdot \cv^{n+1}
      + 
        \Phi   ,   $$
where $\Phi$ is a homogeneous $\cd$-polynomial of degree $n+1$
and $\Phi$ does not contain the term $\cv^{n+1}$.
\end{enumerate}
\label{theorem_manifold}
\end{theorem}
\begin{proof}
Observe that the poset $P$ has rank $n+2$.
By~\cite[Theorem~3.8.9]{Stanley_EC_I} we know that
every interval~$[x,y]$ strictly contained in $P$
is Eulerian. When the rank of $P$ is odd this implies
that $P$ is also Eulerian; see~\cite[Exercise 69c]{Stanley_EC_I}.
Hence in this case the $\ab$-index of $P$
can be expressed as a $\cd$-index.
When $n$ is even, we use~\cite[Theorem~4.2]{Ehrenborg_k-Eulerian}
to conclude that
the $\ab$-index of $P$ belongs to
$\mathbb{R}\langle\cv,\dv,(\av-\bv)^{n+1}\rangle$.
Since
$\Psi(P)$ has degree $n+1$,
the $\ab$-index $\Psi(P)$ can be written in the form
\[      \Psi(P) 
    = 
        c_{1} \cdot (\av-\bv)^{n+1}
      + 
        c_{2} \cdot \cv^{n+1}
      + 
        \Phi   ,   \]
where $\Phi$ is a homogeneous $\cd$-polynomial of degree $n+1$
that does not contain any~$\cv^{n+1}$ terms. 
By looking at the coefficients
of $\av^{n+1}$ and $\bv^{n+1}$, we have
$c_{1} + c_{2} = 1$ and
$c_{2} - c_{1} = \mu(P) = \chi(\mathcal{M}) - 1$,
where the last identity is again~\cite[Theorem~3.8.9]{Stanley_EC_I}.
Solving for $c_{1}$ and $c_{2}$ proves the result.
\end{proof}

For the $n$-dimensional torus Theorem~\ref{theorem_manifold}
can be expressed as follows.
\begin{corollary}
Let $\Omega$ be a regular cell complex whose geometric realization
is the $n$-dimensional torus~$T^{n}$. Then the $\ab$-index
of the face poset $P$ of $\Omega$ has the following form:
$$      \Psi(P) 
    = 
        (\av-\bv)^{n+1}
      + 
        \Phi   ,   $$
where $\Phi$ is a homogeneous $\cd$-polynomial of degree $n+1$
and $\Phi$ does not contain the term $\cv^{n+1}$.
\end{corollary}
\begin{proof}
When $n$ is even this is Theorem~\ref{theorem_manifold}.
When $n$ is odd this is Theorem~\ref{theorem_manifold}
together with the two facts that
$\chi(T^{n}) = 0$ and
$(\av-\bv)^{n+1} = (\cv^{2} - 2 \dv)^{(n+1)/2}$.
\end{proof}

\section{Toric arrangements}
\label{section_toric}

\subsection{Toric subspaces and arrangements}

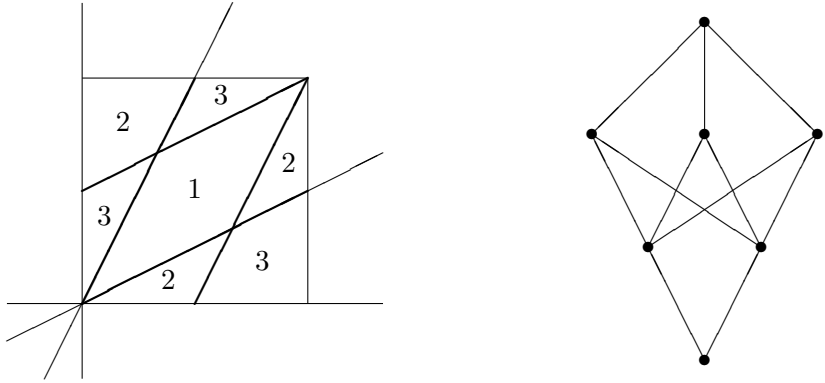
\begin{figure}
\setlength{\unitlength}{0.5mm}
\begin{center}
\begin{picture}(100,100)(-20,-20)

\put(0,-20){\line(0,1){100}}
\put(-20,0){\line(1,0){100}}


\put(0,60){\line(1,0){60}}
\put(60,0){\line(0,1){60}}

\put(-20,-10){\line(2,1){100}}
\put(-10,-20){\line(1,2){50}}

\thicklines

\put(0,0){\line(2,1){60}}
\put(0,0){\line(1,2){30}}
\put(0,30){\line(2,1){60}}
\put(30,0){\line(1,2){30}}

\put(28,28){{\small 1}}

\put(53,35){{\small 2}}
\put(9,46){{\small 2}}
\put(21,4){{\small 2}}

\put(4,21){{\small 3}}
\put(46,9){{\small 3}}
\put(35,53){{\small 3}}

\end{picture}
\hspace*{15 mm}
\begin{picture}(100,100)(-5,-5)

\put(45,0){\circle*{3}}
\multiput(30,30)(30,0){2}{\put(0,0){\circle*{3}}}
\multiput(15,60)(30,0){3}{\put(0,0){\circle*{3}}}
\put(45,90){\circle*{3}}

\put(45,0){\line(-1,2){15}}
\put(45,0){\line(1,2){15}}

\put(30,30){\line(-1,2){15}}
\put(30,30){\line(1,2){15}}
\put(30,30){\line(3,2){45}}

\put(60,30){\line(-3,2){45}}
\put(60,30){\line(-1,2){15}}
\put(60,30){\line(1,2){15}}

\put(45,90){\line(-1,-1){30}}
\put(45,90){\line(0,-1){30}}
\put(45,90){\line(1,-1){30}}
\end{picture}
\end{center}
\caption{A toric line arrangement
which subdivides
the torus $T^{2}$ into a non-regular cell complex
and its intersection poset.}
\label{figure_toric_one}
\end{figure}

The $n$-dimensional torus $T^{n}$ is defined as the quotient
$\mathbb{R}^{n}/\mathbb{Z}^{n}$.
Recall that the torus~$T^{n}$ is an abelian group.
When identifying the torus $T^{n}$ with the set $[0,1)^{n}$,
the group structure is componentwise addition modulo $1$.
\begin{lemma}
Let $V$ be a $k$-dimensional affine subspace
in $\mathbb{R}^{n}$ with rational coefficients. That is, $V$ has the form
$$     V 
   =
       \{ \vec{v} \in \mathbb{R}^{n} \:\: : \:\: A \vec{v} = \vec{b}\} , $$
where the matrix $A$ has rational entries and the vector $\vec{b}$
is allowed to have real entries.
Then the image of $V$ under the
quotient map $\mathbb{R}^{n} \to \mathbb{R}^{n}/\mathbb{Z}^{n}$,
denoted by $\overline{V}$,
is a $k$-dimensional torus.
\end{lemma}
\begin{proof}
By translating $V$, we may assume that
the vector $\vec{b}$ is the zero vector,
and therefore
$V$ is a subspace.
In this case, the intersection of $V$ with the integer lattice~$\mathbb{Z}^{n}$
is a subgroup of the free abelian group $\mathbb{Z}^{n}$.
Since the matrix $A$ has all rational entries,
the rank of this subgroup is $k$,
that is, the subgroup is isomorphic to $\mathbb{Z}^{k}$.
Hence the image $\overline{V}$
is the quotient $V/(V \cap \mathbb{Z}^{n})$,
which is isomorphic to
the quotient~$\mathbb{R}^{k}/\mathbb{Z}^{k}$, that is,
a $k$-dimensional torus.
\end{proof}

We call the image $\overline{V}$ 
a {\em toric subspace} of the torus $T^{n}$
because it is homeomorphic to some $k$-dimensional torus.
When we remove the condition that the matrix $A$ is rational, 
the image is not necessarily homeomorphic to a torus.

The intersection of two toric subspaces is in general not a toric subspace,
but instead is the disjoint union of a finite number of
toric subspaces.
For two affine subspaces $V$ and $W$ with rational coefficients,
we have that $\overline{V \cap W} \subseteq \overline{V} \cap \overline{W}$.
In general, this containment is strict.

Define the translate of a toric subspace $U$ by a point $x$
on the torus to be the toric subspace $U + x = \{ u+x  : u \in U \}$.
Alternatively, one may lift the toric subspace to an affine subspace
in Euclidean space, translate it and then map back to the torus.
Then for two toric subspaces $V$ and $W$, their intersection
has the form
$$     V \cap W = \bigcup_{p=1}^{r} (U + x_{p})  , $$
where $U$ is a toric subspace, $r$ is a non-negative integer
and $x_{1}, \ldots, x_{r}$ are points on the torus $T^{n}$.

\begin{figure}
\setlength{\unitlength}{0.5mm}
\begin{center}
\begin{picture}(100,100)(-20,-20)

\put(0,-20){\line(0,1){100}}
\put(-20,0){\line(1,0){100}}

\put(0,60){\line(1,0){60}}
\put(60,0){\line(0,1){60}}

\put(-20,-10){\line(2,1){100}}
\put(-6,-18){\line(1,3){32}}
\put(-20,12){\line(1,0){100}}

\thicklines

\put(0,0){\line(2,1){60}}
\put(0,30){\line(2,1){60}}
\put(0,0){\line(1,3){20}}
\put(20,0){\line(1,3){20}}
\put(40,0){\line(1,3){20}}
\put(0,12){\line(1,0){60}}
\end{picture}
\hspace*{15 mm}
\begin{picture}(100,100)(-5,-5)

\put(45,0){\circle*{3}}
\multiput(15,30)(30,0){3}{\put(0,0){\circle*{3}}}
\multiput(0,60)(15,0){7}{\put(0,0){\circle*{3}}}
\put(45,90){\circle*{3}}

\put(45,0){\line(-1,1){30}}
\put(45,0){\line(0,1){30}}
\put(45,0){\line(1,1){30}}

\put(15,30){\line(-1,2){15}}
\put(15,30){\line(0,1){30}}
\put(15,30){\line(1,2){15}}
\put(15,30){\line(1,1){30}}
\put(15,30){\line(3,2){45}}

\put(45,30){\line(-3,2){45}}
\put(45,30){\line(-1,1){30}}
\put(45,30){\line(-1,2){15}}
\put(45,30){\line(0,1){30}}
\put(45,30){\line(1,2){15}}
\put(45,30){\line(1,1){30}}
\put(45,30){\line(3,2){45}}

\put(75,30){\line(-1,2){15}}
\put(75,30){\line(0,1){30}}
\put(75,30){\line(1,2){15}}

\put(45,90){\line(-3,-2){45}}
\put(45,90){\line(-1,-1){30}}
\put(45,90){\line(-1,-2){15}}
\put(45,90){\line(0,-1){30}}
\put(45,90){\line(1,-2){15}}
\put(45,90){\line(1,-1){30}}
\put(45,90){\line(3,-2){45}}
\end{picture}
\end{center}
\caption{A toric line arrangement
and its intersection poset.}
\label{figure_toric_two}
\end{figure}
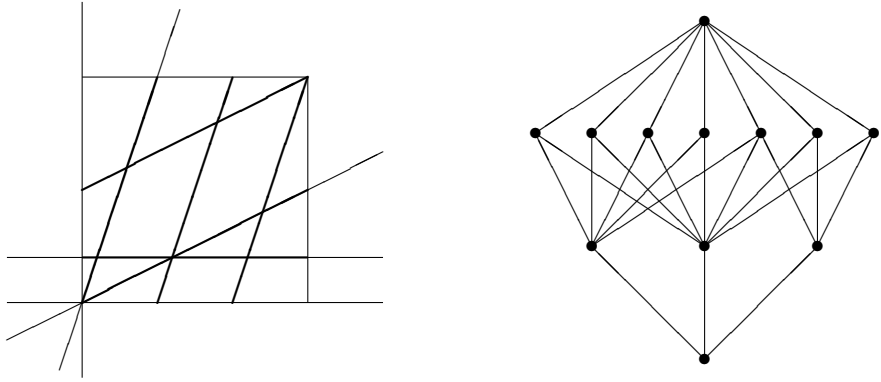

A {\em toric hyperplane arrangement} $\mathcal{H} = \{H_{1}, \ldots, H_{m}\}$
is a finite collection of toric hyperplanes.
Define the {\em intersection poset} $\mathcal{P}$
of a toric arrangement
to be the set of all connected components
arising from
all possible intersections of the toric hyperplanes,
that is, 
all connected components of
$\bigcap_{i \in S} H_{i}$
where $S \subseteq \{1, \ldots, m\}$,
together with the empty set.
Order the elements of
the intersection poset $\mathcal{P}$ by reverse inclusion,
that is, the torus $T^{n}$ is the minimal element of $\mathcal{P}$
corresponding to the empty intersection,
and the empty set is the maximal element.
A toric subspace $V$ is contained in the intersection
poset $\mathcal{P}$ if there are
toric hyperplanes $H_{i_{1}}, \ldots, H_{i_{k}}$
in the arrangement such that
$V \subseteq H_{i_{1}} \cap \cdots \cap H_{i_{k}}$
and there is no toric subspace $W$
satisfying
$V \subset W \subseteq H_{i_{1}} \cap \cdots \cap H_{i_{k}}$.
In other words,
$V$ has to be a maximal toric subspace in some intersection
of toric hyperplanes from the arrangement.

The notion of using the intersection poset can be found
in
work of Zaslavsky, where he considers
topological dissections~\cite{Zaslavsky_paper}.
In this setting there is not an intersection
lattice, but rather an intersection poset.

To every toric hyperplane arrangement $\mathcal{H} = \{H_{1}, \ldots, H_{m}\}$
there is an associated periodic hyperplane arrangement $\widetilde{\mathcal{H}}$
in the Euclidean space $\mathbb{R}^{n}$. Namely, the inverse
image of the toric hyperplane~$H_{i}$ under the quotient map
$\mathbb{R}^{n} \to \mathbb{R}^{n}/\mathbb{Z}^{n}$ is the union of parallel integer
translates of a real hyperplane.
Let~$\widetilde{\mathcal{H}}$
be the collection of all these integer translates.
Observe that every face of the toric arrangement~$\mathcal{H}$
can be lifted to a parallel class of faces in
the periodic real arrangement $\widetilde{\mathcal{H}}$.

As in the case of real arrangements, a toric arrangment subdivides
the torus into a number of regions.  Let $T_t$ denote the poset of
regions in the induced subdivision of the torus.

For a toric hyperplane arrangement $\mathcal{H}$
define the {\em toric characteristic polynomial}
to be
$$   \chi(\mathcal{H}; t)
   =
     \sum_{\onethingatopanother{x \in \mathcal{P}}{x \neq \emptyset}}
             \mu(\hz,x) \cdot t^{\dim(x)}  .   $$

\begin{example}
Consider the line arrangement consisting
of the two lines $y = 2 \cdot x$
and $x = 2 \cdot y$ in the plane $\mathbb{R}^{2}$.
In $\mathbb{R}^{2}$ they intersect in one point,
namely
the origin,
whereas on the torus $T^{2}$ they intersect
in three points,
namely
$(0,0)$, $(2/3,1/3)$, and~$(1/3,2/3)$.
The characteristic polynomial
is given by $\chi(\mathcal{H};t) = t^{2} - 2 \cdot t + 3$.
However, this arrangement is not
regular, since the induced subdivision of $T^{2}$
is not regular. The boundary of each region
is a wedge of two circles.
See Figure~\ref{figure_toric_one}.
\label{example_toric_one}
\end{example}

\begin{example}
Consider the line arrangement consisting
of the three lines $y = 3 \cdot x$,
$x = 2 \cdot y$, and
$y = 1/5$.
It subdivides the torus into a regular cell complex.
The subdivision and the associated
intersection poset are shown in 
Figure~\ref{figure_toric_two}.
The characteristic polynomial
is given by $\chi(\mathcal{H};t) = t^{2} - 3 \cdot t + 8$.
Furthermore, the $\ab$-index of the
subdivision of the torus is given by
$\Psi(T_{t}) = (\av-\bv)^{3} + 7 \cdot \mdc + 8 \cdot \mcd$,
as the following calculation shows.
\newcommand{\ts}{\:\:\:\:\:\:}
$$
  \begin{array}{c r r c r c c}
 S & f_{S} & h_{S} & u_{S} & (\av-\bv)^{3} 
   & 7 \cdot \mdc & 8 \cdot \mcd \\ \hline
\emptyset &  1 &  1 & \av\av\av &  1 \ts & 0 & 0 \\
\{1\}     &  7 &  6 & \bv\av\av & -1 \ts & 7 & 0 \\
\{2\}     & 15 & 14 & \av\bv\av & -1 \ts & 7 & 8 \\
\{3\}     &  8 &  7 & \av\av\bv & -1 \ts & 0 & 8 \\
\{1,2\}   & 30 &  9 & \bv\bv\av &  1 \ts & 0 & 8 \\
\{1,3\}   & 30 & 16 & \bv\av\bv &  1 \ts & 7 & 8 \\
\{2,3\}   & 30 &  8 & \av\bv\bv &  1 \ts & 7 & 0 \\
\{1,2,3\} & 60 & -1 & \bv\bv\bv & -1 \ts & 0 & 0 \\
  \end{array} $$
Recall that
$\mdc = \maba + \nabb + \nbaa + \mbab$
and
$\mcd = \naab + \maba + \mbab + \nbba$.
Here in the last three columns we indicate
the contribution of a given term to each $\ab$-monomial.
Observe that the sum of the last three columns gives
the flag $h$-vector entries.
\label{example_toric_two}
\end{example}

We now give a natural interpretation of the toric characteristic polynomial.
Recall that the intersection of toric subspaces
is the disjoint union of toric subspaces that are translates of each other.
Let~$G$ be the collection of finite intersections of toric
subspaces of the $n$-dimensional torus $T^{n}$,
that is, $G$ consists of sets of the form
$V = W_{1} \capdots W_{q}$, where $W_{1}, \ldots, W_{q}$ are
toric subspaces.
Such a set~$V$ can be written as a union
$V = \bigcup_{p=1}^{r} (U + x_{p})$,
where $U$ is a toric subspace, $r$ a non-negative integer,
and $x_{1}, \ldots, x_{r}$ are points on the torus.
Observe that the empty set $\emptyset$
and the torus $T^{n}$ belong to $G$.
Furthermore, $G$ is closed under finite intersections.
Let $L$ be the distributive lattice consisting of all
subsets of the torus $T^{n}$ that are obtained from
the collection~$G$ by finite intersections, finite unions
and complements. The set~$G$ is the generating set for the lattice~$L$.
A {\em valuation} $v$ on the lattice $L$ is a function on $L$
to an abelian group
satisfying
$v(\emptyset) = 0$ and $v(A) + v(B) = v(A \cap B) + v(A \cup B)$
for all sets~$A, B \in L$.

The next theorem is analogous to 
Theorem~2.1 in~\cite{Ehrenborg_Readdy_valuation_1}.
The proof here is more involved
due to the fact that the collection of
toric subspaces is not closed under intersections.
\begin{theorem}
There is a valuation $v$ on the distributive lattice~$L$
to integer polynomials in the variable $t$
such that 
for a $k$-dimensional toric subspace $V$
its valuation is $v(V) = t^{k}$.
\end{theorem}
\begin{proof}
Define the function $v$ on the generating set $G$ by
\[     v\left(\bigcup_{p=1}^{r} (U + x_{p})\right)
    =
       r \cdot t^{k}   ,  \]
where we assume that $U$ is a $k$-dimensional toric subspace
and the $r$ translates
$U + x_{1}, \ldots, U + x_{r}$ are pairwise disjoint.
Observe that the function $v$ is additive with respect to
disjoint unions, that is, for elements
$V_{1}, \ldots, V_{m}$ in $G$
which are pairwise disjoint
and 
$V_{1} \cupdots V_{m} \in G$.
In this case, each $V_{i}$ is a disjoint union of translates of
the same affine subspace $U$ and both sides of the identity
$v(V_{1}) + \cdots + v(V_{m})
    =
 v(V_{1} \cupdots V_{m})$
count the number of translates of $U$ 
times $t^{\dim(U)}$.

Groemer's integral theorem~\cite{Groemer}
(see also~\cite[Theorem~2.2.1]{Klain_Rota})
states that a function $v$
defined on a generating set $G$ extends to a valuation
on the distributive lattice generated by $G$ if
for all $V_{1}, \ldots, V_{m}$ in $G$ such that
$V_{1} \cupdots V_{m} \in G$, the inclusion-exclusion
formula holds:
\begin{equation}
v(V_{1} \cupdots V_{m})
  =
\sum_{i} v(V_{i})
  -
\sum_{i<j} v(V_{i} \cap V_{j})
  +
\cdots .
\label{equation_inclusion_exclusion}
\end{equation}
To verify this relation for our generating set $G$,
first consider the case when 
the union $V_{1} \cupdots V_{m}$
is a toric subspace.
This case implies that
$V_{1} \cupdots V_{m} = V_{i}$ for some index $i$.
It then follows that the inclusion-exclusion
formula~(\ref{equation_inclusion_exclusion}) holds
trivially.

Before considering the general case,
we introduce some notation.
For $S$ a non-empty subset of the index set $\{1, \ldots, m\}$,
let $V_{S} = \bigcap_{i \in S} V_{i}$.
Equation~(\ref{equation_inclusion_exclusion})
can then be written as
$$
v(V_{1} \cupdots V_{m})
  =
\sum_{S} (-1)^{|S|-1} \cdot v(V_{S})  ,
$$
where the sum ranges over non-empty subsets $S$ of $\{1, \ldots, m\}$.
Now assume that 
$V_{1} \cupdots V_{m}$ is the disjoint union
$(U + x_{1}) \cupdots (U + x_{r})$.
Let $V_{S,p}$ denote the intersection
$V_{S} \cap (U + x_{p})$.
Observe that
$U + x_{p} = \bigcup_{i=1}^{m} V_{\{i\},p}$
and since $U + x_{p}$ is itself a toric subspace,
we have already proved that
the inclusion-exclusion formula~(\ref{equation_inclusion_exclusion})
holds for this union.
Hence we have
\begin{eqnarray*}
v(V_{1} \cupdots V_{m})
  & = &
\sum_{p=1}^{r} v(U + x_{p}) \\
  & = &
\sum_{p=1}^{r} \sum_{S} (-1)^{|S|-1} \cdot v(V_{S,p}) \\
  & = &
\sum_{S} (-1)^{|S|-1} \cdot \sum_{p=1}^{r} v(V_{S,p}) \\
  & = &
\sum_{S} (-1)^{|S|-1} \cdot v(V_{S}) ,
\end{eqnarray*}
where $S$ ranges over all non-empty subsets of $\{1, \ldots, m\}$.
The last step follows since the terms in the 
union $V_{S} = \bigcup_{p=1}^{r} V_{S,p}$
are pairwise disjoint.
\end{proof}

By M\"obius inversion we directly have the following theorem.
The proof is standard.  See the 
references~\cite{Athanasiadis,Chen,Ehrenborg_Readdy_valuation_1,Jozefiak_Sagan}.
\begin{theorem}
The characteristic polynomial of a toric arrangement is given by
$$  \chi(\mathcal{H}) = v\left( T^{n} - \bigcup_{i=1}^{m} H_{i} \right) .  $$
\label{theorem_characteristic}
\end{theorem}

When each region is an open ball we can now determine the number of
regions in a toric arrangement.  The proof is analogous to the proofs
in~\cite{Ehrenborg_Readdy_valuation_1,Ehrenborg_Readdy_valuation_2}.
Recall that the Euler characteristic can be viewed as a valuation.  Here we use
the notation $\varepsilon$ to indicate that we are viewing the Euler
valuation as a valuation.

\begin{theorem}
Let $\mathcal{H}$ be a toric hyperplane arrangement on the $n$-dimensional
torus $T^{n}$ that subdivides the torus into 
regions that are open $n$-dimensional balls.
Then the complement of the arrangement has
$(-1)^{n} \cdot \chi(\mathcal{H}; 0)$
regions.
\label{theorem_toric_Zaslavsky_version_1}
\end{theorem}
\begin{proof}
Observe that the Euler valuation~$\varepsilon$ of a $k$-dimensional torus
is given by the Kronecker delta~$\delta_{k,0}$.
Hence for a toric subspace $V$ of the $n$-dimensional torus, 
the Euler valuation of $V$ is obtained by setting $t = 0$
in the valuation, that is, 
$\varepsilon(V) = v(V)|_{t = 0}$.
Since the two valuations $\varepsilon$ and $v|_{t=0}$
are additive with respect to disjoint unions,
they agree for any member of the generating set $G$.
Hence they also agree for any member in the distributive lattice $L$.
In particular, 
\begin{equation}
    \varepsilon\left( T^{n} - \bigcup_{i=1}^{m} H_{i} \right) 
  =
    \left. v\left( T^{n} - \bigcup_{i=1}^{m} H_{i} \right)\right|_{t = 0} .
\label{equation_Euler_and_v}
\end{equation}
Since the Euler valuation of an open ball
is $(-1)^{n}$ and
$T^{n} - \bigcup_{i=1}^{m} H_{i}$
is a disjoint union of open balls,
the left-hand side 
of~(\ref{equation_Euler_and_v})
is $(-1)^{n}$ times the number of regions.
The right-hand side is $\chi(\mathcal{H}; t=0)$ by 
Theorem~\ref{theorem_characteristic}.
\end{proof}

\addtocounter{theorem}{-5}
\begin{continuation}
\textrm{
Setting $t=0$ in the characteristic polynomial
in Example~\ref{example_toric_one}
we obtain $3$, which is indeed
the number of regions of this arrangement.
}
\end{continuation}
\addtocounter{theorem}{4}

We call a toric hyperplane arrangement $\mathcal{H} = \{H_{1}, \ldots, H_{m}\}$
{\em rational}
if each hyperplane $H_{i}$
is of the form
$\vec{a}_{i} \cdot \vec{x} = b_{i}$
where the vector $\vec{a}_{i}$ has integer entries and
$b_{i}$ is an integer for $1 \leq i \leq m$.
This is equivalent to assuming every constant $b_{i}$ is rational
since every vector $\vec{a}_{i}$ was already assumed to be rational.
In what follows it will be convenient to
assume every coefficient is integral in a given rational
arrangement.

Define
$\MLCM(\mathcal{H})$ to be
the least common multiple of all the
$n \times n$ minors of the~$n \times m$ matrix
$(\vec{a}_{1}, \ldots, \vec{a}_{m})$.
We can now give a different interpretation
of the toric chromatic polynomial
by counting lattice points.
\begin{theorem}
For a rational hyperplane arrangement $\mathcal{H}$
there exists a constant~$k$ such that
for every $q > k$ where $q$ is a multiple of $\MLCM(\mathcal{H})$,
the toric characteristic polynomial
evaluated at $q$ is given by
the number of lattice points in
$\left( \frac{1}{q} \mathbb{Z} \right)^{n}/\mathbb{Z}^{n}$
that do not lie on any of the toric hyperplanes $H_{i}$, that is,
\[
    \chi(\mathcal{H}; q)
  =
    \left|
      \left( \frac{1}{q} \mathbb{Z} \right)^{n}/\mathbb{Z}^{n}
        -
      \bigcup_{i=1}^{m} H_{i}
    \right|    . 
\]
\label{theorem_lattice_points}
\end{theorem}
The condition that $q$ is a multiple of $\MLCM(\mathcal{H})$
implies that every subspace $x$
in the intersection poset~$\mathcal{P}$
intersects the toric lattice
$\left(\frac{1}{q} \mathbb{Z} \right)^{n}/\mathbb{Z}^{n}$
in exactly $q^{\dim(x)}$ points.
Theorem~\ref{theorem_lattice_points}
now follows by M\"obius inversion.
This theorem
is the toric analogue of
the
finite field method of
Athanasiadis.
See~\cite[Theorem~2.1]{Athanasiadis_II} in particular.

In the case when $\MLCM(\mathcal{H}) = 1$, the toric arrangement 
$\mathcal{H}$ is
called {\em unimodular}. 
Novik, Postnikov, and Sturmfels~\cite{Novik_Postnikov_Sturmfels}
state Theorem~\ref{theorem_toric_Zaslavsky_version_1}
in the special case of unimodular arrangements.
Their first proof is based upon Zaslavsky's result on the number
of bounded regions in an affine arrangement.
The second proof, due to 
Reiner, is equivalent to our
proof for arbitrary toric arrangements.
See also the paper~\cite{Zaslavsky_paper}
by Zaslavsky, where more general arrangements
are considered.

\subsection{Graphical arrangements}

We digress in this subsection to discuss an application
to graphical arrangements, which are hyperplane arrangements 
arising from graphs.
For a graph $G$ on the vertex set $\{1, \ldots, n\}$
define the \define{graphical arrangement} $\mathcal{H}_{G}$ to
be the collection of hyperplanes of the form $x_{i} = x_{j}$
for each edge $ij$ in the graph $G$.
\begin{corollary}
For a connected graph $G$ on $n$ vertices the regions
in the complement
of the graphical arrangement $\mathcal{H}_{G}$ on the torus $T^{n}$
are each homotopy equivalent to the $1$-dimensional torus $T^{1}$.
Furthermore, the number of regions is given by
$(-1)^{n-1}$ times the linear coefficient of the chromatic
polynomial of $G$.
\end{corollary}
\begin{proof}
The chromatic polynomial of the graph $G$ is equal to
the characteristic polynomial of the graphical arrangement
$\mathcal{H}_{G}$.
Furthermore, the intersection lattice
of the real arrangement $\mathcal{H}_{G}$ is the same
as the intersection poset of the toric arrangement~$\mathcal{H}_{G}$.
Translating the graphic arrangement in the
direction $(1, \ldots, 1)$ leaves the
arrangement on the torus invariant.
Since $G$ is connected this is the only direction
that leaves the arrangement invariant.
Hence each region is
homotopy equivalent to $T^{1}$.
By adding the hyperplane $x_{1} = 0$ to the arrangement
we obtain a new arrangement $\mathcal{H}^{\prime}$
with the same number of regions, 
but with each region homeomorphic to a ball.
Since the intersection lattice of $\mathcal{H}^{\prime}$
is just the Cartesian product of the two-element poset
with the  intersection lattice of $\mathcal{H}_{G}$, we have
$$   \chi(\mathcal{H}^{\prime}, t)
   = 
     (t-1) \cdot \chi(\mathcal{H}_{G}, t)/t .  $$
The number of regions is obtained by setting $t=0$ in this equality.
\end{proof}

A similar statement holds for graphs that are disconnected.
The result follows from the fact that
the complement of the graphical arrangement
is the product of the complements
of each connected component.
\begin{corollary}
For a graph $G$ on $n$ vertices 
consisting of $k$ components,
the regions
in the complement
of the graphical arrangement $\mathcal{H}_{G}$ on the torus $T^{n}$
are each homotopy equivalent to the $k$-dimensional torus $T^{k}$.
The number of regions is given by
$(-1)^{n-k}$ times the coefficient of $t^{k}$ in the chromatic
polynomial of $G$.
\end{corollary}

Stanley~\cite{Stanley_acyclic} proved the celebrated result
that the chromatic polynomial of a graph
evaluated at $t=-1$
is $(-1)^{n}$ times the number of acyclic orientations of the
graph.
A similar interpretation for the linear coefficient
of the chromatic polynomial is due to
Greene and Zaslavsky~\cite{Greene_Zaslavsky}:
\begin{theorem}[Greene--Zaslavsky]
Let $G$ be a connected graph and $v$ a given vertex of the graph.
The linear coefficient of the chromatic polynomial
is $(-1)^{n-1}$ times the number of acyclic orientations
of the graph such that the only sink is the vertex $v$.
\end{theorem}
\begin{proof}
It is enough to give a bijection between
regions in the complement of the graphical arrangement
on the torus $T^{n}$
and acyclic orientations with the vertex $v$ as the unique sink.
For a region $R$ of the arrangement intersect it with
the hyperplane~$x_{v} = 0$ to obtain the face $S$.
Let $\mathcal{H}^{\prime}$ be the arrangement~$\mathcal{H}_{G}$
together with the hyperplane $x_{v} = 0$.
Lift $S$ to a face $\widetilde{S}$
in the periodic arrangement~$\widetilde{\mathcal{H}^{\prime}}$
in~$\mathbb{R}^{n}$.
Observe that $\widetilde{S}$ is the interior of a polytope.
When minimizing the linear functional
$L(x) = x_{1} + \cdots + x_{n}$
on the closure of the face $\widetilde{S}$, the optimum
is a lattice point $k = (k_{1}, \ldots, k_{n})$.
Pick a point $x = (x_{1}, \ldots, x_{n})$ 
in $\widetilde{S}$ close to the optimum,
that is, such that each coordinate $x_{i}$ lies in the interval
$[k_{i},k_{i}+\epsilon)$ for some small $\epsilon > 0$.

Let $y = (y_{1}, \ldots, y_{n})$ be the image of the point
$x$ on the torus $T^{n}$, that is,
$y_{i} = x_{i} \bmod 1$. Note that
each entry $y_{i}$ lies in the half open interval $[0,1)$
and that $y_{v} = 0$.
Construct an orientation of the graph $G$
by letting the edge $ij$ be oriented $i \rightarrow j$ if~$y_{i} > y_{j}$.
Note that this orientation is acyclic and has the vertex $v$
as a sink. 

To show that the vertex $v$ is the unique sink,
assume that the vertex $i$ is also a sink, where $i \neq v$.
In other words, for all neighbors $j$ of the vertex $i$ we have 
that~$y_{i} < y_{j}$. We can continuously move the point
$x$ in $\widetilde{S}$
by decreasing the value of the $i$th coordinate $x_{i}$.
Observe that there is no hyperplane in the periodic
arrangement blocking the coordinate $x_{i}$ from passing through
the integer value $k_{i}$ and continuing down to $k_{i}-1+\epsilon$.
This contradicts the fact
that we chose the original point~$x$ close to the optimum
of the linear functional $L$. Hence the vertex $i$ cannot be a sink.

It is straightforward to verify that
this map from regions to the set of acyclic orientations
with the unique sink at $v$ is a bijection.
\end{proof}

The technique of assigning a point to every region
of a toric arrangement using a linear functional
was used by
Novik, Postnikov and Sturmfels
in their paper~\cite{Novik_Postnikov_Sturmfels}.
See their first proof of the number of regions
of a toric arrangement.

\subsection{The toric Bayer--Sturmfels result}

Define the {\em toric Zaslavsky invariant} of a graded poset $P$
having $\hz$ and $\ho$ by
$$    \Zt(P) 
   = 
      \sum_{x \mbox{ \textrm{ \tiny coatom of} } P}
      (-1)^{\rho(\hz,x)} \cdot \mu(\hz,x)    
   = 
      (-1)^{\rho(P) - 1} \cdot
      \sum_{x \mbox{ \textrm{ \tiny coatom of} } P}
      \mu(\hz,x)    .  $$

We reformulate 
Theorem~\ref{theorem_toric_Zaslavsky_version_1}
as follows.
\begin{theorem}
For a toric hyperplane arrangement $\mathcal{H}$ 
on the torus $T^n$
that subdivides the torus into 
open $n$-dimensional balls,
the number of
regions is given by $\Zt(\mathcal{P})$, where $\mathcal{P}$
is the intersection poset of the arrangement~$\mathcal{H}$.
\label{theorem_toric_Zaslavsky_version_2}
\end{theorem}

As a corollary of Theorem~\ref{theorem_toric_Zaslavsky_version_2},
we can describe the $f$-vector of the subdivision $T_{t}$ of the torus.
For similar results for more general manifolds
see~\cite[Section~3]{Zaslavsky_paper}.
\begin{corollary}
The number of $i$-dimensional regions in the subdivision $T_{t}$
of the $n$-dimensional torus is given by the sum
$$    f_{i+1}(T_{t})
    =
      (-1)^{i}
    \cdot
      \sum_{\onethingatopanother{x \leq y}
            {\onethingatopanother{\dim(x) = i}{\dim(y) = 0}}}
              \mu(x,y)  ,
$$
where $\mu(x,y)$ denotes the M\"obius function
of the interval $[x,y]$
in the intersection poset $\mathcal{P}$.
\label{corollary_f_vector_I}
\end{corollary}
\begin{proof}
Each $i$-dimensional region is contained in a unique 
$i$-dimensional subspace $x$. By restricting the arrangement
to the subspace $x$ and applying
Theorem~\ref{theorem_toric_Zaslavsky_version_1}, we have that
the number of $i$-dimensional regions in $x$ is given by
\[
(-1)^{i} \cdot
 \sum_{x \leq y, \dim(y) = 0}   \mu(x,y).\]
Summing over all $x$, the result follows.
\end{proof}

For the remainder of this section we will assume that
the induced subdivision of the torus is a regular cell complex.
Let $T_{t}$ be the face poset of 
the subdivision of the torus induced by the
toric arrangement.
Define the map
$\zt : T_{t}^{*} \longrightarrow \Pzero$
by sending each face to the
smallest toric subspace in the intersection poset that contains
the face
and
sending
the minimal element
in $T_{t}^{*}$ to $\hz$.
Observe
that the map $\zt$ is order- and rank-preserving,
as well as being surjective.
As in the central hyperplane arrangement case,
we view the map $\zt$ as a map from the set of chains
of $T_{t}^{*}$ to the set of chains of $\Pzero$.

Let $x$ be an element in the intersection poset $\mathcal{P}$
of a toric hyperplane arrangement~$\mathcal{H}$.
Then the interval $[x,\ho]$ is the intersection poset
of a toric arrangement in the toric subspace $x$.
The atoms of the interval $[x,\ho]$ are the toric hyperplanes
in this smaller toric arrangement.

More interesting is the geometric interpretation
of the interval $[\hz,x]$. It is the intersection lattice
of a central hyperplane arrangement
in $\mathbb{R}^{n - \dim(x)}$.
Without loss of generality we may assume that
$x$ contains the zero point $(0, \ldots, 0)$,
that is, when we lift the toric subspace $x$ to an affine subspace $V$
in~$\mathbb{R}^{n}$
we may assume that $V$ is a subspace of $\mathbb{R}^{n}$.
Any toric subspace $y$ in the interval $[\hz,x]$,
that is, a toric subspace containing $x$,
can be lifted to a subspace $W$ containing the subspace $V$.
In particular, the toric hyperplanes in $[\hz,x]$
lift to hyperplanes in $\mathbb{R}^{n}$ containing $V$.
This lifting is a poset isomorphism and
we obtain an essential central arrangement
of dimension $n - \dim(x)$
by quotienting out by the subspace~$V$.
We conclude by noticing that an interval $[x,y]$ in
$\mathcal{P}$, where $y < \ho$, is the intersection lattice
of a central hyperplane arrangement.

The toric analogue of Theorem~\ref{theorem_Bayer_Sturmfels}
is as follows.
\begin{theorem}
Let $\mathcal{P}$ be the intersection poset of a toric
hyperplane arrangement
whose induced subdivision is regular.
Let $c = \{\hz = x_{0} < x_{1} < \cdots < x_{k} = \ho\}$
be a chain in $\Pzero$
with $k \geq 2$.
Then the cardinality
of the inverse image of the chain~$c$ is given by the product
\[      |\zt^{-1}(c)|
    =
        \prod_{i=2}^{k-1}
               Z([x_{i-1},x_{i}])  
    \cdot
               \Zt([x_{k-1},x_{k}])     .  \]
\label{theorem_Bayer_Sturmfels_toric}
\end{theorem}
\begin{proof}
We need to count the number of ways we can select a chain
$d = \{\hz = y_0 < y_1 < \dots < y_k = \ho\}$ in $T_{t}^{*}$
such that $\zt(y_{i}) = x_{i}$.
The number of ways to select
the element $y_{k-1}$ in $T_{t}^{*}$
is the number of regions
in the arrangement restricted to the toric subspace $x_{k-1}$.
By Theorem~\ref{theorem_toric_Zaslavsky_version_2}
this can be done in $\Zt([x_{k-1},x_{k}])$ ways.
Observe now that
all other elements in the chain $d$
contain the face~$y_{k-1}$.

To count the number of ways
to select the element $y_{k-2}$, we follow the original argument
of Bayer--Sturmfels. 
We would like to pick a face $y_{k-2}$
such that it contains the face $y_{k-1}$ and 
it is a region in the toric subspace $x_{k-2}$. 
This is equal to the number of regions in
the central arrangement having the intersection lattice $[x_{k-2},x_{k-1}]$,
which is given by $Z([x_{k-2},x_{k-1}])$.
By iterating this procedure
until we reach the element $y_{1}$, the result follows.
\end{proof}

\begin{corollary}
The flag $f$-vector entry
$f_{S}(T_{t})$
of the face poset
$T_{t}$ of a toric arrangement
whose induced subdivision is regular subdivision of $T^{n}$
is divisible by~$2^{|S|-1}$
for $S \subseteq \{1, \ldots, n+1\}$
with
$S \neq \emptyset$.
\label{corollary_evenness}
\end{corollary}
\begin{proof}
The proof follows from the fact
that the Zaslavsky invariant $Z$ is an even integer
and that a given flag $f$-vector entry is the appropriate
sum of products appearing in
Theorem~\ref{theorem_Bayer_Sturmfels_toric}.
\end{proof}

\subsection{The connection between posets and coalgebras}

For an $\ab$-monomial $v$ 
define the linear map $\lambda_{t}$ 
by letting
$$
    \lambda_{t}(v)
  =
     \begin{cases}
        (\av-\bv)^{m}   & \text{if $v=\bv^{m}$    for some $m \geq 0$,} \\
        (\av-\bv)^{m+1} & \text{if $v=\bv^{m}\av$ for some $m \geq 0$,} \\
        0               & \text{otherwise.}
\end{cases}
$$

Define the linear operator $H^{\prime}$ on $\Zab$
to be the one which removes 
the last letter in each $\ab$-monomial, that is,
$H^{\prime}(w \cdot \av) = H^{\prime}(w \cdot \bv) = w$
and $H^{\prime}(1) = 0$.
We use the prime in the notation to distinguish it from the~$H$ map
defined in~\cite[Section~8]{Billera_Ehrenborg_Readdy_om}
which instead removes the first letter in each $\ab$-monomial.
From~\cite{Billera_Ehrenborg_Readdy_om} we have the
following lemma.
\begin{lemma}
For a graded poset $P$ with $\ho$ of rank greater than or equal to $2$,
the following identity holds:
$$
H^{\prime}(\Psi(P))
      =
\sum_{x \mbox{ \textrm{ \tiny coatom of} } P} \Psi([\hz,x]) .
$$
\end{lemma}

The next lemma gives the relation between
the toric Zaslavsky invariant $\Zt$ and
the map~$\lambda_{t}$.
\begin{lemma}
For a graded poset $P$ with $\ho$ of rank greater than or equal to $1$,
the following identity holds:
$$
     \lambda_{t}(\Psi(P)) = \Zt(P) \cdot (\av-\bv)^{\rho(P)-1}.
$$
\end{lemma}
\begin{proof}
When $P$ has rank $1$, both sides are equal to $1$.
For an $\ab$-monomial $v$ different from $1$, we have that
$\lambda_{t}(v) = \beta(H^{\prime}(v)) \cdot (\av-\bv)$.
Hence
\begin{eqnarray*}
\lambda_{t}(\Psi(P))
  & = &
\beta(H^{\prime}(\Psi(P))) \cdot (\av-\bv) \\
  & = &
\sum_{x \mbox{ \textrm{ \tiny coatom of} } P}
  \beta(\Psi([\hz,x])) \cdot (\av-\bv) \\
  & = &
(-1)^{\rho(P)} \cdot
\sum_{x \mbox{ \textrm{ \tiny coatom of} } P}
  \mu(\hz,x) \cdot (\av-\bv)^{\rho(P)-1} ,
\end{eqnarray*}
which concludes the proof.
\end{proof}

Define a sequence of functions
$\varphi_{t,k}\colon\Zab\to\Zab$ by $\varphi_{t,1}=\kappa$,
and for $k \geq 2$,
$$
\varphi_{t,k}(v)
  =
\sum_{v}
     \kappa(v_{(1)}) \cdot
     \bv \cdot \eta(v_{(2)}) \cdot
     \bv \cdot \eta(v_{(3)}) \cdot \bv \cdots
     \bv \cdot \eta(v_{(k-1)}) \cdot
     \bv \cdot \lambda_{t}(v_{(k)}).                      
$$
Finally, let $\varphi_{t}(v)$ be the
sum $\varphi_{t}(v) = \sum_{k\geq 1}\varphi_{t,k}(v)$.

\begin{theorem}
The $\ab$-index of the face poset $T_{t}$ 
of a toric arrangement is given by
$$
     \Psi(T_{t})^{*} = \varphi_{t}(\Psi(\Pzero)).
$$
\label{theorem_poset_varphi_toric}
\end{theorem}
\begin{proof}
The $\ab$-index of the poset $T_{t}$
is given by the sum 
$\Psi(T_{t}) = \sum_{c} |\zt^{-1}(c)| \cdot \wt(c)$.
Fix $k \geq 2$ and sum over all chains
$c = \{\hz = x_{0} < x_{1} < \cdots < x_{k} = \ho\}$
of length $k$.  We then have
\begin{eqnarray*}
  &   &
\sum_{c} |\zt^{-1}(c)| \cdot \wt(c) \\
  & = &
\sum_{c} \prod_{i=2}^{k-1} Z([x_{i-1}, x_{i}]) 
         \cdot \Zt([x_{k-1}, x_{k}])
  \cdot
      (\av-\bv)^{\rho(x_{0},x_{1})-1}
        \cdot
      \bv
        \cdots
      \bv
        \cdot
      (\av-\bv)^{\rho(x_{k-1},x_{k})-1} \\
  & = &
\sum_{c} 
   \kappa(\Psi([x_{0},x_{1}]))
       \cdot 
   \prod_{i=2}^{k-1} \left(\bv \cdot \eta(\Psi([x_{i-1}, x_{i}]))\right)
       \cdot 
   \bv
       \cdot 
   \lambda_{t}(\Psi([x_{k-1},x_{k}])) \\
  & = &
\sum_{w}
   \kappa(w_{(1)})
       \cdot 
   \prod_{i=2}^{k-1} \left(\bv \cdot \eta(w_{(i)})\right)
       \cdot 
   \bv
       \cdot 
   \lambda_{t}(w_{(k)}) \\
  & = &
\varphi_{t,k}(w)  ,  
\end{eqnarray*}
where we let $w$ denote the $\ab$-index of 
the augmented intersection poset $\Pzero$.
For $k=1$ we have that
$(\av-\bv)^{\rho(T_{t})-1} = \varphi_{t,1}(\Psi(\Pzero))$.
Summing over all $k \geq 1$, we obtain the result.
\end{proof}

\subsection{Evaluating the function $\varphi_{t}$}

\begin{proposition}
For an $\ab$-monomial $v$,
the following identity holds:
$$ 
        \varphi_{t}(v) 
    =
        \kappa(v)
      +
        \sum_{v} \varphi(v_{(1)}) \cdot \bv \cdot \lambda_{t}(v_{(2)}) .
$$
\label{proposition_varphi_t}
\end{proposition}
\begin{proof}
Using the coassociative identity
$\Delta^{k-1} = (\Delta^{k-2} \tensor \id) \circ \Delta$,
for $k \geq 2$ we have that
\begin{eqnarray*}
\varphi_{t,k}(v)
  & = &
\sum_{v}              \kappa(v_{(1)}) \cdot
                      \bv \cdot \eta(v_{(2)}) \cdot \bv \cdots
                      \bv \cdot \eta(v_{(k-1)}) \cdot
                      \bv \cdot \lambda_{t}(v_{(k)}) \\             
  & = &
\sum_{v}
\sum_{v_{(1)}}        \kappa(v_{(1,1)}) \cdot
                      \bv \cdot \eta(v_{(1,2)}) \cdot \bv \cdots
                      \bv \cdot \eta(v_{(1,k-1)}) \cdot
                      \bv \cdot \lambda_{t}(v_{(2)}) \\             
  & = &
\sum_{v}
                      \varphi_{k-1}(v_{(1)}) \cdot
                      \bv \cdot \lambda_{t}(v_{(2)}) .
\end{eqnarray*}
By  summing over all $k \geq 1$, the result follows.
\end{proof}

\begin{lemma}
Let $v$ be an $\ab$-monomial that begins with $\av$
and let $x$ be either $\av$ or $\bv$. Then
$$     \varphi_{t}(v \cdot \av \cdot x)
   = 
       \kappa(v \cdot \av \cdot x)
     + 
       {1}/{2} \cdot \omega(v \cdot \av\bv)  .  $$
\label{lemma_toric_varphi_I}
\end{lemma}
\begin{proof}
Using Proposition~\ref{proposition_varphi_t}
we have
\begin{eqnarray*}
       \varphi_{t}(v \cdot \av \cdot x)
  & = &
       \kappa(v \cdot \av \cdot x)
     +
       \varphi(v \cdot \av) \cdot \bv \cdot \lambda_{t}(1)
     +
       \varphi(v) \cdot \bv \cdot \lambda_{t}(x)  \\
  &   &
     + \sum_{v}
       \varphi(v_{(1)}) \cdot \bv \cdot
                 \lambda_{t}(v_{(2)} \cdot \bv \cdot x)  \\
  & = &
       \kappa(v \cdot \av \cdot x)
     +
       \varphi(v) \cdot \cv \cdot \bv
     +
       \varphi(v) \cdot \bv \cdot (\av-\bv) \\
  & = &
       \kappa(v \cdot \av \cdot x)
     +
       \omega(v) \cdot \dv \\
  & = &
       \kappa(v \cdot \av \cdot x)
     +
       1/2 \cdot \omega(v \cdot \av\bv) ,
\end{eqnarray*}
since $\lambda_{t}(v_{(2)} \cdot \bv \cdot x) = 0$.
\end{proof}

\begin{lemma}
Let $v$ be an $\ab$-monomial that begins with $\av$,
let $k$ be a positive integer,
and let $x$ be either $\av$ or $\bv$.
Then
\[     \varphi_{t}(v \cdot \av \bv^{k} \cdot x)
   = 
       \kappa(v \cdot \av \bv^{k} \cdot x)
     + 
       {1}/{2} \cdot \omega(v \cdot \av\bv^{k+1})  .  \]
\label{lemma_toric_varphi_II}
\end{lemma}
\begin{proof_special}
Using Proposition~\ref{proposition_varphi_t}
we have
\begin{eqnarray}
       (\varphi_{t} - \kappa)(v \cdot \av \bv^{k}\cdot x)
  & = &
       \varphi(v \cdot \av \bv^{k}) \cdot \bv \cdot \lambda_{t}(1)
     +
       \varphi(v \cdot \av) \cdot \bv \cdot \lambda_{t}(\bv^{k-1} \cdot x)
\nonumber \\
  &   &
     {}+
       \varphi(v) \cdot \bv \cdot \lambda_{t}(\bv^{k} \cdot x) \nonumber \\
  & & {}+
       \sum_{i+j=k-2}
           \varphi(v \cdot \av \bv^{i+1}) \cdot \bv \cdot \lambda_{t}(\bv^{j} \cdot x)
\nonumber \\
  & = &
       \varphi(v)
     \cdot
       \left(
       \vphantom{\sum_{i+j=k-2}}
       2 \dv \cv^{k-1} \cdot \bv
     +
       \cv \cdot \bv \cdot (\av-\bv)^{k} 
     +
       \bv \cdot (\av-\bv)^{k+1} 
       \right.
\nonumber \\
  &   &
       \left.
     +
       \sum_{i+j=k-2}
           2 \dv \cv^{i} \cdot \bv \cdot (\av-\bv)^{j+1}
       \right) .
\label{equation_messy}
\end{eqnarray}
In order to simplify this expression, consider the butterfly poset
of rank $k$. 
This is the poset consisting
of two rank $i$ elements, for $i = 1, \ldots, k-1$,
adjoined with a minimal and maximal element.
Each of the rank $i$ elements covers
the rank $i-1$ element(s) for $i = 1, \ldots, k-1$.
The butterfly poset
is the unique poset having the $\cd$-index $\cv^{k-1}$.
It is also Eulerian.
Applying~(\ref{equation_Stanley_recursion})
to the butterfly poset, we have
$$      \cv^{k-1}
    =
        (\av-\bv)^{k-1}
      +
        2 \cdot \sum_{i+j=k-2} \cv^{i} \cdot \bv \cdot (\av-\bv)^{j}
.  $$
Using this relation to simplify
equation~(\ref{equation_messy}), we obtain
\begin{eqnarray*}
\hspace*{35 mm}
       \varphi_{t}(v \cdot \av \bv^{k}\cdot x)
     -
       \kappa(v \cdot \av \bv^{k} \cdot x)
  & = &
       \varphi(v)
     \cdot
       \dv \cdot \cv^{k} \\
  & = &
       1/2
     \cdot
       \omega(v \cdot \av \bv^{k+1}) .
\hspace*{35 mm} 
\end{eqnarray*}
This completes the proof. \qed
\end{proof_special}

By combining Lemmas~\ref{lemma_toric_varphi_I}
and~\ref{lemma_toric_varphi_II}, we have
the following proposition.
\begin{proposition}
For an $\ab$-monomial $v$ that begins with the letter $\av$,
$$     \varphi_{t}(v)
    =
       \kappa(v)
    +
       1/2
    \cdot
       \omega(H^{\prime}(v) \cdot \bv)  .  $$
\label{proposition_toric_varphi}
\end{proposition}

We now obtain the main result
for computing  the $\ab$-index
of the face poset of a toric arrangement.
\begin{theorem}
Let $\mathcal{H}$ be a toric hyperplane arrangement on
the $n$-dimensional torus $T^{n}$
that subdivides the torus into a regular cell complex.
Then the $\ab$-index of the face poset $T_{t}$
can be computed from the 
$\ab$-index of the intersection poset $\mathcal{P}$
as follows:
$$     \Psi(T_{t})
    =
        (\av-\bv)^{n+1}
    +
       \frac{1}{2}
     \cdot
       \omega(\av \cdot H^{\prime}(\Psi(\mathcal{P})) \cdot \bv)^{*}  .  $$
\label{theorem_toric}
\end{theorem}

Observe that in
Lemmas~\ref{lemma_toric_varphi_I} and~\ref{lemma_toric_varphi_II},
Proposition~\ref{proposition_toric_varphi}
and
Theorem~\ref{theorem_toric}
no
rational coefficients were introduced.
Only the $\ab$-monomial $\av^{n}$ is mapped
to a $\cd$-polynomial with an odd coefficient,
hence $1/2 \cdot \omega(v \cdot \bv)$
has all integer coefficients.

\addtocounter{theorem}{-20}
\begin{continuation}
\textrm{
The flag $f$-vector of the intersection poset $\mathcal{P}$ in
Example~\ref{example_toric_two}
is given by
$(f_{\emptyset},f_{1},f_{2},f_{12}) = (1,3,7,15)$,
the flag $h$-vector by
$(h_{\emptyset},h_{1},h_{2},h_{12}) = (1,2,6,6)$,
and so the  $\ab$-index is 
$\Psi(P) = \maa + 2 \cdot \mba + 6 \cdot \mab + 6 \cdot \mbb$.
Thus 
\begin{eqnarray*}
       \Psi(T_{t})
  & = &
        (\av-\bv)^{3}
    +
       1/2
     \cdot
       \omega(\av \cdot 
              H^{\prime}(\maa + 2 \cdot \mba + 6 \cdot \mab + 6 \cdot \mbb)
              \cdot \bv)^{*}  \\
  & = &
        (\av-\bv)^{3}
    +
       1/2
     \cdot
       \omega(\av \cdot (7 \cdot \ma + 8 \cdot \mb)
                \cdot \bv)^{*}  \\
  & = &
        (\av-\bv)^{3}
    +
       1/2
     \cdot
       \omega(7 \cdot \maab + 8 \cdot \mabb)^{*}  \\
  & = &
        (\av-\bv)^{3}
    +
       7 \cdot \mdc + 8 \cdot \mcd ,
\end{eqnarray*}
which agrees with the
calculation in Example~\ref{example_toric_two}.
}
\end{continuation}
\addtocounter{theorem}{19}

Theorem~\ref{theorem_toric}
gives a different approach 
from
Corollary~\ref{corollary_f_vector_I}
for determining 
the $f$-vector of $T_{t}$.
For notational ease, for positive integers $i$ and $j$,
let $[i,j] = \{ i, i+1, \ldots, j\}$
and
$[j] = \{1, \ldots, j\}$.

\begin{corollary}
The number of $i$-dimensional regions in the subdivision $T_{t}$
of the $n$-dimensional torus is given by the following
sum
of  flag $h$-vector entries from
the intersection poset $\mathcal{P}$:
$$
      f_{i+1}(T_{t})
  =
      h_{[n-i,n]}(\mathcal{P})
   +
      h_{[n-i,n-1]}(\mathcal{P}) 
   +
      h_{[n-i+1,n]}(\mathcal{P})
   +
      h_{[n-i+1,n-1]}(\mathcal{P})     ,
$$
for $1 \leq i \leq n-1$.
The number
of vertices is given by
$f_{1}(T_{t}) = 1 + h_{n}(\mathcal{P})$
and the number of
maximal regions
by
$f_{n+1}(T_{t})
 = h_{[n-1]}(\mathcal{P})
 + h_{[n]}(\mathcal{P})$.
\label{corollary_f_vector_II}
\end{corollary}
\begin{proof}
Let $\pair{\cdot}{\cdot}$ denote the inner product on $\Zab$
defined by $\pair{u}{v} = \delta_{u,v}$
for two $\ab$-monomials $u$ and $v$.
For $1 \leq i \leq n-1$ we have
\begin{eqnarray*}
         f_{i+1}(T_{t})
  & = &
         1 + h_{i+1}(T_{t}) \\
  & = &
         1 + \pair{\av^{i} \bv \av^{n-i}}{\Psi(T_{t})} \\
  & = &
       \frac{1}{2}
     \cdot
       \pair{\av^{i} \bv \av^{n-i}}
       {\omega(\av \cdot H^{\prime}(\Psi(\mathcal{P})) \cdot \bv)^{*}}  \\
  & = &
       \frac{1}{2}
     \cdot
       [\cv^{i-1} \dv \cv^{n-i}]
       \omega(\av \cdot H^{\prime}(\Psi(\mathcal{P})) \cdot \bv)^{*}  
    +
       \frac{1}{2}
     \cdot
       [\cv^{i} \dv \cv^{n-i-1}]
       \omega(\av \cdot H^{\prime}(\Psi(\mathcal{P})) \cdot \bv)^{*}  \\
  & = &
       \pair{\av^{n-i} \cdot \av\bv \cdot \bv^{i-1}
              +
             \av^{n-i-1} \cdot \av\bv \cdot \bv^{i}}
            {\av \cdot H^{\prime}(\Psi(\mathcal{P})) \cdot \bv} \\
  & = &
       \pair{\av^{n-i-1} \cdot (\av+\bv) \cdot \bv^{i-1}}
            {H^{\prime}(\Psi(\mathcal{P}))} \\
  & = &
       \pair{\av^{n-i-1} \cdot (\av+\bv) \cdot \bv^{i-1} \cdot (\av+\bv)}
            {\Psi(\mathcal{P})} .
\end{eqnarray*}
Expanding in terms of the flag $h$-vector the result follows.
The expressions for $f_{1}$ and $f_{n+1}$ are
obtained by similar calculations.
\end{proof}

The fact that Corollaries~\ref{corollary_f_vector_I}
and~\ref{corollary_f_vector_II}
are equivalent follows from
the coalgebra techniques
in Theorem~\ref{theorem_Ehrenborg_Readdy}.

\section{The complex of unbounded regions}
\label{section_affine}

\subsection{Zaslavsky and Bayer--Sturmfels}

The {\em unbounded Zaslavsky invariant} is defined by
$$    \Zub(P)
    =
      Z(P) - 2 \cdot Z_{b}(P)  .  $$
As the name suggests, the number of unbounded regions in 
a non-central
arrangement is given by this invariant.
By taking the difference of the
two statements in Theorem~\ref{theorem_Zaslavsky_poset} part (ii),
we immediately obtain the following result.
\begin{lemma}
For a non-central hyperplane arrangement $\mathcal{H}$ the number of
unbounded regions is given by $\Zub(\mathcal{L})$,
where $\mathcal{L}$
is the intersection lattice of the arrangement $\mathcal{H}$.
\label{lemma_Zaslavsky_unbounded}
\end{lemma}

\begin{figure}
\begin{center}
\scalebox{.7}{\epsfig{file=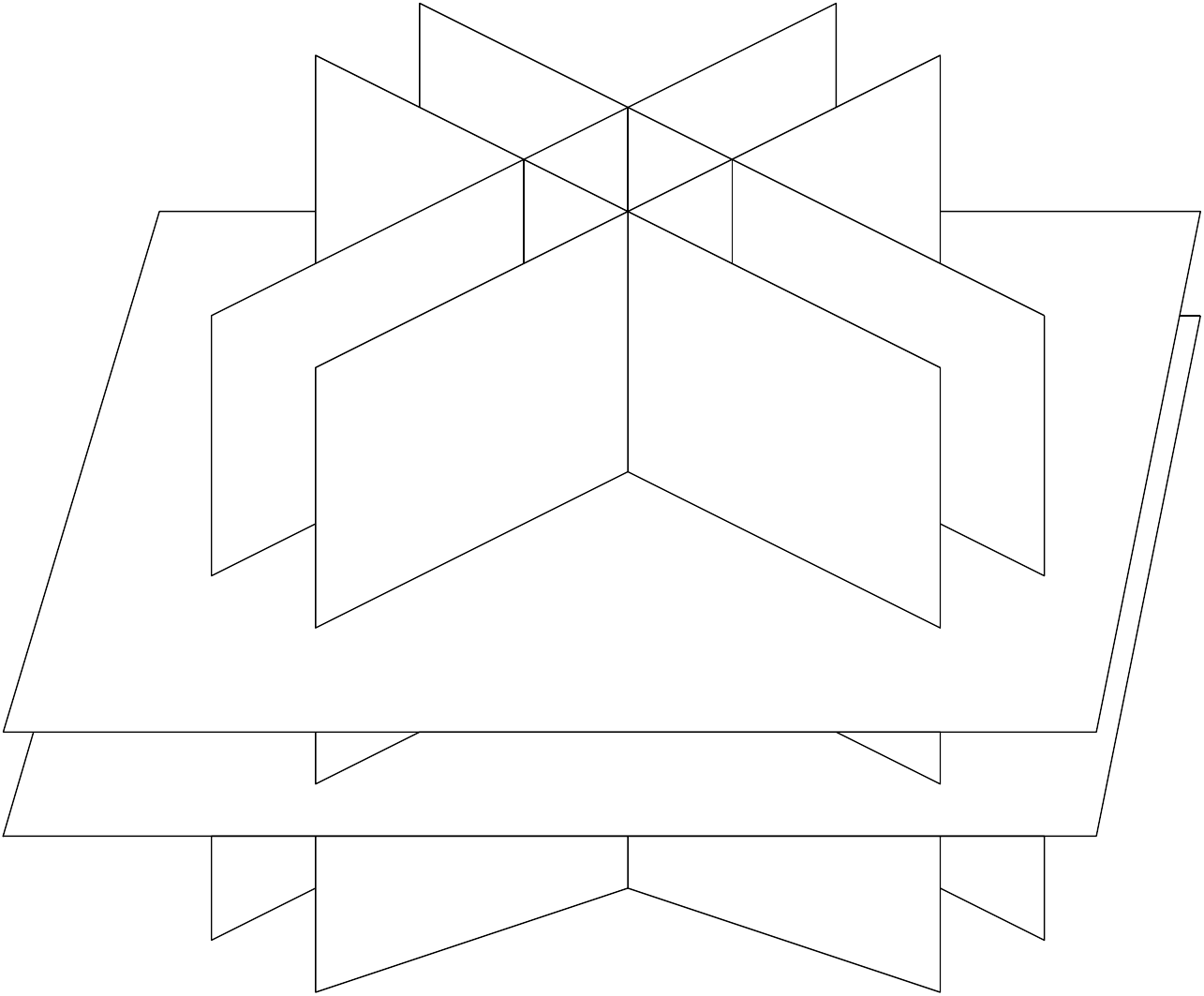}}
\end{center}
\caption{The non-central arrangement $x,y,z=0,1$.}
\label{figure_affine}
\end{figure}

Let $\mathcal{H}$ be a non-central hyperplane arrangement 
in $\mathbb{R}^n$
with intersection
lattice~$\mathcal{L}$ having the empty set $\emptyset$
as the maximal element.
Let $\mathcal{L}_{ub}$ denote 
the {\em unbounded intersection lattice},
that is,
the subposet of the intersection lattice
consisting of all affine subspaces 
with the points (dimension zero affine subspaces) omitted
but with the empty set $\emptyset$ continuing to be the maximal element.
Equivalently, the poset
$\mathcal{L}_{ub}$ is the rank-selected poset
$\mathcal{L}([1, n-1])$,
that is,
the poset
$\mathcal{L}$ with the coatoms removed.

Let $T$ be the face lattice of the arrangement
$\mathcal{H}$ with the minimal element $\hz$ denoting the empty face
and the maximal element denoted by $\ho$.
Similarly, let $T_{ub}$ denote the set of all faces
in the face lattice~$T$ which are not bounded.
Observe that~$T_{ub}$ includes the minimal and maximal elements of $T$
and that $T_{ub}$ is the face poset of
an $(n-1)$-dimensional sphere. 
Pick~$R$ large enough
so that all of the bounded faces
are strictly inside a ball of radius~$R$.
Intersect
the arrangement~$\mathcal{H}$ with a sphere of radius $R$.
The resulting cell complex has face poset $T_{ub}$.
Our goal is to compute the $\cd$-index of $T_{ub}$ in terms
of the $\ab$-index of $\mathcal{L}_{ub}$.

The collection of unbounded
faces of the arrangement $\mathcal{H}$
forms a lower order ideal in the poset $T^{*}$.
Let~$Q$ be the subposet of $T^{*}$
consisting of this ideal with a  maximal element $\ho$
adjoined.
We define the rank of an element in $Q$ to be its rank
in the original poset $T^{*}$, that is, for $x \in Q$
let $\rho_{Q}(x) = \rho_{T^{*}}(x)$.
This rank convention will simplify the later arguments.
As posets, $T_{ub}^{*}$ and $Q$ are isomorphic.
However, since their rank functions differ, their $\ab$-indexes 
satisfy
$\Psi(T_{ub})^{*} \cdot (\av-\bv) = \Psi(Q)$.

Restrict the zero map
$z : T^{*} \longrightarrow \Lzero$
to 
form the map $\zub : Q \longrightarrow \Lzero$.
The map $\zub$ is order- and rank-preserving.
However, it is not necessarily surjective.
As before we view the map $\zub$ as a map from the set of chains
of $Q$ to the set of chains of $\Lzero$.
The following theorem is a toric deformation of 
Theorem~\ref{theorem_Bayer_Sturmfels}.
\newpage
\begin{theorem}
Let $\mathcal{H}$ be a non-central hyperplane arrangement
with intersection lattice $\mathcal{L}$.
Let $c=\{\hz = x_0 < x_1 < \dots < x_k = \ho\}$ be a chain in
$\Lzero$ with $k \geq 2$.
Then the cardinality of the inverse image of the chain $c$
under $\zub$ is given by
$$      |\zub^{-1}(c)|
     =
        \prod_{i=2}^{k-1}
             Z([x_{i-1},x_{i}])
       \cdot
             \Zub([x_{k-1},x_{k}]) . $$
\label{theorem_Bayer_Sturmfels_unbounded}
\end{theorem}
\begin{proof}
We need to count the number of ways we can select a chain
$d = \{\hz = y_0 < y_1 < \dots < y_k = \ho\}$ in 
the poset of unbounded regions $Q$
such that $\zub(y_{i}) = x_{i}$.
The number of ways to select
the element~$y_{k-1}$ in $Q$
is the number of unbounded regions
in the arrangement restricted to the subspace $x_{k-1}$.
By Lemma~\ref{lemma_Zaslavsky_unbounded}
this can be done in $\Zub([x_{k-1},x_{k}])$ ways.
Since $y_{k-1}$ is an unbounded face of the arrangement
and all other elements in the chain $d$
contain the face $y_{k-1}$, the other elements must be unbounded.

The remainder of the proof is the same as
that of Theorem~\ref{theorem_Bayer_Sturmfels_toric}.
\end{proof}

\begin{corollary}
The flag $f$-vector entry
$f_{S}(T_{ub})$
is divisible by $2^{|S|}$
for any index set $S \subseteq \{1, \ldots, n\}$.
\end{corollary}
\begin{proof}
The proof is the same as 
Corollary~\ref{corollary_evenness}
with the extra observation that
the Zaslavsky invariant $\Zub$ is even.
\end{proof}

\subsection{The connection between posets and coalgebras}

Define $\lambda_{ub}$ by
$\lambda_{ub} = \eta - 2 \cdot \beta$.
By equations~(\ref{equation_beta}) 
and~(\ref{equation_eta}),
for a graded poset $P$ we have
$$  \lambda_{ub}(\Psi(P)) = \Zub(P) \cdot (\av-\bv)^{\rho(P)-1}  .  $$
Define a sequence of functions
$\varphi_{ub,k}\colon\Zab\to\Zab$ by $\varphi_{ub,1}=\kappa$ and for
$k>1$,
$$
\varphi_{ub,k}(v)
     =
\sum_{v}
       \kappa(v_{(1)}) \cdot
       \bv \cdot \eta(v_{(2)}) \cdot
       \bv \cdot \eta(v_{(3)}) \cdot \bv \cdots
       \bv \cdot \eta(v_{(k-1)}) \cdot
       \bv \cdot \lambda_{ub}(v_{(k)}).                      
$$
Finally, let $\varphi_{ub}(v)$ be the
sum $\varphi_{ub}(v) = \sum_{k\geq 1}\varphi_{ub,k}(v)$.

Similar to Theorem~\ref{theorem_poset_varphi_toric} we
have the next result.
The proof only differs in replacing
the map
$\zt : T_{t}^{*} \longrightarrow \Pzero$
with 
$\zub : Q \longrightarrow \Lzero$
and the invariant $\Zt$ by $\Zub$.
\begin{theorem}
The $\ab$-index of the poset $Q$ 
of unbounded regions of a
non-central 
arrangement is given by
$$
   \Psi(Q) = \varphi_{ub}(\Psi(\Lzero)).
$$
\label{theorem_poset_varphi_unbounded}
\end{theorem}

\subsection{Evaluating the function $\varphi_{ub}$}

In this subsection we
analyze the behavior of $\varphi_{ub}$.
\begin{lemma}
For any $\ab$-monomial $v$,
$$       \varphi_{ub}(v)
     =
         \varphi(v)
       -
         2 \cdot 
         \sum_v \varphi(v_{(1)}) \cdot \bv \cdot \beta(v_{(2)}) . $$
\label{lemma_phi_ub_recursion}
\end{lemma}
\begin{proof}
Using the coassociative identity
$\Delta^{k-1} = (\Delta^{k-2} \tensor \id) \circ \Delta$,
we have for $k \geq 2$
\begin{eqnarray*}
\varphi_{ub,k}(v)
  & = &
\varphi_{k}(v)
  -
2 \cdot 
\sum_{v}            \kappa(v_{(1)}) \cdot
                    \bv \cdot \eta(v_{(2)}) \cdot
                    \bv \cdots
                    \bv \cdot \eta(v_{(k-1)}) \cdot
                    \bv \cdot \beta(v_{(k)}) \\
  & = &
\varphi_{k}(v)
  -
2 \cdot 
\sum_{v}
\sum_{v_{(1)}}      \kappa(v_{(1,1)}) \cdot
                    \bv \cdot \eta(v_{(1,2)}) \cdot
                    \bv \cdots
                    \bv \cdot \eta(v_{(1,k-1)}) \cdot
                    \bv \cdot \beta(v_{(2)}) \\
  & = &
\varphi_{k}(v)
  -
2 \cdot 
\sum_{v}
\varphi_{k-1}(v_{(1)}) \cdot \bv \cdot \beta(v_{(2)}) .
\end{eqnarray*}
The result then follows by
summing over all $k \geq 2$ and adding
$\varphi_{ub,1}(v) = \kappa(v) = \varphi_{1}(v)$.
\end{proof}

\begin{lemma}
Let $v$ be an $\ab$-monomial.  Then 
$$   \varphi_{ub}(v \cdot \av)
   =
      \varphi(v) \cdot (\av - \bv) .   $$
\label{lemma_varphi_ub_I}
\end{lemma}
\begin{proof}
By Lemma~\ref{lemma_phi_ub_recursion}
and the Leibniz relation~(\ref{equation_Newtonian}) we have
$$   \varphi_{ub}(v \cdot \av)
 = 
     \varphi(v \cdot \av) 
   -
     2 \cdot
     \varphi(v) \cdot \bv \cdot \beta(1)
   -
     2 \cdot
     \sum_v \varphi(v_{(1)}) \cdot \bv \cdot \beta(v_{(2)} \cdot \av) .  $$
By equation~(\ref{equation_varphi_a})
$\varphi(v \cdot \av) = \varphi(v) \cdot \cv$.
The summation above is
zero because $\beta(v_{(2)} \cdot \av)$ is always zero.  Hence
$\varphi_{ub}(v \cdot \av) 
   =
  \varphi(v) \cdot (\cv - 2\bv)
   =
  \varphi(v) \cdot (\av - \bv)$.
\end{proof}

\begin{lemma}
Let $v$ be an $\ab$-monomial.  Then
\[
\varphi_{ub}(v \cdot \bv\bv) = \varphi_{ub}(v \cdot \bv) \cdot (\av - \bv).
\]
\label{lemma_varphi_ub_II}
\end{lemma}
\begin{proof}
Let $u = v \cdot \bv$. 
Applying
Lemma~\ref{lemma_phi_ub_recursion}
and the Leibniz relation~(\ref{equation_Newtonian})
to $u$ gives
\begin{eqnarray*}
\varphi_{ub}(u \cdot \bv) 
  & = &
\varphi(u \cdot \bv)
  - 
    2 \cdot
    \varphi(u) \cdot \bv \cdot \beta(1)
  - 
    2 \cdot
    \sum_{u} \varphi(u_{(1)}) \cdot \bv \cdot \beta(u_{(2)} \cdot \bv) \\
  & = &
\varphi(u) \cdot (\cv - 2\bv)
  - 
    2 \cdot
    \sum_{u} \varphi(u_{(1)}) \cdot \bv \cdot \beta(u_{(2)}) \cdot (\av - \bv) \\
  & = &
\left(\varphi(u) 
      - 
      2 \cdot
      \sum_{u} \varphi(u_{(1)}) \cdot \bv \cdot 
                 \beta(u_{(2)})
\right) \cdot (\av - \bv) \\
  & = &
\varphi_{ub}(u) \cdot (\av - \bv).
\end{eqnarray*}
Here we have used the facts that
$\varphi(u \cdot \bv) = \varphi(u) \cdot \cv$
and 
$\beta(u_{(2)} \cdot \bv) = \beta(u_{(2)}) \cdot (\av-\bv)$.
\end{proof}

\begin{lemma}
Let $v$ be an $\ab$-monomial.
Then $\varphi_{ub}(v \cdot \ab) = 0$.
\label{lemma_varphi_ub_III}
\end{lemma}
\begin{proof}
Directly we have
\begin{align*}
\varphi_{ub}(v \cdot \ab)
  & = 
  \varphi(v \cdot \ab)
   - 
     2 \cdot
     \varphi(v) \cdot \bv \cdot \beta(\bv) \\
   &\vphantom{=}
   - 
     2 \cdot
     \varphi(v \cdot \av) \cdot \bv \cdot \beta(1)  \\
   &\vphantom{=}
   - 
     2 \cdot
     \sum_v \varphi(v_{(1)}) \cdot \bv \cdot \beta(v_{(2)} \cdot \av \bv) \\
  & = 
  \varphi(v) \cdot 2\dv
- 
  2 \cdot \varphi(v) \cdot \bv \cdot (\av - \bv)
- 
  2 \cdot \varphi(v) \cdot \cv \bv \\
  & = 
  2 \cdot \varphi(v) \cdot (\dv - \bv(\av - \bv) - \cv\bv) \\
  & = 
0 ,
\end{align*}
where we have used the facts that
$\varphi(v \cdot \av\bv) = \varphi(v) \cdot 2\dv$
and 
$\beta(v_{(2)} \cdot \av\bv) = 0$.
\end{proof}

The previous three lemmas enable us to determine $\varphi_{ub}$.
In order to obtain more compact notation, 
define a map $r\colon\Zab\to\Zab$ by $r(1)=0$,
$r(v \cdot \av)=v$,
and
$r(v \cdot \bv)=0$.
By using the chain definition of the $\ab$-index,
it is straightforward to see that
$\Psi(\mathcal{L}_{ub}) = r(\Psi(\mathcal{L}))$.

\begin{proposition}
Let $w$ be an $\ab$-polynomial homogeneous of degree greater than zero.
Then
$$    \varphi_{ub}(\av \cdot w)
    =
      \omega(\av \cdot r(w)) \cdot (\av - \bv)   .  $$
\end{proposition}
\begin{proof}
The case
$w = v \cdot \av$ 
follows from Lemma~\ref{lemma_varphi_ub_I}.
The remaining  case is $w = v \cdot \bv$.
Note that $\av \cdot v \cdot \bv$ can be factored as
$u \cdot \av\bv \cdot \bv^{k}$ for a monomial $u$.
Hence
$\varphi_{ub}(u \cdot \av\bv \cdot \bv^{k})
  =
 \varphi_{ub}(u \cdot \av\bv) \cdot (\av-\bv)^{k}
  =
 0$
by Lemmas~\ref{lemma_varphi_ub_II}
and~\ref{lemma_varphi_ub_III}.
\end{proof}

We combine all of these results to 
conclude that the $\cd$-index of
the poset of unbounded regions~$T_{ub}$ can be computed
in terms of the $\ab$-index of the unbounded intersection lattice
$\mathcal{L}_{ub}$.

\begin{theorem}
Let $\mathcal{H}$ be a non-central hyperplane arrangement with
the unbounded intersection
lattice $\mathcal{L}_{ub}$ and poset of unbounded regions $T_{ub}$.
Then the $\ab$-index of $T_{ub}$ is given by
$$    \Psi(T_{ub})
    =
      \omega(\av \cdot \Psi(\mathcal{L}_{ub}))^{*} . $$
\label{theorem_unbounded}
\end{theorem}
\begin{proof}
We have that
\begin{eqnarray*}
\Psi(T_{ub})^{*} \cdot (\av-\bv)
  & = &
\Psi(Q) \\
  & = &
\varphi_{ub}(\av \cdot \Psi(\mathcal{L})) \\
  & = &
\omega(\av \cdot r(\Psi(\mathcal{L}))) \cdot (\av-\bv) \\
  & = &
\omega(\av \cdot \Psi(\mathcal{L}_{ub})) \cdot (\av-\bv) .
\end{eqnarray*}
The result follows by cancelling $\av - \bv$ from both sides of the identity.
\end{proof}

\begin{figure}
\begin{center}
\scalebox{.3}{\epsfig{file=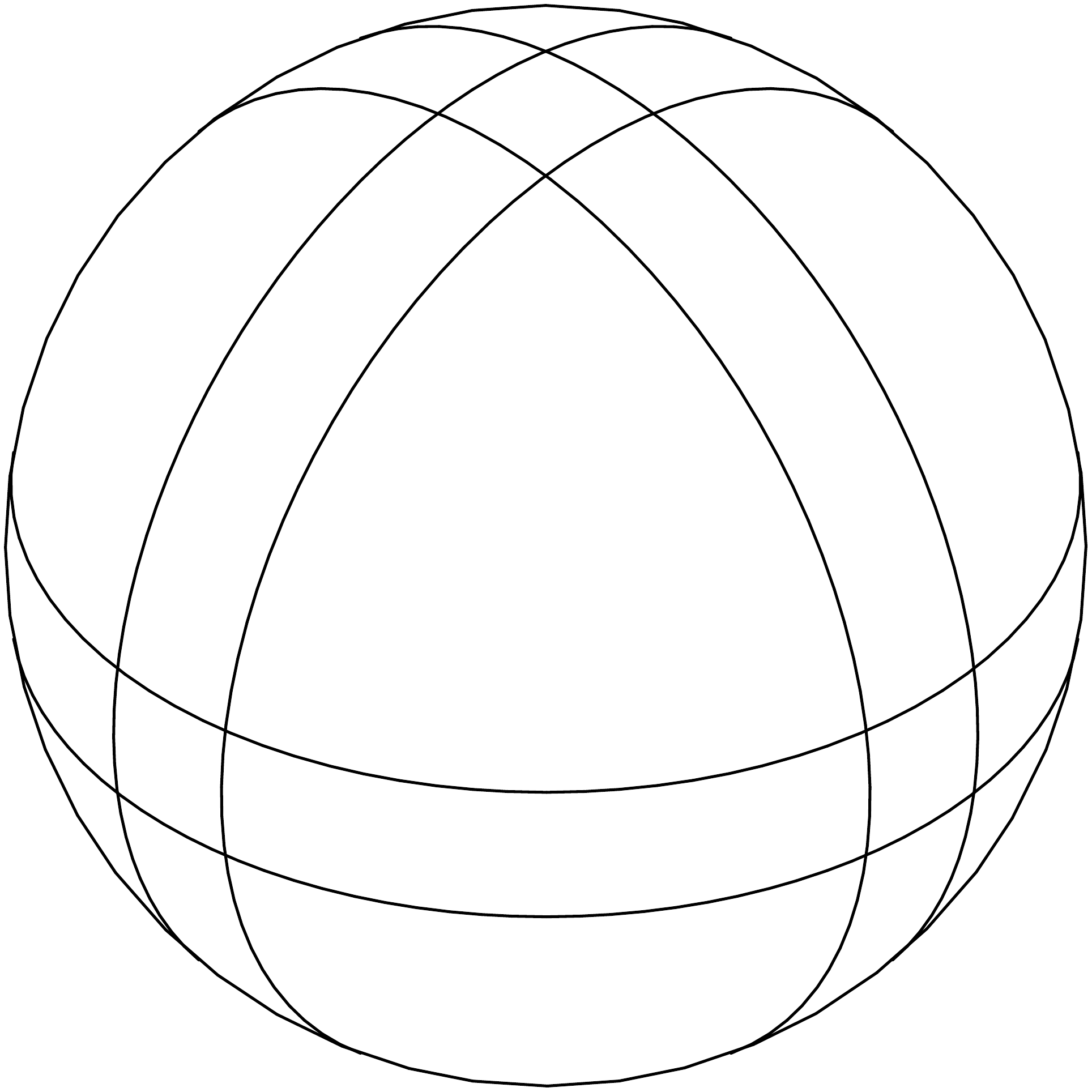}}
\end{center}
\caption{The spherical subdivision
obtained from the non-central arrangement $x,y,z=0,1$.}
\label{figure_spherical}
\end{figure}

\begin{example}
Consider the non-central hyperplane arrangement
consisting of the six hyperplanes
$x = 0,1$, 
$y = 0,1$ and
$z = 0,1$.
See Figure~\ref{figure_affine}.
After intersecting this arrangement
with a sphere of large enough radius
we obtain the cell complex in
Figure~\ref{figure_spherical}.
The polytopal realization of this complex
is known as the rhombicuboctahedron.
The dual of the face lattice of this spherical complex is
not realized by a zonotope. However, one can view
the dual lattice as the face lattice
of a $2 \times 2 \times 2$ pile of cubes.

The intersection lattice $\mathcal{L}$ is the
face lattice of the three-dimensional crosspolytope,
in other words,
the octahedron.
Hence the lattice of unbounded intersections
$\mathcal{L}_{ub}$
has the flag $f$-vector
$(f_{\emptyset},f_{1},f_{2},f_{12}) = (1,6,12,24)$
and the flag $h$-vector 
$(h_{\emptyset},h_{1},h_{2},h_{12}) = (1,5,11,7)$.
The $\ab$-index is given by
$\Psi(\mathcal{L}_{ub}) = \maa + 5 \cdot \mba + 11 \cdot \mab + 7 \cdot \mbb$.
Hence the $\cd$-index of $T_{ub}$ is
\begin{eqnarray*}
\Psi(T_{ub})
  & = &
      \omega(\maaa + 5 \cdot \maba + 11 \cdot \maab + 7 \cdot \mabb)^{*} \\
  & = &
      \mccc + 22 \cdot \mdc + 24 \cdot \mcd  .
\end{eqnarray*}
\end{example}

\section{Concluding remarks}

For regular subdivisions of manifolds questions abound.
\begin{itemize}
\item[(i)] What is the right analogue of a regular subdivision
in order that it  be polytopal? 
Can flag $f$-vectors be classified for
polytopal subdivisions?
\item[(ii)] Is there a Kalai convolution for manifolds
that will generate more inequalities for flag $f$-vectors?
\cite{Kalai}
\item[(iii)] Is there a lifting technique that will yield
             more inequalities for higher dimensional manifolds?
\cite{Ehrenborg_lifting}
\item[(iv)] Are there minimization inequalities for the
$\cd$-coefficients in the polynomial~$\Psi$?
As a first step, can one prove the non-negativity of $\Psi$?
\cite{Billera_Ehrenborg,Ehrenborg_Karu}
\item[(v)] Is there an extension of the toric $g$-inequalities
to manifolds?
\cite{Bayer_Ehrenborg,Kalai_g,Karu,Stanley_h}
\item[(vi)] Can the coefficients for $\Psi$ be minimized for regular toric
arrangements as was done in the case of central hyperplane arrangements?
\cite{Billera_Ehrenborg_Readdy_om}
\end{itemize}

The most straightforward manifold to study is
$n$-dimensional projective space~$P^{n}$. We offer the following
result in obtaining the $\ab$-index of subdivisions of $P^{n}$.
\begin{theorem}
Let $\Omega$ be a centrally symmetric regular subdivision of
the $n$-dimensional sphere $S^{n}$. Assume that when 
antipodal points of the sphere are identified,
a regular subdivision
$\Omega^{\prime}$ of the projective space $P^{n}$
is obtained.
Then the $\ab$-index of $\Omega^{\prime}$ is given by
$$    \Psi(\Omega^{\prime})
   =
      \frac{\cv^{n+1} + (\av-\bv)^{n+1}}{2}
   +
      \frac{\Phi}{2}   ,   $$
where the $\cd$-index of $\Omega$ is
$\Psi(\Omega) = \cv^{n+1} + \Phi$.
\end{theorem}
\begin{proof}
Each chain $c = \{\hz = x_{0} < x_{1} < \cdots < x_{k} = \ho\}$ with
$k \geq 2$ in $\Omega^{\prime}$ corresponds to two chains in~$\Omega$
with the same weight $\wt(c)$. The chain
$c = \{\hz = x_{0} < x_{1} = \ho\}$ corresponds to exactly
one chain in~$\Omega$ and has weight $(\av-\bv)^{n+1}$.
Hence
$\Psi(\Omega) = 2 \cdot \Psi(\Omega^{\prime}) - (\av-\bv)^{n+1}$,
proving the result.
\end{proof}

The results in this chapter have been stated
for hyperplane arrangements. In true generality
one could
work with the underlying oriented matroid,
especially  since there are nonrealizable ones 
such as the non-Pappus oriented matroid.
All of these can be  represented as pseudo-hyperplane
arrangements. 
We chose to work with hyperplane arrangements to preserve
the geometric intuition.

Poset transformations related to the $\omega$ map have
been considered 
in~\cite{Ehrenborg_r_Birkhoff,Ehrenborg_Readdy_Tchebyshev,Hsiao}.
Are there toric or affine analogues of these poset transforms?

Another way to encode the flag $f$-vector data of a poset
is to use the quasisymmetric function of a
poset~\cite{Ehrenborg_Hopf}.
In this language the $\omega$ map is translated to
Stembridge's~$\vartheta$ map;
see~\cite{Billera_Hsiao_van_Willigenburg,Stembridge}.
Would the results of
Theorems~\ref{theorem_toric}
and~\ref{theorem_unbounded}
be appealing in the quasisymmetric function viewpoint?

Richard Stanley has asked if the coefficients of the toric
characteristic polynomial are alternating. If so, is there any
combinatorial interpretation of the absolute values of the
coefficients.

A far reaching generalization of Zaslavsky's results for hyperplane
arrangements is by Goresky and MacPherson~\cite{Goresky_MacPherson}.
Their results determine the cohomology groups of the
complement of a complex hyperplane arrangement. 
For a toric analogue of the Goresky--MacPherson results,
see work of De~Concini and Procesi~\cite{De_Concini_Procesi}.
For algebraic considerations of toric arrangements,
see~\cite{Douglass,Macmeikan_I, Macmeikan_II, Macmeikan_III}.

In Section~\ref{section_toric} we restricted ourselves
to studying arrangements that cut the torus into
regular cell complexes. In a future paper~\cite{Ehrenborg_Slone},
two of the authors are developing the notion
of a $\cd$-index for non-regular cell complexes.

%

\vfill
\begin{center}
Copyright \copyright\ Michael Slone 2008
\end{center}

\vanish{

\newcommand{\journal}[6]{{\sc #1,} #2, {\it #3} {\bf #4} (#5), #6.}
\newcommand{\books}[6]{{\sc #1,} ``#2,'' #3, #4, #5, #6.}
\newcommand{\collection}[6]{{\sc #1,}  #2, #3, in {\it #4}, #5, #6.}
\newcommand{\thesis}[4]{{\sc #1,} ``#2,'' Doctoral dissertation, #3, #4.}
\newcommand{\springer}[4]{{\sc #1,} ``#2,'' Lecture Notes in Math.,
                          Vol.\ #3, Springer-Verlag, Berlin, #4.}
\newcommand{\preprint}[3]{{\sc #1,} #2, preprint #3.}
\newcommand{\preparation}[2]{{\sc #1,} #2, in preparation.}
\newcommand{\appear}[3]{{\sc #1,} #2, to appear in {\it #3}}
\newcommand{\submitted}[3]{{\sc #1,} #2, submitted to {\it #3}}
\newcommand{\JCTA}{J.\ Combin.\ Theory Ser.\ A}
\newcommand{\AdvancesinMathematics}{Adv.\ Math.}
\newcommand{\JournalofAlgebraicCombinatorics}{J.\ Algebraic Combin.}

\newcommand{\communication}[1]{{\sc #1,} personal communication.}


\bigskip

{\em R.\ Ehrenborg, M.\ Readdy and M.\ Slone,
Department of Mathematics,
University of Kentucky,
Lexington, KY 40506-0027,}
\{{\tt jrge},{\tt readdy},{\tt mslone}\}{\tt @ms.uky.edu}

\end{document}

}
%
\setcounter{chapter}{1}
\chapter{Mixing operators}
\label{chap:mixing}

\section{Introduction}

Kalai~\cite{Kalai} showed that a basis for flag $f$-vectors of polytopes is
given by the flag $f$-vectors of polytopes constructed from simplices by
repeatedly taking joins or products.  
Ehrenborg and Readdy~\cite{Ehrenborg_Readdy_c} studied how 
the $\cd$-index changes
under these operations. They discovered bilinear operators on the
Newtonian coalgebra $\mathbb{Z}\fcd$ which they called the mixing
operator (for joins of polytopes) and the diamond operator (for
products of polytopes).  Later, Ehrenborg and Fox~\cite{Ehrenborg_Fox}
analyzed these operators further, obtaining recursive coalgebraic
formulas for the $\cd$-indices of joins and products of polytopes.
Using these formulas, they obtained a $\cd$-index inequality relating
the product of a join with the join of a product, 
providing evidence for Stanley's Gorenstein${}^*$
conjecture, which was only settled later~\cite{Ehrenborg_Karu}.

It is difficult to use
the join and product operations 
to study non-spherical manifolds, such as tori, since both
preserve Eulerianness
and take spheres to spheres.
To remedy this difficulty, we introduce the manifold product.
This is defined on manifolds as the Cartesian product of the
underlying cell complexes, and yields a bilinear operator
on $\ab$-indices.
A manifold product of Eulerian manifolds is not globally Eulerian, 
but it is locally Eulerian.  
We extend inequalities proved by Ehrenborg and Fox to the case
of manifold products.


The mixing and diamond operators are
nonnegative operators on $\cd$-indices.  Therefore, it makes sense to
ask if there is something the coefficients count.  We prove
that the coefficients of the $\cd$-index of the
diamond product of two butterfly posets, which have pure $\cv$-power
$\cd$-indices, can be interpreted as a weighted sum of restricted
lattice paths.  This
also extends to a lattice-path interpretation for the
coefficients of the mixing operator applied to pure $\cv$-power terms.
We also extend this interpretation to the manifold
operator in the situation where the manifold operator yields a near
$\cd$-index which is nonnegative.

\section{Preliminaries}

For any cell complex $X$, let $\lat{X}$ denote its face poset.  The
empty face $\hz$ and the total complex $\ho$ are faces in $\lat{X}$.
If $X$ is a polytope, then $\lat{X}$ is a lattice.   

A \define{graded poset} is a poset $P$ with 
distinct minimum and maximum elements $\hz$ and $\ho$
which is equipped with a
\define{rank function} $\rho\colon P\to\mathbb{N}$.  The rank
function must preserve covers and send the minimum element of $P$
to $0$.  In other words, $\rho(\hz)=0$, and if $x<:y$ in $P$, 
then $\rho(x)+1 = \rho(y)$.
The face poset of a finite regular cell complex, such as a polytope,
is graded by dimension.

Fix once and for all a collection $\G$ which has exactly
one representative of each isomorphism class of finite graded
posets.  From now on, we identify each 
graded poset with its isomorphic representative in $\G$.

Fix a ground ring $k$.  All modules, algebras, and coalgebras we discuss
will be over this ground ring.

An \define{algebra} is a module $A$ together
with linear structure maps
$\nabla\colon A\tensor A\to A$, called the product, and
$\eta\colon k\to A$, called the unit, 
such that the
diagrams
\[\xymatrix{
A\tensor A\tensor A\ar[r]^{~~\nabla\tensor\id}\ar[d]_{\id\tensor\nabla} &
A\tensor A\ar[d]^{\nabla}  \\
A\tensor A\ar[r]_{\nabla} & A 
} \text{\quad and\quad}
\xymatrix{
k\tensor A\ar[r]^{\eta\tensor\id}\ar[dr]_{\cong} & 
A\tensor A\ar[d]^{\nabla} &
A\tensor k\ar[l]_{\id\tensor\eta}\ar[dl]^{\cong} \\
& A &
}\]
are commutative.  
If $A$ and $B$ are algebras, then an algebra morphism
from $A$ to $B$ is a linear map $\varphi\colon A\to B$ which respects
the product and unit.  That is, 
$\nabla_B\circ(\varphi\tensor\varphi) = \varphi\circ \nabla_A$ and $\varphi\circ\eta_A=\eta_B$.

Dually, a \define{coalgebra} is a module $C$ together 
with linear structure maps
$\Delta\colon C\to C\tensor C$, called the coproduct, and 
$\varepsilon\colon C\to k$, called the counit, 
such that
the diagrams
\[\xymatrix{
C\tensor C\tensor C & C\tensor C\ar[l]_{~~~~\Delta\tensor\id} \\
C\tensor C\ar[u]^{\id\tensor\Delta} & C\ar[l]^{\Delta}\ar[u]_{\Delta}
} \text{\quad and\quad}
\xymatrix{
k\tensor C & 
C\tensor C\ar[l]_{\varepsilon\tensor\id}\ar[r]^{\id\tensor\varepsilon} &
C\tensor k \\
& C\ar[ul]^{\cong}\ar[u]^{\Delta}\ar[ur]_{\cong} &
}\]
are commutative.
Coalgebras will generally \emph{not} be assumed to have a counit.
If~$C$ and $D$ are coalgebras, then a coalgebra morphism from $C$ to $D$
is a linear map~$\varphi\colon C\to D$ which respects the coproduct and
counit,
that is, the equations 
$(\varphi\tensor\varphi)\circ\Delta_C=\Delta_D\circ\varphi$ and
$\varepsilon_D\circ\varphi = \varepsilon_C$ hold.

Just as taking a product can be thought of assembling something out
of smaller pieces, taking a coproduct can be thought of as disassembling
something into its constituent pieces.  Following this analogy, we
define a \define{piece} of $c$ to be any term $c_{(1)}$ or $c_{(2)}$
which appears in the expansion $\Delta(c) = \sweedle{}{c_{(1)}\tensor c_{(2)}}$.

We will generally suppress the notation $\nabla$ for product, 
writing $ab$ or $a\cdot b$ instead of $\nabla(a\tensor b)$.  
The sigma notation for coproducts was introduced by Heyneman and
Sweedler~\cite{Heyneman_Sweedler} and is now widely used.  We adopt
a variant of sigma notation, writing the coproduct of $c$ as
\[
  \Delta(c) = \sweedle{\Delta}{c_{(1)}\tensor c_{(2)}}.
\]
If the coproduct is understood, we will generally suppress $\Delta$, 
writing
\[
  \Delta(c) = \sweedle{}{c_{(1)}\tensor c_{(2)}}.
\]
Using sigma notation, the coassociativity condition can be written
as the equation
\[
  \sweedle{}{(c_{(1,1)}\tensor c_{(1,2)})\tensor c_{(2)}} =
  \sweedle{}{c_{(1)}\tensor (c_{(2,1)}\tensor c_{(2,2)})} =
  \sweedle{}{c_{(1)}\tensor c_{(2)}\tensor c_{(3)}}.
\]
while the counital condition can be written as the equation
\[
  c = \sweedle{}{\varepsilon(c_{(1)})c_{(2)}} = \sweedle{}{c_{(1)}\varepsilon(c_{(2)})}.
\]

A \define{bialgebra} is a module with compatible algebra and
coalgebra structure maps.
In other words, the algebra structure maps
are coalgebra morphisms, while the coalgebra structure maps
are algebra morphisms.
If $B$ is a bialgebra with product $\nabla$ and coproduct $\Delta$,
then $\Hom_k(B, B)$ is an algebra with the convolution
product, defined by $f*g = \nabla\circ(f\tensor g)\circ\Delta$.  Using sigma 
notation, the convolution of linear maps $f$ and $g$ is written
\[
  (f*g)(b) = \sweedle{\Delta} f(b_{(1)})g(b_{(2)}).
\]
Observe that the composition $\eta\circ\varepsilon\colon B\to B$
of the unit and counit (if there is one) is the identity under convolution.

A \define{Hopf algebra} is a bialgebra $H$ for which
the identity map $\id\colon H\to H$ has a convolution inverse $S\colon H\to H$,
that is, such that
\[
  (\eta\circ\varepsilon)(h) 
  = \sweedle{\Delta}{S(h_{(1)})h_{(2)}} 
  = \sweedle{\Delta}{h_{(1)}S(h_{(2)})} 
\]
for all $h$ in $H$.  The map $S$, which is always an antihomomorphism,
is called the \define{antipode} of $H$.

A \define{Newtonian coalgebra} is a module $N$ 
with both algebra and coalgebra structure maps such that 
the Leibniz condition 
\[
  \Delta(u\cdot v) = \Delta(u)\cdot v + u\cdot\Delta(v)
\]
holds for all $u$ and $v$.
In other words, the coproduct is 
a derivation over the product.  Newtonian coalgebras were
introduced by Joni and Rota~\cite{Joni_Rota}, who called them infinitesimal 
coalgebras.  A Newtonian coalgebra can have a unit or a
counit, but not both.

Now we indicate the algebras of interest and briefly describe each.

\begin{itemize}
\item
$\NG$, the Newtonian coalgebra of graded posets;

\item
$\NA$, the Newtonian coalgebra of $\ab$-polynomials;

\item
$\HG$, the Hopf algebra of graded posets; and

\item
$\HA$, the nonassociative bialgebra of $\ab$-polynomials.
\end{itemize}


\subsection{The Newtonian coalgebra of graded posets}
Let $\NG=k\G$ be the free module generated by $\G$.
The \define{star product} of two posets $P$, $Q$
in $\G$, denoted by $P\star Q$, is the poset with ground set
$(P\setminus\{\hz_P\})\cup(Q\setminus\{\ho_Q\})$ and order
relation
\[
x \le_{P\star Q} y
\text{\ if and only if\ }
\begin{cases}
x \le_P y \\
x \le_Q y \\
x \in P \text{\ and\ } y\in Q.
\end{cases}
\]
The star product $\star$ makes $\NG$ into an algebra
with the Boolean algebra on a one-element set as the unit.
Ehrenborg and Hetyei showed in unpublished work that $\NG$
is a Newtonian coalgebra.  The coproduct of a poset $P$ is
defined by the formula
\[
  \Delta(P) = \sum_{\hz < x < \ho} [\hz, x]\tensor[x,\ho].
\]
It is straightforward to verify the Leibniz condition:
\[
  \Delta(P \star Q) = \Delta(P) \star Q + P \star \Delta(Q).
\]
Since $\Delta$ is a derivation over the unital product $\star$,
there is no counit.

\subsection{The Newtonian coalgebra of ab-polynomials}
The noncommutative polynomial algebra $\NA=k\fab$ also has
the structure of a Newtonian coalgebra.  The coproduct is 
defined on a monomial $u_1 \cdots u_n$ by the formula
\[
  \Delta(u_1 \cdots u_n) = \sum_{i=1}^n u_1\cdots u_{i-1} \tensor u_{i+1}\cdots u_n.
\]
The $\ab$-index $\Psi(P)$ of a graded poset $P$ is an invariant
of the poset.  Ehrenborg and Readdy~\cite{Ehrenborg_Readdy_c} showed 
that $\Psi$ can be viewed as a morphism $\Psi\colon \NG\to\NA$ of
Newtonian coalgebras.  Moreover, $\Psi$ is surjective.  

Stanley~\cite{Stanley_d}  developed a recursive formula for the $\ab$-index of 
a poset which is amenable to computation and best expressed
using coalgebraic notation.  
Define an algebra endomorphism $\kappa$ on $\NA$ by setting
$\kappa(\av) = \av - \bv$ and $\kappa(\bv) = 0$.
Stanley proved that the $\ab$-index satisfies the recursive formula
\begin{align*}
  \Psi(P) 
  &= \kappa(\Psi(P)) + \sweedle{}{\kappa(\Psi(P_{(1)}))\cdot\bv\cdot\Psi(P_{(2)})} \\
  &= \kappa(\Psi(P)) + \sweedle{}{\Psi(P_{(1)})\cdot\bv\cdot\kappa(\Psi(P_{(2)}))}.
\end{align*}
Applying the surjectivity of $\Psi$, the same
recursive formula holds for every $\ab$-polynomial:
\[
  u 
  = \kappa(u) + \sweedle{}{\kappa(u_{(1)})\cdot\bv\cdot u_{(2)}}
  = \kappa(u) + \sweedle{}{u_{(1)}\cdot\bv\cdot\kappa(u_{(2)})}
\]
This can also be proved inductively for $\ab$-polynomials, or be
viewed as a consequence of the Ehrenborg--Readdy theorem that $\Psi$
is a morphism of Newtonian coalgebras.  In any case, the $\kappa$
morphism is fundamental for the study of the $\ab$-index.

The map $\kappa$ preserves $\av$ and kills $\bv$.  In a similar way
we can define a map $\lambda$ which preserves $\bv$ and kills $\av$.
Let $\swapm{\,\cdot\,}\colon k\fab\to k\fab$ denote the map which swaps
$\av$ and~$\bv$.  Define an algebra endomorphism $\lambda$ on $k\fab$
by $\lambda(u) = \swapm{\kappa(\swapm{u})}$.  Then for any 
$\ab$-polynomial $u$, 
\[
  u 
  = \lambda(u) + \sweedle{}{\lambda(u_{(1)})\cdot\av\cdot u_{(2)}}
  = \lambda(u) + \sweedle{}{u_{(1)}\cdot\av\cdot\lambda(u_{(2)})}.
\]
Note that
the maps $\kappa$ and $\lambda$ act as near-counits in $\NA$.


The Newtonian coalgebra $\NA$ has an important Newtonian subcoalgebra $\NC=k\fcd$,
which is generated by the monomials $\cv = \av + \bv$ and $\dv = \ab + \ba$.
If a poset is Eulerian, its $\ab$-index lives in the subcoalgebra $\NC$.
In general, if $\Psi(P)$ is in $k\fcd$, then we say that $P$ \define{has a $\cd$-index}.
The existence of the $\cd$-index was conjectured by Fine.
Bayer and Klapper~\cite{Bayer_Klapper} showed that a poset has a $\cd$-index
if and only if it satisfies the generalized Dehn-Sommerville relations, while
Stanley~\cite{Stanley_d} provided
an alternative proof for Eulerian posets and established that the $\cd$-index
of a polytope has nonnegative coefficients.
Several proofs of the 
existence of 
the $\cd$-index
have been given~\cite{Bayer_Klapper},\cite{Ehrenborg_k-Eulerian},\cite{Ehrenborg_Readdy_homology},\cite{Stanley_d}.

\subsection{The Hopf algebra of graded posets}
We will also need to make use of the Hopf algebra structure on graded
posets.  
Let $\overline{\G} = \G\cup\{\bullet\}$, where $\bullet$ is the one-point poset.
Then the module $\HG = k\overline{\G}$ has the structure of a Hopf algebra, with
product coming from the Cartesian product and coproduct defined by
\[
  \Delta^*(P) = \sweedle{\Delta^*}P_{(1)}\tensor P_{(2)}
              = \sum_{x\in P} [\hz, x]\tensor [x,\ho].
\]
Schmitt~\cite{Schmitt} derived an explicit formula for the antipode.
Ehrenborg showed~\cite{Ehrenborg_Hopf} that the antipode plays the role of the 
M\"obius function,
since if we define $\varphi\colon \HG\to k$ by $\varphi(P) = 1$ for each poset $P$,
then $\mu(P) = \varphi(S(P))$.

\subsection{The nonassociative bialgebra of ab-polynomials}
In a similar way, we can extend the Newtonian coalgebra $\NA$ to a
nonassociative bialgebra $\HA=k\fab\oplus k\xi$ via the
formulas
\[
  \av\xi = \bv\xi = \xi\av = \xi\bv = 1 \text{\quad and\quad} \xi^2 = 0.
\]
Note that $\xi$ does not usually associate, so one must exercise care with
its use.  If~$\xi$ is flanked by two copies of $\av$ or $\bv$, then it does
associate, yielding the identities $\av\xi\av = \av$ and $\bv\xi\bv=\bv$.
However, $(\dv\xi)\dv = \cv\dv$ while $\dv(\xi\dv) = \dv\cv$.
This bialgebra was introduced by Ehrenborg and Fox~\cite{Ehrenborg_Fox}. 

The Stanley recursion for the $\ab$-index may be expressed
more briefly in this bialgebra:
\begin{align*}
u
&= \sweedle{\Delta^*}{\kappa(u_{(1)})\cdot\bv\cdot u_{(2)}}
 = \sweedle{\Delta^*}{u_{(1)}\cdot\bv\cdot\kappa(u_{(2)})} \\
&= \sweedle{\Delta^*}{\lambda(u_{(1)})\cdot\av\cdot u_{(2)}}
 = \sweedle{\Delta^*}{u_{(1)}\cdot\av\cdot\lambda(u_{(2)})} \\
\end{align*}
Observe that $\kappa$ and $\lambda$ act as near-counits in $\HA$.

Just as $\NA$ has a subcoalgebra $\NC$ of $\cd$-polynomials, 
$\HA$ has a sub-bialgebra $\HC$ of $\cd$-polynomials with $\xi$.

\section{Binary operations on posets}

Kalai~\cite{Kalai} constructed a
basis of polytopes obtained from simplices by repeatedly taking
joins and direct sums.  He showed that the face lattice of a join
of polytopes is the Cartesian product of the respective face lattices,
and the face lattice of a direct sum is the diamond product of the 
face lattices.
That is,
\begin{align*}
\lat{X\freejoin Y}  &= \lat{X} \times   \lat{Y} \\
\lat{X\times Y} &= \lat{X} \diamond \lat{Y},
\end{align*}
where $\freejoin$ denotes the join operation and $\diamond$
denotes the diamond product.
Recall that the \define{diamond product} (or lower truncated
product) of posets $P$ and $Q$ is defined by 
\[
  P \diamond Q = (P\setminus\{\hz\}) \times (Q\setminus\{\hz\}) \cup \{\hz\}.
\]
There is also a \define{dual diamond product} (or upper truncated product),
which we denote by $\diamond^*$:
\[
  P \diamond^* Q = (P\setminus\{\ho\}) \times (Q\setminus\{\ho\}) \cup \{\ho\}.
\]
The diamond product and dual diamond product are related by
the identity
\[
  P \diamond^* Q = (P^* \diamond Q^*)^*,
\]
where $P^*$ denotes the dual of the poset $P$.

The geometric operations of pyramid and prism arise from $\times$ and 
$\diamond$ on the poset level, since
\[
  \lat{\Pyr(P)} = \lat{P}\times B_1 \text{\ and\ } \lat{\Pri(P)} = \lat{P}\diamond B_2,
\]
where $B_i$ denotes the Boolean algebra on $i$ elements.
Since the $\ab$-index encodes the flag $f$-vector, it is of interest
to study the effects of~$\times$ and $\diamond$ on the $\ab$-index.
Ehrenborg~\cite{Ehrenborg_Hopf} used quasisymmetric functions to show that $\Psi(P\times Q)$ is a 
function of $\Psi(P)$ and $\Psi(Q)$.  
Ehrenborg and Readdy~\cite{Ehrenborg_Readdy_c} derived
recursive formulas for $\Psi(P\times Q)$ which were improved by 
Ehrenborg and Fox~\cite{Ehrenborg_Fox}.

In preparation for the study of the manifold product, we present a
completely coalgebraic derivation of the recursive formulas for 
$\Psi(P\times Q)$ and $\Psi(P\diamond Q)$.
We need
two basic facts.  
First, we need the Stanley recursion
discussed above.  Second, we need to know the coproduct of a Cartesian
product of posets.  

Since the Cartesian product is the product in the Hopf algebra of graded
posets, 
\[
  \Delta^*(P\times Q) 
  = \Delta^*(P)\times\Delta^*(Q) 
  =\sweedle{\Delta^*}{ (P_{(1)}\times Q_{(1)})\tensor (P_{(2)}\times Q_{(2)}) }.
\]
Hence the coproduct of a Cartesian product is
\[
  \Delta^*(u \times v)
= \sweedle{\Delta^*}{(u_{(1)}\times v_{(1)}) \tensor (u_{(2)}\times v_{(2)})}.
\]
Using Stanley's recursion for the $\ab$-index, 
we obtain the following recursive formula for the mixing operator $\times$
applied to the $\ab$-polynomials $u$ and $v$:
\begin{align*}
u \times v
&= \sweedle{\Delta^*}{\kappa(u_{(1)}\times v_{(1)})\cdot\bv\cdot (u_{(2)}\times v_{(2)})} \\
&= \kappa(u\times v) + \kappa(u)\cdot\bv + \kappa(v)\cdot\bv\cdot u \\
&\phantom{= } + \sweedle{}{\kappa(u_{(1)})\cdot\bv\cdot(u_{(2)}\times v)} +
\sweedle{}{\kappa(v_{(1)})\cdot\bv\cdot(u\times v_{(2)})} \\
&\phantom{= } + \sweedle{}{\kappa(u\times v_{(1)})\cdot \bv\cdot v_{(2)}} +
\sweedle{}{\kappa(u_{(1)}\times u_{(2)}) \cdot\bv\cdot (u_{(2)}\times v_{(2)})}.
\end{align*}

\subsection{Computing the cd-index of a Cartesian product}

For any graded poset $P$, the coefficient of the pure $\av$ term
is always $1$.  Hence $\kappa(P)$ depends only on the rank of $P$,
that is, $\kappa(P) = (\av - \bv)^{\rho(P) - 1}$.
If $P$ and $Q$ are graded posets, their Cartesian product has rank 
$\rho(P) + \rho(Q) + 1$.  So
\begin{align*}
\kappa(\Psi(P\times Q)) &= \kappa(\Psi(P)\cdot\av\cdot\Psi(Q)) \\
\end{align*}
Hence for
any $\ab$-polynomials $u$ and $v$,
\begin{align*}
\kappa(u\times v)   &= \kappa(u)\cdot\kappa(v)\cdot(\av-\bv) \\
\end{align*}
Analogously,
\begin{align*}
\lambda(u\times v)   &= \lambda(u)\cdot\lambda(v)\cdot(\bv-\av) \\
\end{align*}
We use these facts to prove the following lemma.  

\begin{lemma}[Ehrenborg--Readdy~{\cite[Proposition 4.2]{Ehrenborg_Readdy_c}}]\label{lemma:pyramid}
For any $\ab$-polynomial~$u$,
\begin{align}
  u\times 1
&= \sweedle{\Delta^*}{u_{(1)}\cdot\ba\cdot u_{(2)}}
 = \av\cdot u + u\cdot\bv + \sweedle{}{u_{(1)}\cdot\ba\cdot u_{(2)}}\label{eqn:up1ba} \\
&= \sweedle{\Delta^*}{u_{(1)}\cdot\ab\cdot u_{(2)}}
 = \bv\cdot u + u\cdot\av + \sweedle{}{u_{(1)}\cdot\ab\cdot u_{(2)}}.\label{eqn:up1ab}
\end{align}
Since the formula for $u\times 1$ is invariant under the action of the 
involution $\swapm{\,\cdot\,}$ which swaps $\av$ and $\bv$, if $u$ is 
a $\cd$-polynomial,
then so is $u\times 1$.
\end{lemma}

\begin{proof}
Since $1$ is the $\ab$-index of the Boolean algebra $B_1$, the expression
$1\times 1$ is the $\ab$-index of the product $B_1\times B_1 = B_2$,
that is, $1\times 1 = \Psi(B_2) = \cv$.
Equations~(\ref{eqn:up1ba}) and~(\ref{eqn:up1ab}) both hold when $u = 1$.

To complete the proof, assume for induction
that Equation~(\ref{eqn:up1ba}) holds for all pieces of $u$,
that is, for any polynomial $u_{(1)}$ or $u_{(2)}$ appearing in the
coproduct of $u$.
Since $\Delta^*(1) = 1\tensor\xi + \xi\tensor 1$, 
\begin{align*}
u\times 1
&=
   \sweedle{\Delta^*}{\kappa(u_{(1)}\times 1)\cdot\bv\cdot u_{(2)}}
 + \sweedle{\Delta^*}{\kappa(u_{(1)})\cdot\bv\cdot (u_{(2)}\times 1)} \\
&= 
   \bv\cdot u + \kappa(u\times 1) + \kappa(u)\cdot\bv \\
&\phantom{=}
 + \sweedle{\Delta}{(\av-\bv)\cdot\kappa(u_{(1)})\cdot\bv\cdot u_{(2)}}
 + \sweedle{\Delta}{\kappa(u_{(1)})\cdot\bv\cdot u_{(2)}\cdot\bv} \\
&\phantom{=}
 + \sweedle{\Delta}{\kappa(u_{(1)})\cdot\ba\cdot u_{(2)}}
 + \sweedle{\Delta}{\kappa(u_{(1)})\cdot\bv\cdot u_{(2)}\cdot\ba\cdot u_{(3)}} \\
\end{align*}
Use the identity
$\kappa(u\times 1) = \kappa(u)\cdot (\av - \bv)$ to combine two of
the isolated terms, and apply the induction hypothesis to expand
the second summation.  The Stanley recursion permits the terms above to
be expressed in a much simpler way.
\begin{align*}
u\times 1
&= \bv\cdot u + \kappa(u)\cdot\av \\
&\phantom{=}
 + (\av - \bv)\cdot(u - \kappa(u)) + (u - \kappa(u))\cdot\bv \\
&\phantom{=}
 + \sweedle{\Delta}{u_{(1)}\cdot\ba\cdot u_{(2)}} \\
&= \av\cdot u + u\cdot\bv + \sweedle{\Delta}{u_{(1)}\cdot\ba\cdot u_{(2)}}.
\end{align*}

Equation~(\ref{eqn:up1ab}) could be proved by imitating the one just given,
replacing $\kappa$ with $\lambda$ and making other appropriate changes.
However, it is more direct to apply the fact that the star involution 
is a Newtonian coalgebra anti-isomorphism.  Hence
\begin{align*}
u\times 1
&= (u^* \times 1^*)^* \\
&= 
 \left[ \av\cdot u^* + u^*\cdot\bv 
 + \sweedle{\Delta}{u_{(2)}^*\cdot\ba\cdot u_{(1)}^*} \right]^* \\
&= 
 \bv\cdot u + u\cdot\av + \sweedle{\Delta}{u_{(1)}\cdot\ab\cdot u_{(2)}},
\end{align*}
which completes the proof.
\end{proof}


\begin{lemma}[Ehrenborg--Fox~{\cite[Proposition 5.8]{Ehrenborg_Fox}}]\label{lemma:timesabrecursion}
For any $\ab$-polynomials $u$ and $v$, the identities
\begin{align}
  u\times (v\cdot\av) 
&= \sweedle{\Delta^*}{(u_{(1)}\times v)\cdot\ab\cdot u_{(2)}} \notag \\
&= v\cdot\ab\cdot u + (u\times v)\cdot\av + \sweedle{}{(u_{(1)}\times v)\cdot\ab\cdot u_{(2)}} \label{eqn:upva} \\
  u\times (v\cdot\bv) 
&= \sweedle{\Delta^*}{(u_{(1)}\times v)\cdot\ba\cdot u_{(2)}} \notag \\
&= v\cdot\ba\cdot u + (u\times v)\cdot\bv + \sweedle{}{(u_{(1)}\times v)\cdot\ba\cdot u_{(2)}} \label{eqn:upvb}
\end{align}
hold.
\end{lemma}

\begin{proof}
The proof is a double induction on the lengths of $u$ and $v$.
By explicitly constructing appropriate posets, one can 
compute that
\[
  1\times\av = \cv^2 - \bv^2 
  \text{\quad and\quad}
  1\times\bv = (1\times\cv) -(1\times\av)
  = (\cv^2+\dv) - (\cv^2 - \bv^2) = \cv^2 - \av^2.
\]
Thus Equations~(\ref{eqn:upva}) and~(\ref{eqn:upvb}) are both
satisfied if $u = v = 1$.

Now assume for induction that Equation~(\ref{eqn:upvb})
holds for $v = 1$ and any piece of $u$.
Expand $u\times\bv$ via the general recursion for products,
keeping in mind that $\kappa(w\times\bv) = 0$ for any $w$.
\begin{align*}
u\times\bv
&= 
   \sweedle{\Delta^*}{\kappa(u_{(1)}\times 1)\cdot\bv\cdot(u_{(2)}\times 1)}
 + \sweedle{\Delta^*}{\kappa(u_{(1)})\cdot\bv\cdot(u_{(2)}\times\bv)}.
\end{align*}
Apply Lemma~\ref{lemma:pyramid} to the first summation and the induction
hypothesis to the second summation.
\begin{align*}
u\times\bv
&= 
   \sweedle{\Delta^*}{\kappa(u_{(1)}\times 1)\cdot\bv\cdot(u_{(2)}\times \xi)\cdot\ba\cdot u_{(3)}} \\
&\phantom{=}
 + \sweedle{\Delta^*}{\kappa(u_{(1)}\times\xi)\cdot\bv\cdot(u_{(2)}\times 1)\cdot\ba\cdot u_{(3)}}.
\end{align*}
The part of the above expression preceding $\ba$ is recognizable as
an expansion of the product $u_{(1)}\times 1$.
\[
u\times\bv
= 
   \sweedle{\Delta^*}{(u_{(1)}\times 1)\cdot\ba\cdot u_{(2)}},
\]
which is what needed to be shown.

To complete the double induction, assume that Equation~(\ref{eqn:upvb})
holds for any piece of $u$ or $v$.  Since 
$\Delta^*(v\cdot\bv) = \Delta^*(v)\cdot\bv + v\cdot\bv\tensor\xi$,
\begin{align*}
u\times(v\cdot\bv)
&= 
   \sweedle{\Delta^*}{\kappa(u_{(1)}\times v_{(1)})\cdot\bv\cdot(u_{(2)}\times (v_{(2)}\cdot\bv))} \\
&\phantom{=}
 + \sweedle{\Delta^*}{\kappa(u_{(1)}\times (v\cdot\bv))\cdot\bv\cdot(u_{(2)}\times\xi)}.
\end{align*}
The second summation vanishes because $\kappa$ kills $\bv$.  Apply the
induction hypothesis to expand the first summation.  As in the case $v=1$,
this results in a recognizable expansion of a product.
No parentheses are needed below because $u_{(2)}\times v_{(2)}$, 
the only expression which could be $\xi$, is flanked by copies of~$\bv$.
\begin{align*}
u\times (v\cdot\bv)
&=
  \sweedle{\Delta^*}{\kappa(u_{(1)}\times v_{(1)})\cdot\bv\cdot(u_{(2)}\times v_{(2)})\cdot\ba\cdot u_{(3)}} \\
&= 
  \sweedle{\Delta^*}{(u_{(1)}\times v)\cdot\ba\cdot u_{(2)}}.
\end{align*}
This completes the proof of Equation~(\ref{eqn:upvb}).

Equation~(\ref{eqn:upva})
can be proved in a similar way, replacing
$\kappa$ with $\lambda$ and making other appropriate changes.
\end{proof}

In the previous lemmas identities appeared in pairs differing only by
the action of the involution $\swapm{\,\cdot\,}$.  This suggests that $\times$
respects the action of $\swapm{\,\cdot\,}$.  This is a consequence of the
identities proved in Lemma~\ref{lemma:timesabrecursion}, but it is more
fundamentally a consequence of the existence of the paired recursive
formulas
\[
  u = \sweedle{\Delta^*}{\kappa(u_{(1)})\cdot\bv\cdot u_{(2)}}
    = \sweedle{\Delta^*}{\lambda(u_{(1)})\cdot\av\cdot u_{(2)}}.
\]
Ehrenborg and Fox proved that $\times$ respects the involution $\swapm{\,\cdot\,}$.
We offer the following alternative proof.

\begin{proposition}[Ehrenborg--Fox~{\cite[Lemma 5.5]{Ehrenborg_Fox}}]\label{prop:timesrespectsswapm}
For any $\ab$-polynomials $u$ and $v$, the identity
\[
  \swapm{u\times v} = \swapm{u}\times\swapm{v}
\]
holds.
\end{proposition}

\begin{proof}
If $u = v = 1$, there is nothing to prove.  Suppose the claim holds for
pieces of~$u$ and $v$.  By the recursive formula for the product $u\times v$,
\begin{align*}
\swapm{u\times v}
&= \sweedle{\Delta^*}{\swapm{\kappa(u_{(1)}\times v_{(1)})}\cdot\av\cdot\swapm{u_{(2)}\times v_{(2)}}} \\
&= \sweedle{\Delta^*}{\lambda(\swapm{u_{(1)}\times v_{(1)}})\cdot\av\cdot\swapm{u_{(2)}\times v_{(2)}}}.
\end{align*}
Now apply the induction hypothesis and the fact that $\swapm{\,\cdot\,}$ is a
coalgebra morphism.
\begin{align*}
\swapm{u\times v}
&= \sweedle{\Delta^*}{\lambda(\swapm{u_{(1)}}\times\swapm{v_{(1)}})\cdot\av\cdot(\swapm{u_{(2)}}\times\swapm{v_{(2)}})} \\
&= \swapm{u}\times\swapm{v}.
\end{align*}
This completes the proof.
\end{proof}

\begin{corollary}[Ehrenborg--Fox~{\cite[Theorem 5.1]{Ehrenborg_Fox}}]\label{cor:timescdrecursion}
For any $\cd$-polynomials $u$ and~$v$, the identities
\begin{align*}
  u\times (v\cdot\cv) 
&= \sweedle{\Delta^*}{(u_{(1)}\times v)\cdot\dv\cdot u_{(2)}} \\
  u\times (v\cdot\dv) 
&= \sweedle{\Delta^*}{(u_{(1)}\times v)\cdot\dv\cdot \Pyr(u_{(2)})}
\end{align*}
hold.
\end{corollary}

\begin{proof}
Expand $\cv$ and $\dv$, then apply
Lemma~\ref{lemma:timesabrecursion}.  Thus
\begin{align*}
u\times (v\cdot\cv) 
&= u\times (v\cdot\av+\bv) \\
&= \sweedle{\Delta^*}{ (u_{(1)}\times v)\cdot(\ab + \ba)\cdot u_{(2)}} \\
&= \sweedle{\Delta^*}{(u_{(1)}\times v)\cdot\dv\cdot u_{(2)}}.
\end{align*}
To compute $u\times (v\cdot\dv)$, the lemma must be invoked twice.
We have
\begin{align*}
u\times (v\cdot\ab)
&= \sweedle{\Delta^*}{(u_{(1)}\times(v\cdot\av))\cdot\ba\cdot u_{(2)}} \\
&= \sweedle{\Delta^*}{(u_{(1)}\times v)\cdot\ab\cdot u_{(2)}\cdot\ba\cdot u_{(3)}}.
\end{align*}
By Lemma~\ref{lemma:pyramid}, we can collapse $u_{(2)}\cdot\ba\cdot u_{(3)}$ 
into $\Pyr(u_{(2)})$.  Similarly,
\[
  u \times (v\cdot\ba) = \sweedle{\Delta^*}{(u_{(1)}\times v)\cdot\ba\cdot\Pyr(u_{(2)})},
\]
from which the recursive formula for $u\times (v\cdot\dv)$ follows.
\end{proof}


\subsection{Computing the cd-index of a diamond product}

Just as with the Cartesian product, the algebra maps $\kappa$ and
$\lambda$ interact nicely with the $\ab$-index of a diamond
product of posets.
If $P$ is a graded poset, then $\kappa(P)$ is given by
$\kappa(\Psi(P)) = (\av - \bv)^{\rho(P)-1}$.
If $P$ and $Q$ are graded posets, 
their diamond product has rank 
$\rho(P) + \rho(Q)$.  Thus
\begin{align*}
\kappa(\Psi(P\diamond Q)) &= \kappa(\Psi(P)\cdot\Psi(Q)) \\
\end{align*}
Hence for
any $\ab$-polynomials $u$ and $v$,
\begin{align*}
\kappa(u\diamond v) &= \kappa(u)\cdot\kappa(v).
\end{align*}
Analogously,
\begin{align*}
\lambda(u\diamond v) &= \lambda(u)\cdot\lambda(v).
\end{align*}
These formulas describe the situation in the algebra $\NA$.  For
simplicity, we require that the above formulas hold in $\HA$, even if $u$ or $v$
is $\xi$, subject to the constraint that~$\kappa(\xi) = 0$.
In particular, $\kappa(u\diamond\xi) = 0$ for any
$u$, which implies that 
$  u \diamond \xi = 0$
for any $u$.  This may conflict with the intuition that
\[
  P \diamond \bullet = (P \setminus\{\hz\})\times \emptyset\cup \{\hz\} = \emptyset\cup\{\hz\} = \bullet,
\]
but has the advantage of maintaining homogeneity of degree in the recursive
formulas that follow.  Since an $\ab$-index of a poset is always homogeneous
in degree, we accept failure of intuition in exchange for correctness of
formulas.

We summarize the basic properties of the diamond product with the following
result from Ehrenborg and Fox.

\begin{proposition}[Ehrenborg--Fox~{\cite[Corollary 6.3]{Ehrenborg_Fox}}]\label{cor:diamondfacts}
The diamond product $\diamond$ makes $\NA$ into an abelian monoid with unit $1$
and makes $\HA$ into a commutative semigroup satisfying the rules
\begin{align*}
u\diamond 1  &= u \text{\ for any $u$ in $\NA$} \\
u\diamond\xi &= 0 \text{\ for any $u$ in $\HA$}.
\end{align*}
\end{proposition}
\noindent
The diamond product obeys the coalgebraic recursive formula
\[
  u \diamond v 
= \sweedle{\Delta^*}{\kappa(u_{(1)}\diamond v_{(1)})\cdot\bv\cdot(u_{(2)}\times v_{(2)})}
\]
as well as the analogous formulas obtained by moving $\kappa$ or 
replacing $\kappa$ and $\bv$ with~$\lambda$ and $\av$.

The following lemma is the diamond version of Lemma~\ref{lemma:timesabrecursion}.

\begin{lemma}\label{lemma:diamondabrecursion}
For any $\ab$-polynomials $u$ and $v$, the identities
\begin{align}
  u\diamond (v\cdot\av) 
&= \sweedle{\Delta^*}{(u_{(1)}\diamond v)\cdot\ab\cdot u_{(2)}} \notag \\
&= (u\diamond v)\cdot\av + \sweedle{}{(u_{(1)}\diamond v)\cdot\ab\cdot u_{(2)}}\label{eqn:udva} \\
  u\diamond (v\cdot\bv) 
&= \sweedle{\Delta^*}{(u_{(1)}\diamond v)\cdot\ba\cdot u_{(2)}} \notag \\
&= (u\diamond v)\cdot\bv + \sweedle{}{(u_{(1)}\diamond v)\cdot\ba\cdot u_{(2)}}\label{eqn:udvb}
\end{align}
hold.
\end{lemma}

\begin{proof}
This lemma is essentially a corollary of Lemma~\ref{lemma:timesabrecursion}.
Here we demonstrate Equation~(\ref{eqn:udvb}).  Since 
$\Delta^*(v\cdot\bv) = \Delta^*(v)\cdot\bv + v\cdot\bv\tensor\xi$,
\begin{align*}
u\diamond (v\cdot\bv)
&=
   \sweedle{\Delta^*}{\kappa(u_{(1)}\diamond v_{(1)})\cdot\bv\cdot(u_{(2)}\times (v_{(2)}\cdot\bv))} \\
&\phantom{=}
 + \sweedle{\Delta^*}{\kappa(u_{(1)}\diamond (v\cdot\bv))\cdot\bv\cdot(u_{(2)}\times \xi)}.
\end{align*}
Expand the first summation using the recursion for $\times$, and notice
that the second summation vanishes.  Finally, recognize the left factor
of the expression as an expansion of the diamond product.
\begin{align*}
u\diamond v
&=
   \sweedle{\Delta^*}{\kappa(u_{(1)}\diamond v_{(1)})\cdot\bv\cdot(u_{(2)}\times v_{(2)})\cdot\ba\cdot u_{(3)}} \\
&= \sweedle{\Delta^*}{(u_{(1)}\diamond v)\cdot\ba\cdot u_{(2)}}.
\end{align*}
The proof of Equation~(\ref{eqn:udva}) is similar.
\end{proof}

The diamond product also respects the involution $\swapm{\,\cdot\,}$.  Combining
the recursive formulas for $u\diamond(v\cdot\av)$ and $u\diamond(v\cdot\bv)$
produces recursive formulas for $u\diamond(v\cdot\cv)$ and~$u\diamond (v\cdot\dv)$.

\begin{corollary}[Ehrenborg--Fox~{\cite[Theorem 7.1]{Ehrenborg_Fox}}]\label{cor:diamondcdrecursion}
For any $\cd$-polynomials $u$ and~$v$, the identities
\begin{align}
  u\diamond (v\cdot\cv) 
&= \sweedle{\Delta^*}{(u_{(1)}\diamond v)\cdot\dv\cdot u_{(2)}} \notag \\
&= (u\diamond v)\cdot\cv + \sweedle{\Delta}{(u_{(1)}\diamond v)\cdot\dv\cdot u_{(2)}} \label{eqn:udvc} \\
  u\diamond(v\cdot\dv) 
&= \sweedle{\Delta^*}{(u_{(1)}\diamond v)\cdot\dv\cdot \Pyr(u_{(2)})} \notag \\
&= (u\diamond v)\cdot\dv + \sweedle{\Delta}{(u_{(1)}\diamond v)\cdot\dv\cdot \Pyr(u_{(2)})} \label{eqn:udvd}
\end{align}
hold.
\end{corollary}

%
%


\section{Lattice-path interpretation of mixing operators}

Equation~(\ref{eqn:udvc}) can be used to give an explicit recursive 
formula for $\cv^p\diamond\cv^q$.  In this section we display this
formula and show how to interpret its coefficients as counting 
weighted lattice paths.

First we define the algebra of lattice paths.
Consider the noncommutative polynomial algebra on the generators $\Dv$, 
$\rv$, and $\uv$, where 
$\Dv$ has degree~2 and
$\rv$ and $\uv$ have degree~1.  The generators correspond 
to the steps
\begin{align*}
\text{\textbf{D}iagonal} &= (1,1) \\
\text{\textbf{R}ight}    &= (1,0) \\
\text{\textbf{U}p}       &= (0,1).
\end{align*}
This algebra admits a bigrading into homogeneous parts indexed by $p$ and $q$
and generated
by monomials with $p$ occurrences of $\Dv$ or $\rv$ and $q$ occurrences
of $\Dv$ or $\uv$.
Note that $\Dv$, which represents a diagonal step,
counts toward both $p$ and $q$.  
The~$(p,q)$ summand of this algebra represents
lattice paths in $\mathbb{N}\times\mathbb{N}$ from the origin 
to $(p,q)$ which use only $\Dv$, $\rv$, and $\uv$ steps.

To avoid overcounting in what follows, we need to restrict to a 
submodule.  Let~$\Lambda$ denote the submodule generated by monomials
which do not contain $\ur$ as a contiguous subword.
It inherits a grading $\Lambda = \bigoplus_{p,q} \Lambda_{p,q}$
from the grading of the polynomial algebra.

\begin{example}
By direct computation, one can verify that
\begin{align*}
\cv^3\diamond \cv^2
&= \cv^5 + 2\cv^3\dv+ 4\cv^2\dv\cv+ 4\cv\dv\cv^2 + 2\dv\cv^3 + 
   4\cv\dv^2 + 4\dv\cv\dv + 4\dv^2\cv.
\end{align*}
Compare this polynomial with Figure~\ref{fig:pathsample},
which displays each $\Dv\rv\uv$-word in $\Lambda_{3,2}$
together with its associated path.
The coefficients of the terms in $\cv^3\diamond\cv^2$ can be obtained
by weighting $\rv$ and $\uv$ steps by $\cv$ and weighting $\Dv$ steps
by $2\dv$.  Note that the pair of terms $\rv\rv\Dv\uv$ and $\rv\uv\Dv\rv$
contribute to the same term of $\cv^3\diamond\cv^2$, as do the
pair of terms $\rv\Dv\rv\uv$ and $\uv\Dv\rv\rv$.  Hence $\cv^3\diamond\cv^2$
has only eight terms, even though there are ten $\Dv\rv\uv$-words in 
$\Lambda_{3,2}$.
\end{example}

\begin{figure}
\begin{center}
\begin{picture}(4,4)(0.5,1.5)
\thicklines
\put(0,0){\line(1,0){1}}
\put(1,0){\line(1,0){1}}
\put(2,0){\line(1,0){1}}
\put(3,0){\line(0,1){1}}
\put(3,1){\line(0,1){1}}
\linethickness{0.0625mm}
\put(0,0){\line(1,0){3}}
\put(3,0){\line(0,1){2}}
\put(3,2){\line(-1,0){3}}
\put(0,2){\line(0,-1){2}}
\linethickness{0.03125mm}
\put(1,0){\line(0,1){2}}
\put(2,0){\line(0,1){2}}
\put(0,1){\line(1,0){3}}
\put(0,-1){$\mathbf{RRRUU}$}
\end{picture}
\begin{picture}(4,4)(0.5,1.5)
\thicklines
\put(0,0){\line(1,0){1}}
\put(1,0){\line(1,0){1}}
\put(2,0){\line(0,1){1}}
\put(2,1){\line(1,1){1}}
\linethickness{0.0625mm}
\put(0,0){\line(1,0){3}}
\put(3,0){\line(0,1){2}}
\put(3,2){\line(-1,0){3}}
\put(0,2){\line(0,-1){2}}
\linethickness{0.03125mm}
\put(1,0){\line(0,1){2}}
\put(2,0){\line(0,1){2}}
\put(0,1){\line(1,0){3}}
\put(0,-1){$\mathbf{RRUD}$}
\end{picture}
\begin{picture}(4,4)(0.5,1.5)
\thicklines
\put(0,0){\line(1,0){1}}
\put(1,0){\line(1,0){1}}
\put(2,0){\line(1,1){1}}
\put(3,1){\line(0,1){1}}
\linethickness{0.0625mm}
\put(0,0){\line(1,0){3}}
\put(3,0){\line(0,1){2}}
\put(3,2){\line(-1,0){3}}
\put(0,2){\line(0,-1){2}}
\linethickness{0.03125mm}
\put(1,0){\line(0,1){2}}
\put(2,0){\line(0,1){2}}
\put(0,1){\line(1,0){3}}
\put(0,-1){$\mathbf{RRDU}$}
\end{picture}
\begin{picture}(4,4)(0.5,1.5)
\thicklines
\put(0,0){\line(1,0){1}}
\put(1,0){\line(0,1){1}}
\put(1,1){\line(1,1){1}}
\put(2,2){\line(1,0){1}}
\linethickness{0.0625mm}
\put(0,0){\line(1,0){3}}
\put(3,0){\line(0,1){2}}
\put(3,2){\line(-1,0){3}}
\put(0,2){\line(0,-1){2}}
\linethickness{0.03125mm}
\put(1,0){\line(0,1){2}}
\put(2,0){\line(0,1){2}}
\put(0,1){\line(1,0){3}}
\put(0,-1){$\mathbf{RUDR}$}
\end{picture}
\begin{picture}(4,4)(0.5,1.5)
\thicklines
\put(0,0){\line(1,0){1}}
\put(1,0){\line(1,1){1}}
\put(2,1){\line(1,0){1}}
\put(3,1){\line(0,1){1}}
\linethickness{0.0625mm}
\put(0,0){\line(1,0){3}}
\put(3,0){\line(0,1){2}}
\put(3,2){\line(-1,0){3}}
\put(0,2){\line(0,-1){2}}
\linethickness{0.03125mm}
\put(1,0){\line(0,1){2}}
\put(2,0){\line(0,1){2}}
\put(0,1){\line(1,0){3}}
\put(0,-1){$\mathbf{RDRU}$}
\end{picture} \newline
\begin{picture}(4,4)(0.5,1.5)
\thicklines
\put(0,0){\line(1,0){1}}
\put(1,0){\line(1,1){1}}
\put(2,1){\line(1,1){1}}
\linethickness{0.0625mm}
\put(0,0){\line(1,0){3}}
\put(3,0){\line(0,1){2}}
\put(3,2){\line(-1,0){3}}
\put(0,2){\line(0,-1){2}}
\linethickness{0.03125mm}
\put(1,0){\line(0,1){2}}
\put(2,0){\line(0,1){2}}
\put(0,1){\line(1,0){3}}
\put(0,-1){$\mathbf{RDD}$}
\end{picture}
\begin{picture}(4,4)(0.5,1.5)
\thicklines
\put(0,0){\line(0,1){1}}
\put(0,1){\line(1,1){1}}
\put(1,2){\line(1,0){1}}
\put(2,2){\line(1,0){1}}
\linethickness{0.0625mm}
\put(0,0){\line(1,0){3}}
\put(3,0){\line(0,1){2}}
\put(3,2){\line(-1,0){3}}
\put(0,2){\line(0,-1){2}}
\linethickness{0.03125mm}
\put(1,0){\line(0,1){2}}
\put(2,0){\line(0,1){2}}
\put(0,1){\line(1,0){3}}
\put(0,-1){$\mathbf{UDRR}$}
\end{picture}
\begin{picture}(4,4)(0.5,1.5)
\thicklines
\put(0,0){\line(1,1){1}}
\put(1,1){\line(1,0){1}}
\put(2,1){\line(1,0){1}}
\put(3,1){\line(0,1){1}}
\linethickness{0.0625mm}
\put(0,0){\line(1,0){3}}
\put(3,0){\line(0,1){2}}
\put(3,2){\line(-1,0){3}}
\put(0,2){\line(0,-1){2}}
\linethickness{0.03125mm}
\put(1,0){\line(0,1){2}}
\put(2,0){\line(0,1){2}}
\put(0,1){\line(1,0){3}}
\put(0,-1){$\mathbf{DRRU}$}
\end{picture}
\begin{picture}(4,4)(0.5,1.5)
\thicklines
\put(0,0){\line(1,1){1}}
\put(1,1){\line(1,0){1}}
\put(2,1){\line(1,1){1}}
\linethickness{0.0625mm}
\put(0,0){\line(1,0){3}}
\put(3,0){\line(0,1){2}}
\put(3,2){\line(-1,0){3}}
\put(0,2){\line(0,-1){2}}
\linethickness{0.03125mm}
\put(1,0){\line(0,1){2}}
\put(2,0){\line(0,1){2}}
\put(0,1){\line(1,0){3}}
\put(0,-1){$\mathbf{DRD}$}
\end{picture}
\begin{picture}(4,4)(0.5,1.5)
\thicklines
\put(0,0){\line(1,1){1}}
\put(1,1){\line(1,1){1}}
\put(2,2){\line(1,0){1}}
\linethickness{0.0625mm}
\put(0,0){\line(1,0){3}}
\put(3,0){\line(0,1){2}}
\put(3,2){\line(-1,0){3}}
\put(0,2){\line(0,-1){2}}
\linethickness{0.03125mm}
\put(1,0){\line(0,1){2}}
\put(2,0){\line(0,1){2}}
\put(0,1){\line(1,0){3}}
\put(0,-1){$\mathbf{DDR}$}
\end{picture}
\newline
\end{center}
\caption{Paths in $\Lambda_{3,2}$ correspond to terms of $\cv^3\diamond\cv^2$.}
\label{fig:pathsample}
\end{figure}
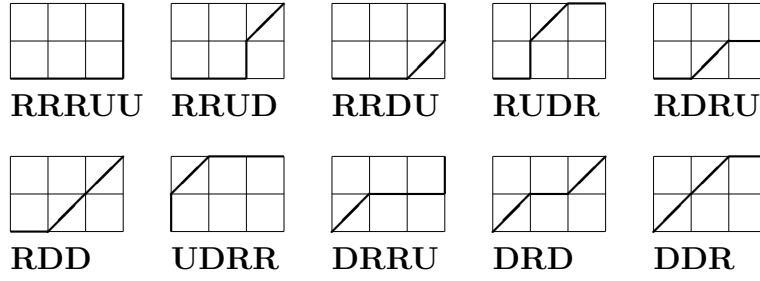

The following proposition shows that this situation is general.

\begin{proposition}
Let $\wt\colon\Lambda\to\NC$ be the linear map determined by
\[
  \wt(\Dv) = 2\dv\text{\quad and\quad} \wt(\rv) = \wt(\uv) = \cv.
\]
Then for any natural numbers $p$ and $q$, the $\cd$-index $\cv^p\diamond\cv^q$
is given by the formula
\[
  \cv^p\diamond\cv^q = \sum_{P\in\Lambda_{p,q}} \wt(P).
\]
\end{proposition}

\begin{proof}
The proof proceeds by induction on $p$ and $q$.  
For $p=q=0$ there is nothing to show.  Suppose the
weighted lattice path interpretation is correct for $(p', q')$ 
strictly smaller than $(p, q)$ in at least one coordinate.
As a consequence of Corollary~\ref{cor:diamondcdrecursion},
\begin{align*}
  \cv^p\diamond(\cv^{q-1}\cdot\cv)
  &=
  (\cv^p\diamond\cv^{q-1})\cdot\cv \\
  &\vphantom{=}
  +
  (\cv^{p-1}\diamond\cv^{q-1})\cdot 2\dv \\
  &\vphantom{=}
  +
  \sum_{k=0}^{p-2} (\cv^k\diamond\cv^{q-1})\cdot 2\dv\cdot \cv^{p-1-k}.
\end{align*}
Applying the induction assumption, the first summand is
\[
  \cv^p\diamond(\cv^{q-1}\cdot\cv) 
  = 
  \sum_{P\in\Lambda_{p,q-1}}\wt(P\cdot\uv),
\]
and the second summand is
\[
  (\cv^{p-1}\diamond\cv^{q-1})\cdot 2\dv 
  =
  \sum_{P\in\Lambda_{p-1,q-1}}\wt(P\cdot\Dv).
\]
The summation corresponds to lattice paths to which an $\rv$
can be appended, that is,
\[
  \sum_{k=0}^{p-2} (\cv^k\diamond\cv^{q-1})\cdot 2\dv\cdot \cv^{p-1-k}
  =
  \sum_{\substack{P\in\Lambda_{p-1,q} \\ \text{$P$ does not end in $\uv$}}} \wt(P\cdot\rv).
\]
But the module $\Lambda_{p,q}$ decomposes as
\begin{align*}
\Lambda_{p, q} 
&= 
   \{ P \cdot \Dv \colon P \in \Lambda_{p-1,q-1} \} \\
&\vphantom{=}
   \oplus
   \{ P \cdot \uv \colon P \in \Lambda_{p,q-1} \} \\
&\vphantom{=}
   \oplus
   \{ P \cdot \rv \colon P \in \Lambda_{p-1,q} \text{\ and $P$ does not end in $\uv$\ }\}.
\end{align*}
This completes the proof.
\end{proof}

As an application, we prove that the $\cd$-polynomial
$\cv^p\diamond\cv^q$ is always symmetric.

\begin{proposition}\label{prop:diamondlatpath}
For any natural numbers $p$ and $q$, 
\[
  (\cv^p\diamond\cv^q)^* = \cv^p\diamond\cv^q.
\]
\end{proposition}

\begin{proof}
We prove the claim by constructing an involution on the lattice 
paths in $\Lambda_{p,q}$.
Suppose $\alpha = \alpha_1\dots\alpha_n$ is a $\ur$-avoiding
path from $(0,0)$ to $(p,q)$.  Following the steps of $\alpha$
in reverse order yields the path $\alpha^* = \alpha_n\dots\alpha_1$.
Now, $\alpha^*$ is a path from $(0,0)$ to $(p,q)$, but it could
contain $\ur$ as a contiguous subword.  Adjust $\alpha^*$ to 
$\varphi(\alpha)$ by replacing each instance of $\uv^k\rv^{\ell}$
with $\rv^{\ell}\uv^k$.  In other words, we push in any ``bumps''
we find in the path.  The map $\alpha\mapsto\varphi(\alpha)$ is
an involution, and since $\uv$ and~$\rv$ have the same weight,
\[
  \wt(\varphi(\alpha)) = \wt(\alpha)^*.
\]
This completes the proof.
\end{proof}

Since we can interpret the coefficients of $\cv^p\diamond\cv^q$ 
as counting lattice paths, it is natural to ask whether we can
interpret the coefficients of $\cv^p\times\cv^q$ in a similar way.
First, recall the recursive formula for $\cv^p\times\cv^q$:
\[
  \cv^p\times\cv^q
    =
  (\cv^p \times \cv^{q-1})\cdot\cv
    +
  \cv^q\cdot\dv\cdot\cv^p 
    +
  \sum_{j+k=p-1}(\cv^j\times\cv^{q-1})\cdot 2\dv\cdot\cv^k.
\]
If the coefficients are to represent lattice paths in a straightforward
way, then it seems natural that the term $(\cv^p\times \cv^{q-1})\cdot\cv$
represents lattice paths which pass through $(p, q-1)$ and end in $\uv$,
so that they pass through $(p,q-1)$, while a term of the
form $(\cv^j\times\cv^{q-1})\cdot 2\dv\cdot\cv^k$ represents
lattice paths which pass through $(i, q-1)$ and end in~$\Dv\rv^k$.
But how are we to interpret the term $\cv^q\cdot\dv\cdot\cv^p$?  It 
seems to require a lattice path of the form $\uv^q\cdot\rv^p$, which
contains the forbidden subpath $\uv\rv$.

We can avoid forbidden subpaths by introducing another step $\sv = (0,0)$.
Thus $\sv$ represents standing still for a moment to avoid $\uv\rv$.  It
can also be thought of as marking a particular point on a lattice path.

Now we develop our argument more formally.
Consider the noncommutative polynomial algebra on the generators $\Dv$, 
$\rv$, $\uv$, and $\sv$, where 
$\Dv$ has degree~2 and the other generators have degree~1.
The generators correspond 
to the steps
\begin{align*}
\text{\textbf{D}iagonal} &= (1,1) \\
\text{\textbf{R}ight}    &= (1,0) \\
\text{\textbf{U}p}       &= (0,1) \\
\text{\textbf{S}tand}    &= (0,0).
\end{align*}
For natural numbers $p$ and $q$, let $\Lambda'_{p,q}$ be the 
module generated by monomials of degree $p+q+1$ with 
$p$ occurrences of $\Dv$ or $\rv$ and $q$ occurrences of $\Dv$ or $\uv$
which do not contain $\ur$ as a contiguous subword.
In this context, we can prove a proposition analogous to
Proposition~\ref{prop:diamondlatpath} for the Cartesian product.



We can prove that the diamond product is unimodal.

\begin{proposition}[Unimodality of diamond product]
The sequence
\[
  1\diamond\cv^{2n}, \cv\diamond\cv^{2n-1}, \dots, \cv^n\diamond\cv^n,
  \cv^{n+1}\diamond\cv^{n-1}, \dots, \cv^{2n}\diamond 1
\]
is unimodal.
\end{proposition}

%

\section{Concluding remarks}
\label{section_mixing_remarks}

In addition to the mixing operators studied above, there is also
the \define{manifold product} (or doubly-truncated product), 
denoted by $P\fp Q$ and defined by 
\[
  P \fp Q = (P\setminus\{\hz_P, \ho_P\})\times (Q\setminus\{\hz_Q,\ho_Q\})\cup\{\hz,\ho\}.
\]
The name comes from its relation with manifolds.  For example, if $P$ and $Q$ are
face lattices of polytopes, then $P\fp Q$ is the face poset of the torus 
which is the Cartesian product of the boundary complexes of the polytopes.

While the Cartesian product $\times$ increases degree by 1 and 
the diamond product $\diamond$ preserves degree,
the manifold product $\fp$ decreases degree by 1.  For any posets $P$ and $Q$ 
with rank at
least $2$,
\[
  \kappa(\Psi(P)\fp\Psi(Q)) = \kappa(\Psi(P))\cdot\kappa(\Psi(Q))/(\av-\bv).
\]
Hence
\[
  \kappa(u\fp v) = \kappa(u\cdot v)/(\av - \bv)
  + \sweedle{\Delta^*}{ (u_{(1)}\diamond v_{(1)})\cdot\bv\cdot\kappa(u_{(2)}\diamond^* v_{(2)}}
\]
whenever $u$ and $v$ have sufficiently large degree.

While the operations $\times$ and $\diamond$ have the $\cd$-polynomials
$\xi$ and $1$ respectively as units, the unit of $\fp$ is $\av$.  Hence
the manifold product does not preserve the $\cd$-index.  There are still
recursive rules for computing $u\fp v$.  In particular,
\begin{align*}
u\fp(v\cdot(\av - \bv)) &= (u\fp v)\cdot(\av - \bv) \text{\ and} \\
u\fp(v\cdot\dv) &= \sweedle{\Delta}{(u_{(1)}\diamond v)\cdot 2\dv\cdot u_{(2)}}.
\end{align*}

Although the manifold product does not generally preserve $\cd$-polynomial
or nonnegativity, there are some special cases where it does.  In particular,
$\cv^p\fp\cv^q$ is a $\cd$-polynomial if $p+q$ is odd, and 
$\cv^n\fp\cv^{n+1}$ is a nonnegative $\cd$-polynomial for any $n$.  Increasing
the difference in degree between the arguments rapidly introduces negative terms.
Since these expressions denote $\ab$-indices of products of spheres of 
different dimensions, we would like to give conditions which guarantee 
nonnegativity of the coefficients.

%
%
%

\vfill
\begin{center}
Copyright \copyright\ Michael Slone 2008
\end{center}
%
\setcounter{chapter}{2}
\chapter{A geometric approach to acyclic orientations}
\label{chap:acyclic}




The set of acyclic orientations of a connected graph with a given sink
has a natural poset structure. We give a geometric proof of a result
of 
Propp: this poset is the disjoint union of distributive lattices.

Let $G$ be a connected graph on the vertex set
$[\underline{n}] = \{0\}\cup[n]$,
where $[n]$ denotes the set $\{1, \ldots, n\}$.
Let $P$ denote the collection of acyclic orientations of $G$, and 
let $P_0$ denote the collection of acyclic orientations of $G$ with $0$ as a sink.
If $\Omega$ is an orientation in $P$ with the vertex $i$
as a source, we can 
obtain a new orientation $\Omega'$ with $i$ as a sink by \emph{firing} the
vertex $i$,
reorienting all the edges adjacent to $i$ towards $i$.
The orientations $\Omega$ and~$\Omega'$ agree away from $i$.

A \emph{firing sequence} from $\Omega$ to $\Omega'$ in $P$ consists of a sequence
$\Omega=\Omega_1,\dots,\Omega_{m+1}=\Omega'$ of orientations and a
function $F\colon [m]\to [\underline{n}]$ such that for each $i\in [m]$, the 
orientation~$\Omega_{i+1}$ is obtained from $\Omega_i$ by firing the vertex $F(i)$.
We will abuse language by calling $F$ itself a firing sequence.
We make $P$ into a preorder by writing $\Omega\le\Omega'$ if and only if
there is a firing sequence from $\Omega$ to $\Omega'$.  From
the definition it is clear that $P$
is reflexive and transitive.
While $P$ is only a preorder, $P_0$ is a poset.  By finiteness, 
antisymmetry can be verified by showing that firing sequences in $P_0$
cannot be arbitrarily long.  This is a consequence of the fact that
neighbors of the distinguished sink $0$ cannot fire.
The proof depends on the
following lemma.

\begin{lemma}
Let $F\colon [m]\to [n]$ be a firing sequence for the graph $G$.  If $i$ and $j$
are adjacent vertices in $G$, then
\[
  |F^{-1}(i)| \le |F^{-1}(j)| + 1.
\]
\end{lemma}

\begin{proof}
A vertex can fire only if it is a source.  Firing $i$ reverses the orientation
of its edge to~$j$.  Hence $i$ cannot fire again until the orientation is again
reversed, which can only happen by firing $j$.
\end{proof}

As a corollary, firing sequences have bounded length, implying that $P_0$ 
is a poset.

\begin{corollary}
The preorder $P_0$ of acyclic orientations with a distinguished sink is a poset.
\end{corollary}

\begin{proof}
Let $F\colon [m]\to [n]$ be a firing sequence.  By 
iterating the lemma, $|F^{-1}(i)| \le d(0, i) - 1$, so
\[
  m =   \sum_{i\in [n]} |F^{-1}(i)| 
    \le \sum_{i\in [n]} d(0, i) - 1
\]
Hence firing sequences cannot be arbitrarily long, implying that $P_0$ is
antisymmetric.
\end{proof}

For a real number $a$ let
$\lfloor a \rfloor$ denote the largest integer
less than or equal to $a$. 
Similarly,
let
$\lceil a \rceil$ denote the least integer
greater than or equal to $a$. 
Finally, let $\fracp{a}$ denote the fractional part of the real number $a$,
that is, $\fracp{a} = a - \lfloor a \rfloor$.
Observe that the range of the function $x \longmapsto \fracp{x}$
is the half open interval $[0,1)$.
In this chapter we use $\fracp{a}$ only to denote the fractional
part and never to denote a singleton set.

Let $\widetilde{\mathcal{H}} = \widetilde{\mathcal{H}}(G)$ be
the \emph{periodic graphic arrangement} of the graph $G$,
that is, 
$\widetilde{\mathcal{H}}$ is the collection of all hyperplanes
of the form
\[
     x_{i} = x_{j} + k  , 
\]
where $ij$ is an edge in the graph $G$ and $k$ is an integer.
This hyperplane arrangement 
cuts~$\mathbb{R}^{n+1}$ into open regions.
Note that each region is translation-invariant
in the direction $(1, \ldots, 1)$.
Let $C$ denote the complement of $\widetilde{\mathcal{H}}$,
that is,
\[    C
    =
      \mathbb{R}^{n+1} \setminus \bigcup_{H \in \widetilde{\mathcal{H}}} H  . \]
Define a map $\varphi\colon C\to P$ from the complement of the
periodic graphic arrangement to the
preorder of acyclic orientations as follows.
For a point $x = (x_{0}, \ldots, x_{n})$ and an edge $ij$
observe that $\fracp{x_{i}} \neq \fracp{x_{j}}$
since the point does not lie on any hyperplane
of the form $x_{i} = x_{k} + k$.
Hence orient the edge $ij$ towards $i$ if $\fracp{x_{i}} < \fracp{x_{j}}$
and towards~$j$ if the inequality is reversed.
This defines the orientation $\varphi(x)$.  Also note that this
is an acyclic orientation, since no directed cycles can occur.

\vanish{
The poset of orientations consists of several connected components, each of
which is self-dual.  To get a better intuition for the structure of $P$, 
we will work with the \emph{periodic graphic arrangement} 
$\mathcal{A}=\mathcal{A}(G)$ in $\mathbb{R}^{n+1}$.  For each edge
$ij$ in $G$ and each integer $k$, the periodic graphic arrangement includes
the hyperplane $x_i = x_j + k$.  
}

\vanish{
Let $S$ denote the set obtained by intersecting the complement 
$\mathbb{R}^{n+1}\setminus\bigcup\mathcal{A}$ with the hyperplane $x_0 = 0$.
We now define a map $\varphi\colon S\to P$.  Let $x\in S$.  Since $x$ is
not on any hyperplane in $\mathcal{A}$, its coordinates are all distinct
modulo 1.  Hence $x$ corresponds to a permutation of $[\underline{n}]$, 
inducing an orientation $\Omega$ of the graph $G$.  Since $x_0 = 0$
but the other coordinates of $x$ are strictly positive modulo $1$, 
this orientation makes vertex $0$ a sink.  Hence~$\Omega$ is in $P$, 
and we can define the map $\varphi$ by sending $x$ to $\Omega$.
}

Let $H_{0}$ be the coordinate hyperplane
$\{x \in \mathbb{R}^{n+1} \: : \: x_{0} = 0\}$.
The map $\varphi$ sends points of the intersection $C_0 = C \cap H_{0}$
to acyclic orientations in $P_0$.

\vanish{
Especially, if we restrict our attention to
the complement $C$ intersected the hyperplane $H_{0}$,
we obtain that $\varphi$ maps points to acyclic orientations
with a sink at the vertex $0$.
}

The real line $\mathbb{R}$ is a distributive lattice; meet is
minimum and join is maximum.  Since~$\mathbb{R}^{n+1}$ is a 
product of copies of $\mathbb{R}$, it is also a distributive
lattice, with meet and join given by componentwise minimum
and maximum.  That is,
given two points 
in~$\mathbb{R}^n$, say 
$x = (x_{0}, \ldots, x_{n})$
and
$y = (y_{0}, \ldots, y_{n})$,
their meet and join are given by
\[
  x \meet y = (\min(x_{0},y_{0}), \ldots, \min(x_{n},y_{n}))
\]
and
\[
  x \join y = (\max(x_{0},y_{0}), \ldots, \max(x_{n},y_{n}))
\]
respectively.

\vanish{
Recall that $\mathbb{R}^{n+1}$ is a finite product of chains, hence a 
distributive lattice.  In this lattice, the meet of two points is the
coordinatewise minimum and the join of two points is the coordinatewise 
maximum.  The set $S$ inherits the poset structure of $\mathbb{R}^{n+1}$,
and each component is a distributive sublattice.
}

\begin{lemma}
Each region $R$ in the complement $C$ of the periodic graphic arrangement
$\widetilde{\mathcal{H}}$
is a distributive sublattice of $\mathbb{R}^{n+1}$.
Hence the intersection $R \cap H_{0}$, which is a region in $C_{0}$,
is also a distributive sublattice of $\mathbb{R}^{n+1}$.
\end{lemma}
\begin{proof}
Since each region $R$ is the intersection of 
slices of the form
\[
    T
  =
    \{ x \in \mathbb{R} \:\: : \:\: x_i + k < x_j < x_i + k + 1 \},
\]
it is enough to prove that each slice is a sublattice of $\mathbb{R}^{n+1}$.
Let $x$ and $y$ be two points in the slice $T$.
Then
$\min(x_i,y_i) + k   =    \min(x_i + k, y_i + k)
                     <    \min(x_j, y_j)
                     <    \min(x_i + k+1, y_i + k+1)
                     =    \min(x_i,y_i) + k+1$,
implying that $x \meet y$ also lies in the slice $T$.
A dual argument shows that the slice $T$ is closed under the join operation.
Thus the region $R$ is a sublattice.
Since distributivity is preserved under taking sublattices,
it follows that $R$ is a distributive sublattice of $\mathbb{R}^{n+1}$.
\end{proof}

In the remainder of this chapter we let $R$ be a region
in $C_{0}$.

\begin{lemma}
Consider the restriction $\varphi|_{R}$ of the map $\varphi$
to the region~$R$.
The inverse image of an acyclic orientation in $P_{0}$ is
of the form:
\[
  R \cap \biggl(\{0\} \times \prod_{i=1}^{n} [a_{i},a_{i}+1) \biggr), 
\]
where each $a_{i}$ is an integer. That is, the inverse image
of an orientation is the intersection of the region
$R$ with a half-open lattice
cube.
Hence the inverse image is a sublattice of~$\mathbb{R}^{n+1}$.
\label{lemma_inverse_image}
\end{lemma}
\begin{proof}
Assume that $x$ and $y$ lie in the region $R$.
Define the integers $a_{i}$ and $b_{i}$ by
$a_{i} = \lfloor x_{i} \rfloor$
and
$b_{i} = \lfloor y_{i} \rfloor$.
Hence the coordinate $x_{i}$ lies in the half-open
interval $[a_{i},a_{i}+1)$
and 
the coordinate $y_{i}$ lie in the half-open
interval $[b_{i},b_{i}+1)$.
Lastly, assume that $\varphi|_{R}$ maps $x$ and $y$ to the
same acyclic orientation.
The last condition implies that for every edge $ij$
that
$0 \leq x_{i} - a_{i} < x_{j} - a_{j} < 1$ is equivalent to
$0 \leq y_{i} - b_{i} < y_{j} - b_{j} < 1$.
Consider an edge that is directed from $j$ to $i$.
Since $x$ and~$y$ both lie in the region~$R$,
there exists an integer $k$ such that           
$x_{i} + k < x_{j} < x_{i} + k + 1$ 
and         
$y_{i} + k < y_{j} < y_{i} + k + 1$.
Now we have that        
$a_{j} - a_{i} < x_{j} - x_{i} < k+1$.          
Furthermore, observe that           
$x_{j} - a_{j} - 1 < 0 \leq x_{i} - a_{i}$.     
Hence       
$a_{j} - a_{i} > x_{j} - x_{i} - 1 > k-1$.      
Since $a_{j} - a_{i}$ is an integer, the two    
bounds implies that     
$a_{j} - a_{i} = k$.    
By similar reasoning we obtain that
$b_{j} - b_{i} = k$.    

Hence for every edge $ij$ we know that          
$a_{j} - a_{i} = b_{j} - b_{i}$.    
Since $a_{0} = b_{0} = 0$ and the graph $G$ is         
connected we obtain that $a_{i} = b_{i}$ for all
vertices $i$.
\end{proof} 

\begin{lemma}
The restriction $\varphi|_{R} : R \to P_0$ is a poset map.
\end{lemma}
\begin{proof}
Assume that $y$ and $z$ belong to the region $R$ and that $y \leq z$.
Since the region~$R$ is convex, the line segment from $y$ to $z$ is 
contained in~$R$. 
Let a point $x$ move continuously from $y$ to $z$ along this line
segment
and consider what happens with the associated
acyclic orientations $\varphi(x)$. Note that each coordinate
$x_{i}$ is non-decreasing.
When the point $x$ crosses an hyperplane of the form $x_{i} = p$
where $p$ is an integer, observe that the value $\fracp{x_{i}}$
approaches $1$ and then jumps down to $0$. Hence the vertex~$i$ 
switches from being a source to being a sink, that is, the
vertex~$i$ fires.

Observe that two adjacent nodes $i$ and $j$ cannot fire at the same
time, since the intersection of the two hyperplanes
$x_{i} = p$ and $x_{j} = q$ is contained in the hyperplane
$x_{i} = x_{j} + (p-q)$ which is not in the region $R$.

Hence we obtain a firing sequence from the acyclic orientation
$\varphi(y)$ to $\varphi(z)$,
proving that $\varphi(y) \leq \varphi(z)$.
\end{proof}

\vanish{
Along any monotonic path in $C_0$ from $x$ to $y$, the value of $\varphi$
changes only finitely many times, and we may assume it changes exactly once.
For each $i$ in $[n]$, let $\alpha_i$ be the straight-line path in $C_0$
from $(y_1, \ldots, y_{i-1}, x_i, x_{i+1}, \ldots, x_n)$ 
to   $(y_1, \ldots, y_{i-1}, y_i, x_{i+1}, \ldots, x_n)$.
By following these paths in sequence we get a monotonic path from $x$ to $y$.
The value of $\varphi$ changes in precisely one path $\alpha_{i+1}$.
The only 
change that occurs in this path is the monotonic increase
of the $i$ coordinate 
from $x_i$ to $y_i$.  Thus there must be an integer between $x_i$ and $y_i$.  
Since $\alpha_{i+1}$ is a path in $C_0$, the $i$th
coordinate must be maximal in $x$ and minimal in $y$.
In other words, $\varphi(y)$ is obtained from $\varphi(x)$
by firing the vertex $i$.
Hence $\varphi(x) \le \varphi(y)$.
}

\vanish{
\begin{lemma}
Let $R$ be a region of the complement $C$ and
$x$ a point in the region $R$. Let $\Omega'$ be 
an acyclic orientation comparable to the acyclic orientation
$\Omega = \varphi(x)$. Then there exists a point $y$ in
the region $R$ such that $\varphi(y) = \Omega'$.
\end{lemma}
}

\begin{lemma}
Let $x$ be a point in the region $R$.  Let $\Omega'$ be
an acyclic orientation comparable to $\Omega = \varphi(x)$ in the poset $P_0$.
Then there exists a point $z$ in the region of $R$ as $x$
such that $\varphi(z) = \Omega'$.
\label{lemma_lifting}
\end{lemma}
\begin{proof}
It is enough to prove this for cover relations in the poset $P$.
We begin by considering the case when $\Omega'$ covers
$\Omega$ in $P$.  Thus $\Omega'$ is obtained from $\Omega$ by firing
a vertex $i$.

First pick a positive real number $\lambda$ such that
$\fracp{x_{j}} < 1 - \lambda$ for each nonzero vertex~$j$.
Let $y$ be the point $y = x + \lambda \cdot (0,1, \ldots, 1)$.
Observe that $y$ belongs to the same region $R$ and that
$\varphi$ maps $y$ to the same acyclic orientation as the point $x$.

Since $i$ is a source in $\Omega$, the value
$\fracp{y_{i}}$ is larger than any other 
value 
$\fracp{y_{j}}$ for vertexes~$j$ adjacent to the vertex $i$.
Let $z$ be the point with coordinates
$z_{j} = y_{j}$ for $j \neq i$ and
$z_{i} = \lceil y_{i} \rceil + \lambda/2$.
Observe that moving from $y$ to the point $z$
we do not cross any hyperplanes of the form $x_{i} = x_{j} + k$.
Hence the point $z$ also belongs to region $R$.

However, we did cross a hyperplane of the form $x_{i} = p$,
corresponding to firing the vertex~$i$. Hence we have that
$\varphi(z) = \Omega'$.  Now we can iterate this to extend
to the general case when $\Omega < \Omega'$.

The case when $\Omega'$ is covered by $\Omega$ is done
similarly. However this case is easier since one can skip
the middle step of defining the point $y$. Hence this case
is omitted.
\end{proof}

A connected component of a finite poset is a weakly connected
component of its associated comparability graph.
That is, a finite poset is the disjoint union
of its connected components.

\begin{lemma}
Let $Q$ be a connected component
of the poset of acyclic orientations~$P_{0}$.
Then there exists a region $R$ in $C_{0}$ such that
the map $\varphi$ maps $R$ onto the component $Q$.
\label{lemma_lifting_components}
\end{lemma}
\begin{proof}
Let $\Omega$ be an orientation in the component $Q$.
Since $\varphi$ is surjective we can lift~$\Omega$ to a point
$x$ in $C_{0}$. Say that the point $x$ lies in the region $R$.
It is enough to show that every orientation
$\Omega^{\prime}$ in $Q$ can be lifted
to a point in $R$. The two orientations
$\Omega$ and $\Omega^{\prime}$ are related by a sequence in $Q$
of orientations 
$\Omega = \Omega_{1}, \Omega_{2}, \ldots, \Omega_{k} = \Omega^{\prime}$
such that $\Omega_{i}$ and $\Omega_{i+1}$ are comparable.
By iterating 
Lemma~\ref{lemma_lifting}
we obtain points~$x_{i}$ in~$R$
such that $\varphi(x_{i}) = \Omega_{i}$.  In particular, 
$\varphi(x_{k}) = \Omega^{\prime}$.
\end{proof}

\begin{proposition}
Let $Q$ be a connected component of the poset of
acyclic orientations $P_{0}$. Then the component $Q$ as a poset is a lattice.
Moreover, let $R$ be a region of $C_{0}$ that maps onto $Q$ by $\varphi$.
Then the poset map $\varphi|_{R} : R \longrightarrow Q$ is
a lattice homomorphism.
\end{proposition}
\begin{proof}
The previous discussion showed that we can lift the component
$Q$ to a region~$R$. Consider two acyclic orientations
$\Omega$ and $\Omega'$. We can lift them to two points~$x$ and~$y$
in~$R$, that is,
$\varphi(x) = \Omega$
and
$\varphi(y) = \Omega'$.
Since $\varphi|_{R}$ is a poset map we obtain that
$\varphi(x \meet y)$ is a lower bound for
$\Omega$ and $\Omega'$. It remains to show that the lower bound
is unique. 

Assume that $\Omega''$ is a lower bound of $\Omega$ and $\Omega'$.
By Lemma~\ref{lemma_lifting}
we can lift $\Omega''$ to an element $z$ in $R$ such
that $z \leq x$.
Similarly,
we can lift $\Omega''$ to an element $w$ in $R$ such
that $w \leq y$.
That is we have that $\varphi(z) = \varphi(w) = \Omega''$.
Now by Lemma~\ref{lemma_inverse_image}
we have that $\varphi(z \meet w) = \Omega''$.
But since $z \meet w$ is a lower bound of both $x$ and $y$
we have that $z \meet w \leq x \meet y$.
Now applying $\varphi$ we obtain that
$\varphi(x \meet y)$ is the greatest lower bound, proving that
the meet is well-defined.
A dual argument shows that the join is well-defined, hence
$Q$ is a lattice.

Finally, we have to show that $\varphi|_{R}$ is
a lattice homomorphism.
Let $x$ and $y$ be two points in the region $R$.
By Lemma~\ref{lemma_lifting} we can lift
the inequality $\varphi(x) \meet \varphi(y) \leq \varphi(x)$
to obtain a point $z$ in $R$ such that
$z \leq x$ and $\varphi(z) = \varphi(x) \meet \varphi(y)$.
Similarly, we can lift
the inequality $\varphi(x) \meet \varphi(y) \leq \varphi(y)$
to obtain a point $w$ in $R$ such that
$w \leq y$ and $\varphi(w) = \varphi(x) \meet \varphi(y)$.
By Lemma~\ref{lemma_inverse_image} we know that
$\varphi(z \meet w) = \varphi(x) \meet \varphi(y)$.
But~$z \meet w$ is a lower bound of both $x$ and $y$,
so
$\varphi(x) \meet \varphi(y) = \varphi(z \meet w) \leq \varphi(x \meet y)$.
But since 
$\varphi(x \meet y)$ is a lower bound of both
$\varphi(x)$ and $\varphi(y)$ we have
$\varphi(x \meet y) \leq \varphi(x) \meet \varphi(y)$.
Thus the map $\varphi|_{R}$ preserves the meet operation.
The dual argument proves that 
$\varphi|_{R}$ preserves the join operation,
proving that it is a lattice homomorphism.
\end{proof}

Combining these results we can now prove the result of
Propp~\cite{Propp}.

\begin{theorem}
Each connected component of the poset of acyclic orientations $P_{0}$
is a distributive lattice.
\end{theorem}
\begin{proof}
It is enough to recall that $\mathbb{R}^{n+1}$ is a distributive lattice
and each region $R$ is a sublattice. Furthermore, the image
under a lattice morphism of a distributive lattice is also distributive.
\end{proof}

Observe that the minimal element in each connected component $Q$
is an acyclic orientation with the unique sink at the vertex $0$.
Greene and Zaslavsky~\cite{Greene_Zaslavsky}
proved that the number such orientations
is given by
the sign $-1$ to the power one less than the number of vertices
times
the linear coefficient in the chromatic polynomial
of the graph $G$.
Gebhard and Sagan gave several proofs of this result~\cite{Gebhard_Sagan}.
A geometric proof of this result can be found 
in Chapter~\ref{chap:affinetoric} of this dissertation.

That the connected component are confluent,
that is, each pair of elements has a lower and an upper bound,
can also be shown to follow from 
a special case of chip-firing games~\cite{Bjorner_Lovasz_Shor}.
Is there a geometric way to prove
the confluency of chip-firing?
More discussions relating these distributive lattice
with chip-firing can be found in~\cite{Latapy_Magnien,Latapy_Phan}.

\vfill
\begin{center}
Copyright \copyright\ Michael Slone 2008
\end{center}

%

%
%
%
%
%
%
%
%
%
%
%
%
%
%

%
%
\setcounter{chapter}{3}
\chapter{Critical groups of cleft graphs}
\label{chap:critical}

\section{Introduction}

The number of spanning trees of an undirected graph is an important
invariant of the graph.  The matrix tree theorem reduces the
problem of determining the tree number to linear 
algebra. (The problem of listing all spanning trees for a 
specific graph was solved by Feussner~\cite{Feussner,Feussner_two} using 
what is essentially deletion-contraction.)
\begin{theorem}[Kirchhoff's matrix tree theorem~\cite{Kirchhoff}]\label{thm:matrixtree}
Let $X$ be a graph on $n$ vertices with Laplacian $L$.
Suppose $\lambda_1\le\dots\le\lambda_n$ are the eigenvalues of $L$.
Then the tree number of $X$ is
\[
  \tree{X} = \frac{1}{n}\prod_{i=2}^n \lambda_i.
\]
Equivalently, $\tree{X}$ is the value of any cofactor of $L$.
\end{theorem}
\noindent
Kirchhoff developed this theorem with the theory of
electrical networks in mind.  

More than a hundred years later, the physicists
Bak, Tang, and Wiesenfeld~\cite{Bak_Tang_Wiesenfeld} developed
the apparently unrelated abelian sandpile model in an attempt
to explain flicker noise, an effect which appears in widely varying
physical systems.  In the abelian sandpile model, grains of
sand are added one at a time to small piles of sand.  Since
this is inherently unstable, eventually a pile will collapse,
distributing grains to neighboring piles.  They called a configuration
critical if it is stable but becomes unstable if a single grain is
added anywhere.

The problem of characterizing
critical configurations was studied by graph theorists and other
combinatorialists in the 1990s under the guise of chip-firing games.
A chip-firing game, in the sense of Bj\"orner, Lov\'asz, and 
Shor~\cite{Bjorner_Lovasz_Shor}, involves firing vertices in 
a finite graph~$G$ with
a nonnegative number of chips on each vertex.  A vertex fires by
distributing a chip to each of its neighbors, and cannot fire unless
it has sufficiently many chips.  Only one vertex can fire at a 
time, so it might be expected that the decision of which vertex 
to fire at a particular step is of major importance.  
However, Bj\"orner, Lov\'asz, and Shor
showed that a chip-firing game on a graph is a confluent system.
Hence if an initial configuration is not recurrent, its terminal 
stable configuration
of chips does not depend on the order in which vertices are fired.
Biggs~\cite{Biggs_dollar} developed a variant of this called the
dollar game, which includes one vertex which can fire
if and only if 
no other vertex can fire, even if it would have a negative number
of chips after doing so.  Biggs proved that the number of 
critical configurations of a graph is equal to the order of 
the critical group, which is the torsion part of the cokernel
of the Laplacian.  Thus the problem of counting spanning trees
is subsumed by the problem of understanding the critical group
of a graph.

The critical group is only known for a few classes of graphs.  
In this chapter, we study cleft graphs, which are graphs
obtained from a base graph by replacing each vertex with an anticlique,
that is, a collection of nonadjacent vertices. This
construction is the vertex analogue of that of Lorenzini~\cite{Lorenzini},
who studied the effect of replacing all edges in a graph with
paths of uniform length.  We derive an exact sequence relating
the critical group of a uniformly cleft graph with that of its
base graph.  
Moreover, we also have results in the non-uniform case.  By 
studying the spectrum of the Laplacian we are able to determine
the tree number of a non-uniformly cleft tree.

\section{Preliminaries}

All graphs we consider are simple, loopless, undirected
graphs with no parallel edges.
Our discussion will be greatly simplified if we imagine a graph
as being endowed with an orientation.  None of our results depend
on which orientation is used.  With this in mind, we define
an \define{oriented graph} to be a structure $X = (EX, VX)$
consisting of a 
set of edges $EX$ and a set of vertices $VX$ which are related 
by a pair of structure
maps from edges to vertices, called $s$ for source and $t$ for target.
If $X = (EX, VX)$ and~$B = (EB, VB)$ are oriented graphs, then 
a morphism $\varphi\colon X\to B$ consists of two functions
$\varphi\colon EX\to EB$ and $\varphi\colon VX\to VB$ such that
an oriented edge with source~$u$ and target $v$ is mapped to an oriented
edge with source $\varphi(u)$ and target $\varphi(v)$.  We 
let~$\delta(v)$ denote the neighborhood of a vertex $v$ in the 
unoriented graph.  
The degree of $v$ is denoted by $\deg(v)$ and is the size of the
neighborhood, that is, $\deg(v) = |\delta(v)|$.
In an oriented graph, the neighborhood of a vertex $v$
decomposes as $\delta(v) = \delta^+(v)\sqcup \delta^-(v)$, 
where~$\delta^+(v)$ is the out-neighborhood of $v$, the set of vertices 
reachable from $v$
in one step, and~$\delta^-(v)$ is the in-neighborhood of $v$, 
the set of vertices from which $v$ can be reached in one step.

An oriented graph $X$ can be viewed as an oriented 1-dimensional cell 
complex.  Hence $X$ comes equipped with a chain complex $\Chain{X}{}$, where
\[
  \Chain{X}{1} = \bigoplus_{e\in EX}\mathbb{Z}e
  \text{\quad and\quad}
  \Chain{X}{0} = \bigoplus_{v\in VX}\mathbb{Z}v.
\]
The boundary map $\bd\colon\Chain{X}{1}\to\Chain{X}{0}$ is defined
on an edge $e$ by $\bd(e) = t(e) - s(e)$.  Hence the boundary
map is the same as the incidence matrix of the graph.
The \define{Laplacian} of $X$
is the map $L = \bd\trans{\bd}$, where $\trans{\bd}$ is the transpose
of $\bd$.  Thus $\trans{\bd}$ represents the coboundary of the graph.  
If $X$ has $n$ vertices, we can view
$X$ as an~$n\times n$ matrix.  For vertices $u$ and $v$, one 
can compute that the $(u, v)$
entry of $L$ is 
\[
L(u, v)
=
     \begin{cases}
     \deg(u),   & u = v \\
     -\#[u, v], & u \ne v,
     \end{cases}
\]
where the notation $[u, v]$ indicates the set of edges with 
endpoints $u$ and $v$ in either orientation.  Some authors use this
as the definition of the Laplacian matrix.
Thus~$L = D - A$, where $D$ is the diagonal matrix whose diagonal
gives the degree sequence of $X$ and $A$ is the incidence matrix of $X$.
The \define{critical group} of $X$ is
the torsion part of the cokernel of $L$.  The cokernel can be
found by 
reducing $L$ to its Smith normal form, which can be done using
row and column operations which are invertible over the integers.

\begin{figure}
\begin{center}
\scalebox{1}{\epsfig{file=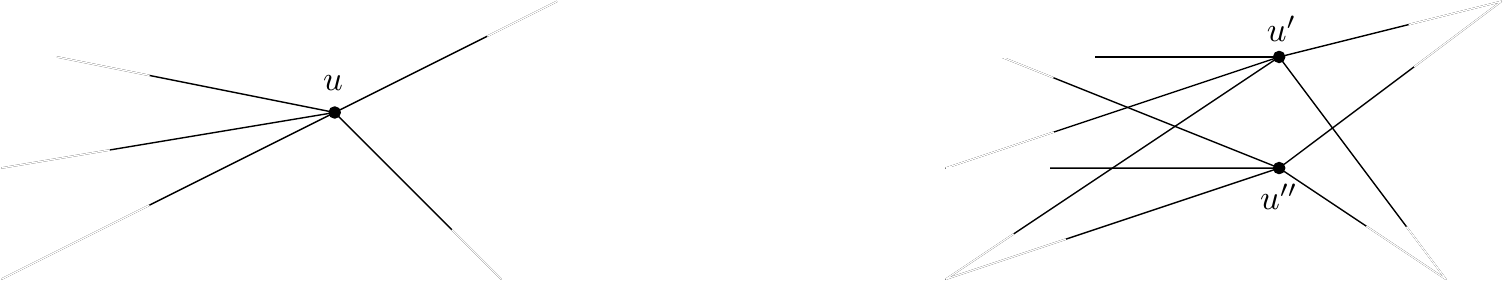}}
\end{center}
\caption{Cleaving the vertex $v$ replaces it with an anticlique.}
\label{fig:cleftexplain}
\end{figure}
Cleft graphs are similar to graph fibrations, but they 
obey a weaker unique lifting condition.  Hence we will adopt some of 
the language, including the notions of total graph and base graph.  
Before presenting the technical definition of cleft graph we offer
the following way to visualize cleaving a single vertex in two.
Suspend the graph by the vertex to be cleft, so that the edges which
connect it to the rest of the graph are hanging downwards.  Carefully
drape these edges and the vertex on a chopping block.  Then take
a very sharp (and infinitely thin) cleaver and cut through the vertex
and its incident edges.  Thus the vertex is cleft into two vertices, and each
of the edges incident with the vertex is cleft into two edges, one
for each half of the cleft vertex.  Thus the vertex to be cleft
has been replaced with two nonadjacent vertices, each of which has
the same neighborhood as the cleft vertex.  See Figure~\ref{fig:cleftexplain}.
In a similar way, we can cleave a vertex $m$-fold, replacing the
vertex with an anticlique of $m$ vertices, each with the same 
neighborhood as the cleft vertex.

The structure of a graph after multiple vertices have been cleft does
not depend on the order in which the cleavings were performed.  So given
a graph $B$ and a weight vector on the vertex set of $B$,
there is a unique graph $X$ which is obtained from~$B$ by cleaving 
each vertex of $B$ according to its weight.  Moreover, there is a natural
projection morphism $p\colon X\to B$ which assigns each vertex in~$X$
to the vertex in $B$ from which it was cleft.
Hence we can define a \define{cleft graph} to be an 
oriented graph morphism $p\colon X\to B$ which satisfies the 
following two properties:
\begin{itemize}
\item (weak unique lifting)
For any vertices $\widetilde{u}$, $\widetilde{v}\in VX$, 
if $e\in EB$ is an edge from~$p(\widetilde{u})$ to~$p(\widetilde{v})$, then 
the edge $e$ has a unique lift $\widetilde{e}\in EX$
with source
$\widetilde{u} = s(\widetilde{e})$ and target~$\widetilde{v} = t(\widetilde{e})$.
\item (cleaving)
Each fibre of $p$ is a nonempty anticlique.
\end{itemize}
Observe that since $p$ is a graph morphism, the edge $\widetilde{e}$
mentioned above is a lift of~$e$.
We say that the cleft graph is induced by the weight
vector~$(m_v)_{v\in VB}$, where the weight 
of a vertex is the size of its fibre, that is,~$m_v = |p^{-1}(v)|$.
A cleft graph is~\define{$m$-uniform} if every fibre has the same size $m$.
An example of a non-uniform cleft graph appears in Figure~\ref{fig:cleftpair}.

\begin{figure}
\begin{center}
\scalebox{1}{\epsfig{file=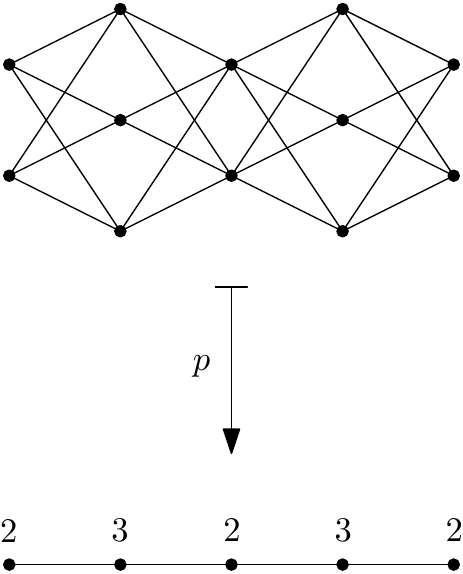}}
\end{center}
\caption{A bipartite graph viewed as a cleft path.}
\label{fig:cleftpair}
\end{figure}

Suppose we know the tree number or critical group of a base graph $B$.  It is 
natural to ask how much we can deduce about the tree number or critical 
group of the total graph of a cleft graph over $B$.  
It turns out that this is not difficult if the cleaving is uniform
or if the base graph is a tree.
In Section~\ref{exact-induced}, we determine the tree number of a
uniformly-cleft graph.  In Section~\ref{cleft-tree}, we determine the
tree number of an non-uniformly-cleft tree.  

\section{The exact sequence of a uniformly-cleft graph}\label{exact-induced}

Let $p\colon X\to B$ be a cleft graph.  Since the projection $p$
is a graph morphism, it commutes with the boundary map, that
is, $\bd p=p\bd$.  The interaction between the coboundary map $\trans{\bd}$
and the projection is slightly more complex, and is described
by the following lemma.

\begin{lemma}\label{lemma:compressedmakessense}
Let $p\colon X\to B$ be a cleft graph with weight vector $\mathbf{m}$.
Define a linear map $\varphi\colon C_0(B)\to C_1(B)$ by
$\varphi(v) = \sum_e \varphi(e, v)\cdot e$, where 
\[
  \varphi(e, v) = \begin{cases}
   m_{s(e)} & t(e) = v \\
  -m_{t(e)} & s(e) = v \\
   0        & \text{otherwise.}
   \end{cases}
\]
Then the diagram
\[\xymatrix{
C_0(X)\ar[r]^{\trans{\bd}}\ar[d]_{p} &
C_1(X)\ar[d]^{p} \\
C_0(B)\ar[r]_{\varphi} &
C_1(B)
}\]
is commutative.  Moreover, if $X$ is an $m$-uniformly cleft graph,
then $\varphi = m\trans{\bd}$.
\end{lemma}

\begin{proof}
The composite map $\varphi p$ is given by
\[
  \varphi p(e, \widetilde{v})
= 
  \sum_{u \in VB} \varphi(e, u) p(u, \widetilde{v}) 
= 
   \varphi(e, p(\widetilde{v})).
\]
On the other hand, the 
composite map $p\trans{\bd}$ is given by
\[
  p\trans{\bd}(e, \widetilde{v})
    = \sum_{\widetilde{f}\in EX} p(e, \widetilde{f}) \trans{\bd}(\widetilde{f}, \widetilde{v})
    = \sum_{\widetilde{e} \in p^{-1}(e)} \trans{\bd}(\widetilde{e}, \widetilde{v}).
\]
The sum vanishes unless $e$ is incident with $p(\widetilde{v})$.
The number of lifts of $e$ which have a given endpoint is given 
by the weight of the vertex at the other endpoint, and the sign of
the term $\trans{\bd}(\widetilde{e}, \widetilde{v})$ is determined
by whether $p(\widetilde{v})$ is the source or target of $e$. 
Hence $p\trans{\bd} = \varphi\bd$, as claimed.
\end{proof}

Combining Lemma~\ref{lemma:compressedmakessense} with the 
commutativity relation $\bd p=p\bd$, we can define the
\define{compressed Laplacian} of a cleft graph $X$ with
respect to its base graph $B$ as the composite map $C = \bd\varphi$.
If we define a vector $(M_v)_{v\in VB}$ by
\[
  M_v = \sum_{u\in\delta(v)} m_u,
\]
then it follows directly that
\[
  C(u,v) = \begin{cases}
  M_u,             & u =   v \\
  -m_u\cdot\#[u,v] & u \ne v
  \end{cases}
\]
for any vertices $u$, $v\in VB$.

\begin{corollary}\label{cor:uniformlycompressed}
Let $p\colon X\to B$ be an $m$-uniformly cleft graph.
Then the compressed Laplacian of $X$ is $m\cdot L(B)$.
\end{corollary}

For the rest of this section we will specialize to the case 
of an $m$-uniformly cleft graphs $p\colon X\to B$.  Let~$(M_v)$
be the vector defined above.
Thus $M_v$ is the degree in $X$ of any lift~$\widetilde{v}$ of $v$.
Let $S$ denote the transpose of $p$.  The map $S$ sends
a vertex $v$ to the sum of its lifts, that is, $S(v) = \sum_{\widetilde{v}\in p^{-1}(v)} \widetilde{v}$.
Since both $L(X)$ and $C=m\cdot L(B)$ are symmetric matrices, it follows
from Lemma~\ref{lemma:compressedmakessense} that the diagram
\[\xymatrix{
  C_0(B)\ar[r]^{C}\ar[d]_{S} &
  C_0(B)\ar[d]^{S} \\
  C_0(X)\ar[r]_{L} &
  C_0(X)
}\]
is commutative.
Since $SC = LS$, there is an injection $\coker C\to\coker L$.
We can use the fact that~$X$ is a uniformly cleft graph to
determine the factor by which splitting increases the tree number.
But first we need a lemma.

\begin{lemma}\label{lemma:split}
Let $p\colon X\to B$ be an $m$-uniformly cleft graph, and
define a vector~$(M_v)_{v\in VB}$ by $M_v = m\cdot|\delta(v)|$.
If $B$ is connected, then there is an exact sequence
\[
  0 \to 
  \coker C \to 
  K(X)\oplus\mathbb{Z} \to
  \bigoplus_{v\in VB} \mathbb{Z}_{M_v}^{m-1}/M_v \to 
  0
\]
of abelian groups.
\end{lemma}

\begin{proof}
Since $B$ is connected, so is $X$.  Thus $\coker L = K(X)\oplus\mathbb{Z}$.
Applying the snake lemma to the commutative diagram
\[\xymatrix{
0 \ar[r] & 
C_0(B) \ar[r]^{S}\ar[d]_{C} &
C_0(X) \ar[r] \ar[d]_{L} &
\coker S\ar[r] \ar[d]_{\overline{L}} &
0 \\
0 \ar[r] & 
C_0(B) \ar[r]_{S} &
C_0(X) \ar[r]  &
\coker S\ar[r]  &
0 
}\]
with exact rows yields the exact sequence
\[
  \ker{\overline{L}} \to
  \coker C\to 
  K(X)\oplus\mathbb{Z} \to
  \coker{\overline{L}}\to 
  0,
\]
where the map $\overline{L}\colon\coker S\to\coker S$ is
induced by $L$.
For any vertex $v$ in $B$, the sum of the lifts of $v$ in $X$
is a representative of zero in 
$\coker S$, but there are no other relations among the 
vertices of $X$.  The Laplacian sends a lift $\widetilde{v}$
of $v$ to 
\[
  L(\widetilde{v}) = M_v\widetilde{v} - m\sum_{u\in\delta^+(v)}S(u),
\]
which by our observation represents $M_v\widetilde{v}$ in $\coker S$.
Hence we can represent $\overline{L}$ by the block matrix
$\bigoplus_{v\in VB} M_v I_{m - 1}$, which is injective and has the
desired cokernel.
\end{proof}

To make this exact sequence useful for enumeration, we need to 
kill the infinite factors in $\coker C$ and $K(X)\oplus\mathbb{Z}$.
The following observation allows us to do this.

\begin{lemma}\label{lemma:generator}
Let $M$ be an $n\times n$ integer matrix with corank $1$.  Let 
$H$ be the submodule of $\mathbb{Z}^n$ generated by all vectors
whose coordinates sum to $0$.  If $\im M\subseteq H$, then
each standard basis vector $e_i$ represents an infinite 
generator of $\coker M$, possibly with nonzero torsion part.
\end{lemma}

\begin{proof}
First observe that $\mathbb{Z}^n$ is isomorphic to $\mathbb{Z}$
and is generated by any standard basis vector $e_i$.  Lifting
$e_i$ to $\coker M$ yields an element of the form
$n\cdot \gamma + r$, where~$\gamma$ is the infinite generator
of $\coker M$ and $r$ is a torsion element.  But this implies
that~$\gamma$ is mapped to $n^{-1}$ times the generator of
$\mathbb{Z}^n/H$ under the canonical surjection.  Hence~$n$ is
a unit.
\end{proof}

\begin{proposition}\label{prop:uniformcase}
Let $p\colon X\to B$ be an $m$-uniformly cleft graph with 
Laplacian $L$ and compressed Laplacian $C$.
If $B$ has $n$ vertices, then
the tree number of $X$ is given by the formula
\[
  \kappa(X) = \kappa(B)\cdot m^{n-2}\cdot\prod_{v\in VB} (m\cdot\deg(v))^{m-1}.
\]
Moreover, if the Smith normal form of $L(B)$ has the form
$\diag(d_1,\dots, d_{n-1}, 0)$, 
then the critical group of $X$ fits into the exact sequence
\[\xymatrix{
0\ar[r] &
\bigoplus_{i=1}^{n-1}\mathbb{Z}_{m\cdot d_i}\ar[r] &
K(X)\ar[r] &
\biggl(\bigoplus_{u}\mathbb{Z}_{m\cdot\deg(u)}^{m-1}\biggr)/\mathbb{Z}_m\ar[r] &
0. 
}\]
\end{proposition}

\begin{proof}
We may assume $B$ is connected.  
The map $\coker C\to\coker L\cong K(X)\oplus\mathbb{Z}$ sends 
the element $v + \im C$ to $S(v) + \im L$. 
By Lemma~\ref{lemma:generator} this element can be rewritten
as $m\cdot \widetilde{v} + \im L$ plus a torsion element, where
$\widetilde{v}$ is a lift of $v$ in $X$.
Hence the map must send the infinite generator of $\coker C$ to
$m$ times the infinite generator of~$\coker L$.  
This allows us to embed $K(X)\oplus\mathbb{Z}$
in the commutative diagram
\[\xymatrix{
 &
0\ar[d] &
0\ar[d] &
0\ar[d] &
 & \\
0\ar[r] &
\mathbb{Z}\ar[r]^m\ar[d] &
\mathbb{Z}\ar[r]\ar[d] &
\mathbb{Z}_m\ar[r]\ar[d] &
0 \\
0\ar[r] &
\coker C\ar[r]\ar[d] &
K(X)\oplus\mathbb{Z}\ar[r]\ar[d] &
\coker\overline{L}\ar[r]\ar[d] &
0 \\
0\ar[r] &
\coker C/\mathbb{Z}\ar[r]\ar[d] &
K(X)\ar[r]\ar[d] &
\coker \overline{L}/\mathbb{Z}_m\ar[r]\ar[d] &
0 \\
&
0 &
0 &
0 &
&
}\]
with exact rows and columns.  
Since $X$ is $m$-uniformly cleft, its compressed Laplacian is
$C = mL(B)$.  Hence $C$ has 
Smith normal form $\diag(m\cdot d_1,\dots,m\cdot d_{n-1},0)$.
This completes the
proof.
\end{proof}

The above proposition measures the growth in tree number produced
by uniform splitting.  We get the following corollary in the
case where the base graph is a tree.

\begin{corollary}\label{cor:uniformtree}
Let $p\colon X\to T$ be an $m$-uniformly cleft graph whose
base graph $T$ is a tree on $n$ vertices.  Then the tree number
of $X$ is
\[
  \kappa(X) = m^{n-2}\cdot\prod_{v\in VT}(m\cdot\deg(v))^{m-1}.
\]
\end{corollary}

We would like to extend this method to the case of non-uniformly cleft
graphs.
Since the compressed Laplacian need not be symmetric, it is unclear
how to do it.  In the next section, we will extend 
Corollary~\ref{cor:uniformtree} to the case of non-uniformly cleft trees.
However, the proof we give makes necessary use of the fact that the
base graph is a tree, and it is unclear how to generalize it.

\section{Tree numbers of cleft trees}\label{cleft-tree}

In this section we count spanning trees of a cleft tree using
a weighted analogue of the following classical theorem.

\begin{theorem}[Poincar\'e~\cite{Poincare}, Chuard~\cite{Chuard}]\label{Poincare_Chuard}
Let $X$ be a graph on $n$ vertices with incidence matrix $A$, 
and let $A'$ be an $n-1\times n-1$ submatrix of $A$.  The matrix
$A'$ is nonsingular (in fact, $\det(A') = \pm 1$) if and
only if the columns of $A'$ represent the edges of a spanning
tree of $X$.
\end{theorem}

To motivate the main ideas behind our argument, we study a 
recursive function on a special class of trees we
call weighted marked trees.  A \define{weighted marked tree}
is a tree $T$ together with a weight vector $\mathbf{m} = (m_v)_{v\in VT}$
and two special vertices, the root $r$ and a marked vertex $q$, which
could also be the root.  We define a function~$F(T, \mathbf{m}, r, q)$
according to the following recursive procedure.
\begin{enumerate}
\item
If $T$ has no edges, then $F(T, \mathbf{m}, r, q) = 1$.
\item
Otherwise:
\begin{enumerate}
\item
Let $v$ be a leaf of $T$.  Do not select the marked vertex $q$ unless it is 
the only leaf.
\item
Let $w$ be the parent of $T$.
\item
Let $T'$ be the tree obtained from $T$ by collapsing the edge connecting 
$v$ and $w$ to $w$.  Let $\mathbf{m}'$ be the restriction of the weight
vector of $T$ to the vertices of $T'$.
\item
Define a tuple $(q', w')$ by
\[
  (q', w') = \begin{cases}
  (q, w), & v \ne q \\
  (w, q), & v =   q.
  \end{cases}
\]
\item
With the above notation, $F(T, \mathbf{m}, r, q) = m_{w'}\cdot F(T', \mathbf{m}', r, q')$.
\end{enumerate}
\end{enumerate}
We illustrate this algorithm by applying it to the tree in 
Figure~\ref{fig:algtree1}.

\begin{figure}[h]
\begin{center}
\scalebox{1}{\epsfig{file=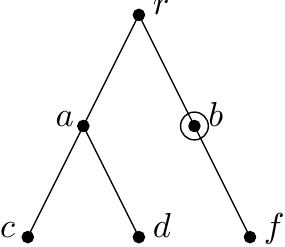}}
\end{center}
\caption{A weighted tree $T$ with root $r$ and marked vertex $b$.}
\label{fig:algtree1}
\end{figure}
\noindent
In order, we select the vertices $c$, $f$, $d$, and $a$, collapsing
the edges $ac$, $bf$, $ad$, and~$ra$, and picking up the weights
$m_a$, $m_b$, $m_a$, and $m_r$.  After these collapses, the tree has been
reduced to the tree $T'$ displayed in Figure~\ref{fig:algtree2}.
\begin{figure}[h]
\begin{center}
\scalebox{1}{\epsfig{file=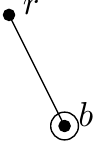}}
\end{center}
\caption{The tree $T'$ obtained from $T$ by collapsing several edges.}
\label{fig:algtree2}
\end{figure}
Now the marked vertex $b$ is the only leaf, so we must select it.
Thus we collapse $rb$ to $r$ and move the marker from~$b$ to $r$.
Since $b$ was marked, we pick up its weight, $m_b$, rather than the
weight of its parent.  The collapsed tree has no more edges, so there
are no more steps to perform.  The value of $F$ on the tree~$T$ is
$m_a^2\cdot m_b^2\cdot m_r$.
Notice that for each non-marked vertex~$v$, the factor $m_v$ appears in~$F$
a total of $\deg(v) - 1$ times.  The factor $m_b$ appears twice.  This 
property holds for any weighted marked tree, as we now show.

\begin{lemma}\label{lemma:weightedmarkedtree}
Let $(T, \mathbf{m}, r, q)$ be a weighted marked tree, and let $F$ be
the function defined above.  Then
\[
  F(T, \mathbf{m}, r, q) = m_q \cdot \prod_{v\in VT} m_v^{\deg(v) - 1}.
\]
\end{lemma}

\begin{proof}
Let $v$ be a vertex of $T$.  There are three cases, depending
on the position of the marked vertex.

\textbf{Case 1.} Neither $v$ nor any of its children is marked.
Each child of $v$ contributes a factor of $m_v$ to the value
of $F$.  Since $v$ is not marked, it contributes the weight
of its parent to the value of $F$ when selected as a leaf.
Hence $v$ contributes a total of~$m_v^{\deg(v) - 1}$ to the
value of $F$.

\textbf{Case 2.} The vertex $v$ is marked.
If $v$ is marked, there is a contribution of $m_v$ for each of its
children as well as a contribution of $m_v$ when it is selected as
a leaf.  Hence $v$ contributes a total of~$m_v^{\deg(v)}$ to
the value of $F$.

\textbf{Case 3.} The vertex $v$ has a marked descendant.
Hereditarily unmarked children of $v$ behave as in Case 1.  
Hence we may assume that $v$ has the marked vertex as its unique child.
When the child of $v$ is selected, it contributes nothing to the
exponent of $m_v$, but then the mark is passed from the child to $v$.
So when $v$ is selected as a leaf, it contributes a weight of $m_v$
to the value of $F$.  Hence $v$ contributes a total of~$m_v^{\deg(v) - 1}$
to the value of $F$.
\end{proof}

The next step is to observe that the function~$F$ is, up to a sign,
the result of computing a determinant by cofactor expansion.  Recall
that the compressed Laplacian $C$ of a cleft graph  factors as 
$C = \bd\varphi$, where $\bd\colon C_1(B)\to C_0(B)$ is the boundary
map and $\varphi\colon C_0(B)\to C_1(B)$ is the map defined in 
Lemma~\ref{lemma:compressedmakessense}.

\begin{lemma}\label{lemma:algtreeisdeterminant}
Let $p\colon X\to T$ a cleft tree with weight vector $\mathbf{m}$.
Select a root $r$ for $T$ and orient all edges away from the root.
Let $M$ be a matrix representing $\varphi$, and let $K$ be a 
matrix representing $\bd$.  Then the determinant of $MK$ is
\[
  \det(MK) = \left(\sum_{q\in VT} m_q\right)\cdot\prod_{v\in VT} m_v^{\deg(v) - 1}.
\]
\end{lemma}

\begin{proof}
By the Binet--Cauchy theorem, the determinant of MK is given by the sum
\[
  \det(MK) = \sum_{q\in VT} \det(M_q)\cdot\det(K_q),
\]
where $M_q$ is obtained from $M$ by striking the column corresponding to $q$,
and $K_q$ is defined similarly.  It follows from Theorem~\ref{Poincare_Chuard}
that $\det(K_q) = \pm 1$.  To evaluate~$\det(M_q)$, select a leaf $v$ of the tree
$T$, let $w$ be the parent of $v$, and let $e$ be the edge from $w$ to $v$.  
If $v\ne q$, then by
cofactor expansion about the $(v, e)$ entry of $M_q$, 
\[
  \det(M_q) = \pm m_w\cdot \det(M'_q),
\]
where $M'_q$ is the submatrix of $M_q$ obtained by striking the column
corresponding to $v$ and the row corresponding to its unique incident edge.
If $v=q$, then by cofactor expansion about the $(w, e)$ entry of $M_q$,
\[
  \det(M_q) = \pm m_v\cdot \det(M'_q).
\]
Up to a sign, this recursive computation of $\det(M_q)$ agrees with
the recursive computation of the function~$F(T, \mathbf{m}, r, q)$.
By computing the determinant of $K_q$ in the same way we see that 
$\det(K_q)$ is equal to the sign of $\det(M_q)$.    Applying
Lemma~\ref{lemma:weightedmarkedtree}, we conclude that
\[
  \det(M_q)\cdot\det(K_q) = m_q \cdot \prod_{v\in VT} m_v^{\deg(v)-1}.
\]
Summing over all $q\in VT$ completes the proof.
\end{proof}

We need the following technical lemma.  

\begin{lemma}[Horn--Johnson~{\cite[Theorem 1.3.20]{Horn_Johnson}}]\label{lemma:eigen}
Suppose $r\le n$.  Let $P$ be an $n\times r$ matrix and $Q$ be an 
$r\times n$ matrix.  Then the eigenvalues of $QP$ are also eigenvalues
of $PQ$, with (at least) the same multiplicity.  All other eigenvalues
of $PQ$ are $0$.
\end{lemma}

Now we use the above results to count the spanning trees of a cleft graph.

\begin{theorem}\label{thm:clefttree}
Let $p\colon X\to T$ be a cleft graph with weight vector $(m_v)_{v\in VT}$,
and define a vector~$(M_v)_{v\in VT}$ by $M_v = \sum_{u\in\delta(v)} m_u$.
If $T$ is a tree, then the tree number of~$X$ is
\[
  \kappa(X) = \prod_{v\in VT} (M_v^{m_v - 1}\cdot m_v^{\deg(v) - 1}).
\]
\end{theorem}

\begin{proof}
The graph $X$ has Laplacian matrix $L$ and compressed 
Laplacian~$C=\bd\varphi$.  Suppose $T$ has $n$ vertices, 
and let $N$ denote the sum
\[
  N = \sum_{v\in VT} m_v,
\]
that is, $N$ is the number of vertices of $X$.
By Theorem~\ref{thm:matrixtree}, 
the tree number of $X$ is
\[
  \kappa(X) = \frac{1}{N}\prod_{i=2}^{N}\lambda_i,
\]
where $\lambda_1\le\dots\le\lambda_{N}$ are the eigenvalues
of~$L$.  The diagonal entries of~$L$ have the form~$M_v$,
each such entry occurring~$m_v$ times.  Hence for 
each~$v\in VT$, the Laplacian of $X$ has eigenvalue~$M_v$
occurring with multiplicity $m_v - 1$.  This leaves $n$
eigenvalues to be determined.  Since the rows and columns of~$L$
sum to zero, one of these eigenvalues is~$\lambda_1 = 0$.

From the fact that $L\trans{p} = \trans{p}\trans{C}$ we conclude
that every eigenvalue of $\trans{C}$ (hence also $C$) is an
eigenvalue of $L$.  
Since $T$ is a tree, it has one more vertex than it has edges, 
so while $C = \bd\varphi$ is an $n\times n$ matrix, its 
companion $\varphi\bd$ is an $n-1\times n-1$ matrix.  Applying
Lemma~\ref{lemma:eigen}, we conclude that the product of the
remaining eigenvalues of~$L$ is $\det(\varphi\bd)$.  But 
it follows from Lemma~\ref{lemma:algtreeisdeterminant} that
\[
  \det(\varphi\bd) = \left(\sum_{q\in VT} m_q\right)\cdot\prod_{v\in VT} m_v^{\deg(v) - 1} = N\cdot\prod_{v\in VT} m_v^{\deg(v) - 1}.
\]
Hence
\[
\kappa(X)
= 
   \prod_{v\in VT} M_v^{m_v - 1}\cdot \prod_{v\in VT} m_v^{\deg(v) - 1},
\]
which is what we wanted to show.
\end{proof}

\section{Concluding remarks}
\label{section_critical_remarks}

The arguments used to study uniformly-cleft graphs and
non-uniformly-cleft trees are different enough that it is
unclear what form a possible common generalization would
take. 
We can compute the critical group explicitly in some 
simple cases, such as a uniformly-cleft path.  However,
the available techniques for working with these structures
do not yet generalize even to the case of uniformly-cleft
trees.  We would like to have a leaf-cutting procedure,
similar to the weighted analogue of the Poincar\'e--Chuard
theorem, which operates on the critical group level.

\vfill
\begin{center}
Copyright \copyright\ Michael Slone 2008
\end{center}

\bibliographystyle{plain}
\bibliography{slone2008-bibliography}
%
\setcounter{secnumdepth}{-1}
\chapter{Vita}

\begin{itemize}
\item
Education:
\begin{itemize}
\item
2008: Ph.D. (expected), University of Kentucky

\item
2003: MA, University of Kentucky

\item
2001: BA, Morehead State University
\end{itemize}

\item
Professional positions held:
\begin{itemize}
\item
2001--2008: Teaching assistant, University of Kentucky

\item
2000: Markup editor, Institute for Regional Analysis and Public Policy

\item
1997--2001: Technical editor, Lexmark-MSU Writing Project
\end{itemize}

\item
Scholastic and professional honors:
\begin{itemize}
\item
Presidential Graduate Fellowship

\item
Edgar Enochs Scholarship in Algebra

\item
Daniel Reedy Quality Fellowship
\end{itemize}
\end{itemize}

\end{document}